\documentclass{amsart}

\usepackage{amsmath} 
\usepackage{amssymb}
\usepackage{bm}

\newcount\skewfactor
\def\mathunderaccent#1#2 {\let\theaccent#1\skewfactor#2
\mathpalette\putaccentunder}
\def\putaccentunder#1#2{\oalign{$#1#2$\crcr\hidewidth
\vbox to.2ex{\hbox{$#1\skew\skewfactor\theaccent{}$}\vss}\hidewidth}}
\def\name{\mathunderaccent\tilde-3 }

\newcommand{\forces}{\Vdash}

\newcommand{\can}{{}^{\omega}2}

\newcommand{\rest}{{\restriction}}
\newcommand{\dom}{{\rm dom}} 

\newcommand{\conc}{{}^\frown\!}
\newcommand{\vtl}{\vartriangleleft} 
\newcommand{\vare}{\varepsilon} 

\newcommand{\diam}{{\rm diam}_{\rho^*}}

\newcommand{\cf}{{\rm cf}}
\newcommand{\rk}{{\rm rk}}
\newcommand{\rksp}{{\rm rk}^{\rm sp}}

\newcommand{\cl}{{\rm cl}}
\newcommand{\stnd}{{\rm stnd}}
\newcommand{\std}{{\rm std}}

\newcommand{\bB}{{\mathbf B}}
\newcommand{\bb}{{\mathbf b}}

\newcommand{\bbC}{{\mathbb C}}

\newcommand{\bD}{{\mathbf D}}
\newcommand{\bF}{{\bf F}}

\newcommand{\bbH}{{\mathbb H}}
\newcommand{\bj}{{\mathbf j}}

\newcommand{\bk}{{\mathbf k}}

\newcommand{\cL}{{\mathcal L}}

\newcommand{\bbM}{{\mathbb M}}

\newcommand{\cP}{{\mathcal P}}
\newcommand{\bbP}{{\mathbb P}}
\newcommand{\bbQ}{{\mathbb Q}}  
\newcommand{\cQ}{{\mathcal Q}}

\newcommand{\bT}{{\mathbf T}}

\newcommand{\cU}{{\mathcal U}}
\newcommand{\bV}{{\mathbf V}}
\newcommand{\cV}{{\mathcal V}}
\newcommand{\cW}{{\mathcal W}}

\newcommand{\bX}{{\mathbf X}}

\newtheorem{theorem}{Theorem}[section] 
\newtheorem{claim}{Claim}[theorem]
\newtheorem{lemma}[theorem]{Lemma} 
\newtheorem{proposition}[theorem]{Proposition} 
\newtheorem{corollary}[theorem]{Corollary} 
\newtheorem{observation}[theorem]{Observation} 

\theoremstyle{definition}
\newtheorem{problem}[theorem]{Problem} 
 
\newtheorem{definition}[theorem]{Definition}
\newtheorem{hypothesis}[theorem]{Assumption}

\theoremstyle{remark}

\newtheorem{conclusion}[theorem]{Conclusion}

\begin{document}

\title{Borel sets without perfectly many overlapping translations, III}

\author{Andrzej Ros{\l}anowski}
\address{Department of Mathematics\\
University of Nebraska at Omaha\\
Omaha, NE 68182-0243, USA}
\email{aroslanowski@unomaha.edu}

\author{Saharon Shelah}
\address{Institute of Mathematics\\
 The Hebrew University of Jerusalem\\
 91904 Jerusalem, Israel\\
 and  Department of Mathematics\\
 Rutgers University\\
 New Brunswick, NJ 08854, USA}
\email{shelah@math.huji.ac.il}
\urladdr{http://shelah.logic.at}

\thanks{Publication 1187 of the second author.\\
The first author thanks the National Science Fundation for supporting his
visit to Rutgers University where this research was carried out, and the 
Rutgers University for their hospitality.\\
Saharon Shelah thanks the Israel Science Foundation for their grant
1838/19.}    

\date{July, 2021}

\begin{abstract}
  We expand the results of Ros{\l}anowski and Shelah \cite{RoSh:1138,
    RoSh:1170} to all perfect Abelian Polish groups $(\bbH,+)$. In
  particular, we show that if $\alpha<\omega_1$ and $4\leq k<\omega$, then
  there is a ccc forcing notion adding a $\Sigma^0_2$ set $B\subseteq \bbH$
  which has $\aleph_\alpha$ many pairwise $k$--overlapping translations but
  not a perfect set of such translations. The technicalities of the forcing
  construction led us to investigations of the question when, in an Abelian
  group, $X-X\subseteq Y-Y$ imply that a translation of $X$ or $-X$ is
  included in $Y$.
\end{abstract}

\maketitle 

\section{Introduction}
For a Polish space $X$ and a set $B\subseteq X\times X$ we say that $B$
contains a  $\mu$--square (perfect square, respectively), if there is a set
$Z$ of cardinality  $\mu$ (a perfect set $Z$, respectively) such that
$Z\times Z\subseteq B$. The problem of Borel sets with large squares 
but no perfect squares was studied and resolved in Shelah \cite{Sh:522}.

Several questions can be phrased in a manner involving $\mu$--squares
and/or perfect squares {\em with some additional structure\/} on them.
For instance, looking at a Polish group  $(\bbH,+)$ we may ask for its Borel
subsets with many, but not too many disjoint translations (or just
translations with small overlaps). This leads to considering the {\em
  spectrum of translation   $k$--disjointness of a set $A\subseteq
  \bbH$,}
\[\std_k(A)=\{(x,y)\in \bbH\times \bbH: |(A+x)\cap (A+y)|\leq k\},\] 
and asking if this set may contain a $\mu$--square  but not a perfect
square. For $k=0$ this is  asking for $\mu$ many pairwise 
disjoint translations of $A$ without a perfect set of such
translations. This direction is related to works of Balcerzak,
Ros{\l}anowski and  Shelah \cite{BRSh:512}, Darji and Keleti
\cite{UdKe03}, Elekes and Stepr\={a}ns \cite{ElSt04},  Zakrzewski
\cite{Zak13} and Elekes and Keleti \cite{ElKe15}.

It is still unresolved if we may repeat the results of \cite{Sh:522} for the
disjointness context, but there is some promising work in progress
\cite{Sh:F1927}.  However a lot of progress has been made in the dual
direction. 

For a set $A\subseteq \bbH$ we consider its {\em spectrum of  translation
  $k$--non-disjointness,}   
\[\stnd_k(A)=\{(x,y)\in \bbH\times \bbH: |(A+x)\cap (A+y)|\geq k\}.\]  
Then a $\mu$--square included in $\stnd_\kappa (A)$ determines a family of
$\mu$ many pairwise $k$--overlapping translations. These were studied
extensively for the context of the Cantor space in Ros{\l}anowski and Rykov
\cite{RoRy18}, and Ros{\l}anowski and Shelah \cite{RoSh:1138,
  RoSh:1170}. Those works fully utilized the algebraic properties of
$(\can,+)$, leaving the general case of Polish groups unresolved.

In the current paper we aim at generalizing their results to perfect Abelian
Polish groups. The main difficulty in this more general case lies in quite
algebraic problem $(\spadesuit)$ given below. Suppose $S\subseteq \bbH$ and
$X\subseteq  \bbH$ is a set of $k$--intersecting translations, i.e.,
\begin{enumerate}
\item[$(\diamondsuit)_X^S$] $|(S+x)\cap (S+y)|\geq k$ for all $x,y\in X$. 
\end{enumerate}
Then for all $c\in \bbH$ the property $(\diamondsuit)_{X+c}^S$ also holds
true. Thus the properties of objects added by our forcing should reflect
some ``translation invariance''. How can we know that a set $Y$ is included
in a translation of $X$? Clearly, if $Y\subseteq X+c$ or $Y\subseteq c-X$, 
then $Y-Y\subseteq X-X$. It would be helpful in our forcing if we knew 
\begin{enumerate}
\item[$(\spadesuit)$] when does $Y-Y\subseteq X-X$ imply that 
  $Y$ is included in a (small) neighborhood of a translation $X+c$ of $X$
  or of a translation $c-X$ of $-X$?    
\end{enumerate}
In the third section we introduce the main algebraic ingredient of our
forcing notion: qifs and quasi independent sets. In forcing, we will use
them in conjunction with differences of elements of the group, but a
relative result for sums also seems interesting, so we present it in Section
4. The third and fourth section might be of interest independently from the
rest of the paper, as they address the question $(\spadesuit)$ giving
interesting (though technical) properties of  perfect Abelian Polish groups
with few elements of rank 2. 

Like in \cite{Sh:522}, the ``no perfect set'' property of the forcing
extension results from the use of a ``splitting rank'' $\rksp$. We remind its 
definition and basic properties in the second section. For the relevant
proofs we refer the reader to \cite{Sh:522,RoSh:1138}.

In the fifth section we prove our main consistency result for groups with few
elements of rank 2. The remaining case when $\bbH$ has many elements of rank
2 is treated in Section 6. We close the paper with summary of our results
and a list of open problems. 

The general case of Polish groups will be investigated in a  subsequent work
\cite{Sh:F1926}.   
\bigskip

\noindent{\bf Notation}:\qquad Our notation is rather standard and
compatible with that of classical textbooks (like Jech \cite{J} or
Bartoszy\'nski and Judah \cite{BaJu95}). However, in forcing we keep the
older convention that {\em a stronger condition is the larger one}.

\begin{enumerate}
\item For a set $u$ we let 
\[u^{\langle 2\rangle}=\{(x,y)\in u\times u:x\neq y\}.\]
\item Ordinal numbers will be denoted be the lower case initial letters of
  the Greek alphabet $\alpha,\beta,\gamma,\delta,\vare,\zeta$. Finite
  ordinals  (non-negative integers) will be denoted by letters
  $i,j,k,\ell,m,n,J,K,L,M,N$ and $\iota$. The Greek letters $\lambda$ and
  $\mu$ will stand for uncountable cardinals.  
\item Finite sequences will be denoted $\sigma,\varsigma$
\item For a forcing notion $\bbP$, all $\bbP$--names for objects in
  the extension via $\bbP$ will be denoted with a tilde below (e.g.,
  $\name{\tau}$, $\name{X}$), and $\name{G}_\bbP$ will stand for the
  canonical $\bbP$--name for the generic filter in $\bbP$.
\item $(\bbH,+,0)$ is an Abelian group (in the main part of the paper it is
 a perfect Polish Abelian group). The elements of $\bbH$ will be  called
 $a,b,c,d$ (with possible indices). For an integer  $\iota$ and $a\in \bbH$,
 we use the notation $\iota a$ to denote the element of $\bbH$ obtained by
 repeated addition of $a$ (or $-a$) $|\iota|$ many times in  the usual way.     
\item For sets $A,B\subseteq \bbH$ we will write $-A=\{-a:a\in 
  \bbH\}$, 
\[A+B=\{a+b:a\in A\ \wedge\ b\in B\}\quad \mbox{ and }\quad A-B= \{a-b:a\in 
  A\ \wedge\ b\in B\}.\] 
\end{enumerate}

\section{Splitting rank $\rksp$} 
Let us recall a rank used in previous papers which will be central  
for the results here too. We quote some definitions and theorems 
from \cite[Section 2]{RoSh:1138}, however they were first given in
\cite[Section 1]{Sh:522}.  
 
Let $\lambda$ be a cardinal and $\bbM$ be a model with the universe  
$\lambda$ and a countable vocabulary $\tau$.  

\begin{definition}
\label{defofrank}
\begin{enumerate}
\item By induction on ordinals $\delta$, for finite non-empty sets  
  $w\subseteq\lambda$ we define when $\rk(w,\bbM)\geq \delta$. Let  
  $w=\{\alpha_0,\ldots,\alpha_n\} \subseteq\lambda$, $|w|=n+1$.  
 \begin{enumerate}
\item[(a)] $\rk(w)\geq 0$ if and only if for every quantifier free formula 
    $\varphi=\varphi(x_0,\ldots,x_n)\in \cL(\tau)$ and each $k\leq n$, if  
$\bbM\models \varphi[\alpha_0,\ldots,\alpha_k,\ldots,\alpha_n]$ then  the
set
\[\big\{\alpha\in \lambda:\bbM\models \varphi[\alpha_0,\ldots,\alpha_{k-1},   
  \alpha,\alpha_{k+1}, \ldots,\alpha_n]\big\}\]
is  uncountable;  
\item[(b)] if $\delta$ is limit, then $\rk(w,\bbM)\geq\delta$ if and only if 
  $\rk(w,\bbM)\geq\gamma$ for all $\gamma<\delta$;  
\item[(c)] $\rk(w,\bbM)\geq\delta+1$ if and only if for every quantifier  
  free  formula $\varphi=\varphi(x_0,\ldots,x_n)\in \cL(\tau)$ and each
  $k\leq n$, if $\bbM\models \varphi[\alpha_0,\ldots,\alpha_k,\ldots,
  \alpha_n]$ then there is   $\alpha^*\in\lambda\setminus w$ such that  
\[\rk(w\cup\{\alpha^*\},\bbM)\geq \delta\quad\mbox{ and }\quad \bbM\models   
  \varphi[\alpha_0,\ldots,\alpha_{k-1},\alpha^*,\alpha_{k+1},\ldots,\alpha_n].\] 
 \end{enumerate}
\end{enumerate}
\end{definition}
 
By a straightforward induction on $\delta$ one easily shows that if 
$\emptyset\neq v\subseteq w$ then 
\[\rk(w,\bbM)\geq\delta\geq\gamma \implies \rk(v,\bbM)\geq \gamma.\] 
Hence we may define the rank function on finite non-empty subsets of  
$\lambda$. 

\begin{definition} 
The rank $\rk(w,\bbM)$ of a finite non-empty set $w\subseteq\lambda$ 
is defined as:  
 \begin{itemize}
\item $\rk(w,\bbM)=-1$ if $\neg (\rk(w,\bbM)\geq 0)$,  
\item $\rk(w,\bbM)=\infty$ if $\rk(w,\bbM)\geq \delta$ for all ordinals  
  $\delta$,  
\item for an ordinal $\delta$: $\rk(w,\bbM)=\delta$ if $\rk(w,\bbM)\geq  
  \delta$ but $\neg(\rk(w,\bbM)\geq\delta+1)$.  
\end{itemize}
\end{definition}

\begin{definition} 
For an ordinal $\vare$ and a cardinal $\lambda$ let ${\rm  
  NPr}^\vare(\lambda)$  be the following statement:  
\begin{quotation}
``there is a model  $\bbM^*$ with the universe $\lambda$ and a countable  
vocabulary $\tau^*$ such  that  $1+\rk(w,\bbM^*)\leq\vare$ for all $w\in
[\lambda]^{<\omega}\setminus\{\emptyset\}$.''      
\end{quotation}
Let ${\rm Pr}^\vare(\lambda)$ be the  negation of ${\rm NPr}^\vare(\lambda)$.
\end{definition}
Note that ${\rm NPr}_\vare$ of \cite[Definition 2.4]{RoSh:1138}
differs from our ${\rm NPr}^\vare$: ``$\sup\{\rk(w,\bbM^*):
\emptyset\neq w\in [\lambda]^{<\omega} \}<\vare$  '' there is replaced
by ``$1+\rk(w,\bbM^*)\leq\vare$'' here. However, the proofs for
\cite[Propositions 2.6, 2.7]{RoSh:1138} show the following results. 

\begin{proposition}
\label{cl1.7-522}
\begin{enumerate}
\item ${\rm NPr}^1(\omega_1)$.
\item If ${\rm NPr}^\vare(\lambda)$, then ${\rm NPr}^{\vare+1}(\lambda^+)$.  
\item If ${\rm NPr}^\vare(\mu)$ for $\mu<\lambda$ and
  $\cf(\lambda)=\omega$, then ${\rm NPr}^\vare(\lambda)$. 
\item If $\alpha<\omega_1$, then ${\rm NPr}^{\alpha}(\aleph_\alpha)$ but  
${\rm Pr}^\alpha(\beth_{\omega_1})$ holds. 
\end{enumerate}
\end{proposition}

\begin{definition}
Let $\tau^\otimes=\{R_{n,j}:n,j< \omega\}$ be a fixed relational
vocabulary where $R_{n,j}$ is an $n$--ary relational symbol (for
$n,j<\omega$).  
\end{definition}

\begin{definition}
\label{hypo2}
Assume that $\vare<\omega_1$ and $\lambda$ is an uncountable cardinal such
that  ${\rm NPr}^\vare(\lambda)$. By this assumption, we may fix a model
$\bbM(\vare,\lambda)= \bbM=(\lambda, \{R^\bbM_{n,j}\}_{n,j<\omega}) $ in the
vocabulary $\tau^\otimes$ with the  universe $\lambda$ such that:  
\begin{enumerate}
\item[$(\circledast)_{\rm a}$] for every $n$ and a quantifier free formula 
  $\varphi(x_0,\ldots,x_{n-1})\in \cL(\tau^\otimes)$ there is $j<\omega$ such
  that for all $\alpha_0,\ldots, \alpha_{n-1}\in \lambda$, 
\[\bbM\models\varphi[\alpha_0,\ldots,\alpha_{n-1}]\Leftrightarrow
  R_{n,j}[\alpha_0, \ldots, \alpha_{n-1}],\]
\item[$(\circledast)_{\rm b}$] the rank of every singleton is at least 0,  
\item[$(\circledast)_{\rm c}$] $1+\rk(v,\bbM)\leq\vare$ for every $v\in
  [\lambda]^{<\omega} \setminus \{\emptyset\}$,
\item[$(\circledast)_{\rm d}$] $\bbM\models R_{2,0}[\alpha_0,\alpha_1]$
  \quad if and only if \quad $\alpha_0<\alpha_1<\lambda$.
\end{enumerate}
For a nonempty finite set $v\subseteq\lambda$ let
$\rk^{\rm sp}(v)=\rk(v,\bbM)$, and we fix witnesses $\bj(v)<\omega$ and
$\bk(v)<|v|$ for the rank of $v$, so that the following demands
$(\circledast)_{\rm e}$--$(\circledast)_{\rm g}$ are satisfied. If
$\{\alpha_0,\ldots, \alpha_k, \ldots \alpha_{n-1}\}$ is the increasing
enumeration of $v$ and $k=\bk(v)$ and $j= \bj(v)$, then  
\begin{enumerate}
\item[$(\circledast)_{\rm e}$] if $\rk^{\rm sp}(v)\geq 0$, then
  $\bbM\models R_{n,j}[\alpha_0,\ldots, \alpha_k,\ldots, \alpha_{n-1}]$
  but there is no $\alpha\in \lambda\setminus v$  such that  
\[\rk^{\rm sp}(v\cup\{\alpha\})\geq \rk^{\rm sp}(v)\ \mbox{ and }\ 
  \bbM\models R_{n,j} [\alpha_0,\ldots, \alpha_{k-1}, \alpha,
  \alpha_{k+1}, \ldots,\alpha_{n-1}],\]  
\item[$(\circledast)_{\rm f}$] if $\rk^{\rm sp}(v)=-1$, then
  $\bbM\models R_{n,j} [\alpha_0,\ldots,\alpha_k,\ldots, \alpha_{n-1}]$
  but the set
  \[\big\{\alpha\in\lambda:\bbM\models R_{n,j}[\alpha_0,\ldots,
    \alpha_{k-1},\alpha, \alpha_{k+1}, \ldots, \alpha_{n-1}]\big\}\] 
is countable,  
\item[$(\circledast)_{\rm g}$] for every
  $\beta_0,\ldots,\beta_{n-1}<\lambda$,\quad if $\bbM\models
  R_{n,j}[\beta_0,\ldots,\beta_{n-1}]$  then $\beta_0<\ldots <\beta_{n-1}$. 
\end{enumerate}
The choices above define functions $\bj:[\lambda]^{<\omega}\setminus \{
\emptyset\} \longrightarrow \omega$, $\bk:[\lambda]^{<\omega}\setminus  
\{\emptyset\} \longrightarrow \omega$, and $\rksp:[\lambda]^{<\omega}
\setminus \{\emptyset\} \longrightarrow \{-1\}\cup (\vare+1)$. 
\end{definition}

\section{QIFs and differences}

\begin{definition}
  \label{quasidef}
Let $(\bbH,+,0)$ be an Abelian group and $\bB\subseteq \bbH$.
   \begin{enumerate}
\item A {\em (2,n)--combination from $\bB$\/} is
  any  sum of the form 
\[\iota_0b_0+\iota_1b_1+\iota_2b_2+\ldots+\iota_{n-1}b_{n-1}\] 
where $b_0,b_1,\ldots,b_{n-1}\in \bB$ are pairwise distinct and
$\iota_0,\iota_1,\iota_2,\ldots,  \iota_{n-1}\in \{-2,-1,0,1,2\}$. The
$(2,n)$--combination is said to be {\em nontrivial\/} when not all
$\iota_0,\ldots,\iota_{n-1}$ are equal $0$.   
\item   We say that the set $\bB$ is {\em quasi independent in $\bbH$\/} if
  $|\bB|\geq 8$ and no nontrivial $(2,8)$--combination from $\bB$ equals to
  $0$.   
\item We say that a family $\cV$ of non-empty subsets of $\bbH$ is an {\em 
    $n$--good qif\/}{}\footnote{short for quasi independent family} if
$|\cV|\geq n$, the sets in $\cV$ are pairwise disjoint and for distinct
$V_0,\ldots,V_{n-1}\in \cV$, for each choice of $b_i,b_i'\in V_i$  (for
$i<n$) and every $\iota_0,\iota_0',\ldots, \iota_{n-1},\iota_{n-1}'\in
\{-1,0,1\}$ such that $\sum_{i=0}^{n-1} (\iota_i+\iota_i')^2\neq 0$ we have     
  \[\iota_0b_0+\iota_0'b_0'+\iota_1b_1+\iota_1'b_1'+\ldots+
    \iota_{n-1}b_{n-1}+\iota_{n-1}'b_{n-1}'\neq 0.\]
An expression as on the left hand side above will be called a {\em
  nontrivial $(2,\cV,n)$--combination} (or a {\em nontrivial
  $(2,n)$--combination from $\cV$}).
\item Let $\cV,\cW\subseteq \cP(\bbH)\setminus \{\emptyset\}$. We will say
  that $\cW$ is {\em immersed in\/} $\cV$ if there is a bijection
  $\pi:\cW\stackrel{1-1}{\longrightarrow} \cV$ such that 
  \begin{itemize}
\item $W\subseteq \pi(W)$ for all $W\in\cW$, and 
\item if $W_0,W_1\in\cW$, and $a,a'\in W_0$, $b\in W_1$, then
  $(a-a')+b\in \pi(W_1)$.  
  \end{itemize}
\end{enumerate}
\end{definition}

\begin{observation}
\label{obs3.2}
  \begin{enumerate}
\item If $\bB$ is quasi independent then all elements of $\bB$ have 
order at least 3 and $\big\{ \{b\}:b\in\bB\big\}$ is an $8$--good qif.  
\item If $\cV$ is an $8$--good qif and $b_V\in V$ (for $V\in \cV$) then
  $\{b_V: V\in \cV\}$ is quasi independent. 
\item Assume $\bbH$ is an Abelian Polish group. Suppose also that, for
  $i<N<\omega$, $V_i\subseteq\bbH$ are disjoint open sets and $b_i\in
  V_i$. Then there are open sets $W_i$ such that $b_i\in W_i\subseteq V_i$
  for $i<N$, and $\{W_i: i<N\}$ is immersed in $\{V_i:i<N\}$.   
  \end{enumerate}
\end{observation}

\begin{proposition}
  \label{existence}
  Assume that
  \begin{enumerate}
\item[(i)] $(\bbH,+,0)$ is a perfect Abelian Polish group,
\item[(ii)] the set of elements of $\bbH$ of order larger than 2 is dense
  in $\bbH$,
\item[(iii)] $U_0,\ldots,U_{n-1}$ are nonempty open subsets of $\bbH$. 
\end{enumerate}
Then there are disjoint open sets $V_i\subseteq U_i$ (for $i<n$) such that
$\{V_i:i<n\}$ is an  $n$--good qif.   
\end{proposition}

\begin{proof}
 Let $H_2$ consists of all elements of $\bbH$ of order $\leq 2$. Then $H_2$
 is a closed subgroup of $\bbH$ and, by the assumption (ii), it has empty
 interior. Consequently, for each $a\in \bbH$ and $i< n$ the set
 $(a+H_2)\cap U_i$ is meager. Therefore, for each $i< n$,
 \begin{enumerate}
 \item[$(\otimes)_i$] the set $\big\{a+H_2: a\in \bbH$ and $(a+H_2)\cap
   U_i\neq \emptyset\ \big\}$  is infinite.
 \end{enumerate}
Let $m_0=10$ and $m_{i+1}=10^{i+1}\cdot\prod\limits_{j\leq i} m_j+10$
(for $i<n$). For each $i< n$ choose a set $A_i\subseteq
U_i\setminus H_2$ such that
\begin{enumerate}
\item[$(\oplus)_0$] $|A_i|=m_i$ and 
\item[$(\oplus)_1$] if $a,b\in A_i$ and $a\neq b$, then $2a\neq 2b$.
\end{enumerate}
(The choice is possible by $(\otimes)_i$ for each $i< n$.) For $0<i< n$ let 
\[\begin{array}{ll}
  X_i=\big\{\iota_0 a_0+\ldots+\iota_{i-1} a_{i-1}: &a_0\in A_0,\ldots,
      a_{i-1}\in A_{i-1}\ \wedge\\
     &\iota_0,\ldots,\iota_{i-1}\in \{-2,-1,0,1,2\} \big\}.
\end{array}\]
By the choice of $m_j$'s we know that $2\cdot |X_i|<m_i=|A_i|$, so we
may choose $b_i^*\in A_i$ such that $2b_i^*,b_i^*\notin X_i$. Let $b_0^*\in  
A_0$ be arbitrary. One easily verifies that every nontrivial
$(2,n)$--combination from $\{b^*_i:i< n\}$ is not zero, so for each
$\iota_0,\iota_0',\ldots, \iota_{n-1},\iota_{n-1}'\in \{-1,0,1\}$ such that
$\sum_{i=0}^{n-1} (\iota_i+\iota_i')^2\neq 0$ we have 
 \[\iota_0b_0^*+\iota_0'b_0^*+\iota_1b_1^*+\iota_1'b_1^*+\ldots+
    \iota_{n-1}b_{n-1}+\iota_{n-1}'b_{n-1}^*\neq 0.\]
For each such combination we may choose disjoint open sets
$V^i_{\iota_0,\iota_0',\ldots,\iota_{n-1},\iota_{n-1}'}$ such that $b^*_i\in
V^i_{\iota_0,\iota_0',\ldots,\iota_{n-1},\iota_{n-1}'}\subseteq U_i$ and for
every $b_i,b_i'\in V^i_{\iota_0,\iota_0',\ldots,\iota_{n-1},\iota_{n-1}'}$,
$i<n$, we have  
  \[\iota_0b_0+\iota_0'b_0'+\iota_1b_1+\iota_1'b_1'+\ldots+
    \iota_{n-1}b_{n-1}+\iota_{n-1}'b_{n-1}'\neq 0.\]
Now, for $i<n$ we set
\[\begin{array}{r}
V_i=\bigcap \big\{V^i_{\iota_0,\iota_0',\ldots,\iota_{n-1},\iota_{n-1}'} : 
  \iota_0,\iota_0',\ldots,\iota_{n-1},\iota_{n-1}' \in \{-1,0,1\}\
    \wedge\ \qquad \\
  (\iota_0-\iota_0')^2+\ldots +(\iota_{n-1}-\iota_{n-1}')^2>0\big\}.
  \end{array}\]      
It is clear that the sets $V_i$ (for $i< n$) are as required.
\end{proof}

\begin{lemma}
\label{fromKupMin}
Suppose that $(\bbH,+,0)$ is an Abelian group and $\rho$ is a translation
invariant metric on it. Let $\cW\subseteq \cP(\bbH)$ be a finite 8--good
qif. Assume that  
\begin{enumerate}
\item[(a)] $\cW$ is immersed in $\cV$, $\cV\subseteq \cP(\bbH)$, 
\item[(b)] $ A'\subseteq A\subseteq \bbH$, $| A'|=8$,
\item[(c)] $A-A\subseteq \bigcup\big\{W-W':W,W'\in\cW\big\}$, 
\item[(d)]  if $a,b\in A$, $a\neq b$, then $\rho(a,b)>{\rm diam}_\rho(W)$
  ($={\rm diam}_\rho(-W)$) for all $W\in \cW$. 
\end{enumerate}
(1)\quad If $c\in \bbH$ is such that $ A'+c\subseteq \bigcup \cW$,
then  also $ A+c\subseteq \bigcup \cV$.\\
(2)\quad If $c\in \bbH$ is such that $c- A'\subseteq \bigcup \cW$,
then also  $c- A\subseteq \bigcup \cV$.
\end{lemma}

\begin{proof}
(1)\quad Suppose that $\cW,\cV$, $ A'\subseteq A\subseteq \bbH$ satisfy the 
assumptions of the Lemma and $c\in \bbH$ is such that $ A'+c\subseteq
\bigcup\cW$.  

Assume $a\in A\setminus A'$ and let us argue that $a+c\in \bigcup\cV$.  

Let $\langle a_i:i<8\rangle$ list the elements of $ A'$. For $i<8$ let
$b_i=a_i+c\in W_i\in \cW$ and note that all $W_i$'s are pairwise distinct 
(by assumption (d); remember $\rho$ is translation invariant). It follows
from assumption (c) that we may choose $b_i'\in W_i'\in \cW$ and $b_i''\in
W_i''\in \cW$ such that $a-a_i=b_i'-b_i''$. Then, for each $i<8$, we have  
\[a+c=a+(b_i-a_i)=(b_i'-b_i''+a_i)+(b_i-a_i)= b_i'-b_i''+b_i.\]

\begin{claim}
  \label{cl4}
  There are distinct $i^*,j^*<8$ such that
  \begin{enumerate}
  \item[$(\heartsuit)_{i^*,j^*}$] \qquad $W_{i^*}\notin \{W_{j^*}' ,
    W_{j^*}''\}$  and $W_{j^*} \notin \{W_{i^*}',W_{i^*}''\}$.  
  \end{enumerate}
\end{claim}

\begin{proof}[Proof of the Claim]
If for some $i_0<8$ we have $|\{j<8:W_{i_0}=W_j''\ \wedge\ j\neq i_0\}|\geq 
3$, then choose $j_0<j_1<j_2<8$ distinct from $i_0$ and such that 
$W_{j_0}''= W_{j_1}''=W_{j_2}''=W_{i_0}$. Since all $W_i$'s are distinct, we
may pick $i^*<8$ such that $i^*\notin \{i_0,j_0,j_1,j_2\}$ and
$W_{i^*}\notin \{W_{j_0}',W_{j_1}', W_{j_2}'\}$. Next let $j^*\in\{j_0,j_1,
j_2\}$ be such that $W_{j^*}\notin\{W_{i^*}',W_{i^*}''\}$. Then also
$W_{i^*}\neq W_{i_0}=W_{j^*}''$ and clearly $(\heartsuit)_{i^*,j^*}$ holds
true. 

If for some $i_0<8$ we have $|\{j<8:W_{i_0}=W_j'\ \wedge\ j\neq i_0\}|\geq 
3$, then by the same argument (just interchanging $W''_j$'s and $W'_j$'s) we
find $i^*,j^*$ so that $(\heartsuit)_{i^*,j^*}$ holds true. 

So now suppose that both $|\{j<7:W_7=W_j'\}|\leq 2$ and
$|\{j<7:W_7=W_j''\}|\leq 2$. Then there are $j_0<j_1<j_2<7$ such that
$W_7\notin \{W_{j_0}', W_{j_0}'', W_{j_1}', W_{j_1}'', W_{j_2}',
W_{j_2}''\}$. Take $j^*\in \{j_0,j_1,j_2\}$ such that $W_{j^*}\notin
\{W_7',W_7''\}$ and note that then $(\heartsuit)_{7,j^*}$ holds true.
\end{proof}

Let distinct $i^*,j^*<8$ be such that $(\heartsuit)_{i^*,j^*}$ holds.

It follows from assumption (d) that $W_{i^*}'\neq W_{i^*}''$ and
$W_{j^*}'\neq W_{j^*}''$ (remember $a_{i^*}\neq a\neq a_{j^*}$). Now, if
$W_{i^*}= W_{i^*}''$, then
\[a+c=b_{i^*}'+(b_{i^*}-b_{i^*}'')\in
  \big(W_{i^*}'+(W_{i^*}''-W_{i^*}'')\big) \subseteq V_{i^*}'\] 
where $W_{i^*}'\subseteq V_{i^*}'\in\cV$ (so we are done).  Similarly, if 
$W_{j^*}=W_{j^*}''$.   

So suppose towards contradiction that both $W_{i^*}\neq W_{i^*}''$ and
$W_{j^*}\neq W_{j^*}''$. Now,
\[b_{i^*}'-b_{i^*}''+b_{i^*}=a+c=b_{j^*}'-b_{j^*}''+b_{j^*},\]
so
\begin{enumerate}
\item[$(\otimes)$] \qquad $(b_{i^*}+b_{i^*}'+b_{j^*}'')-(b_{j^*}+b_{j^*}'
  +b_{i^*}'')=0$.  
\end{enumerate}
Considering known inequalities among $W_{i^*}, W_{i^*}', W_{i^*}'', W_{j^*},
W_{j^*}', W_{j^*}''$,  we notice that no equality between them may involve
more than two sets. Also $W_{i^*}\notin \{W_{j^*}, W_{j^*}', W_{i^*}''\}$,
so the expression on the left hand side of $(\otimes)$ can be  
written as a nontrivial $(2,\cW,8)$--combination, contradicting the
assumption that $\cW$ is an 8--good qif.
\medskip

\noindent (2)\quad Follows from the first part applied to $-A$ and $-A'$. 
\end{proof}

\begin{theorem}
\label{gettranslMin}
 Suppose that $(\bbH,+,0)$ is an Abelian group and $\rho$ is a translation
 invariant metric on it. Assume also that 
\begin{enumerate}
\item[(a)] $\cW, \cV, \cQ\subseteq \cP(\bbH)$ are finite 8--good qifs, and
  $\cW$ is immersed in $\cV$ and  $\cV$ is immersed in $\cQ$, 
\item[(b)] $m\longrightarrow \big(10\big)^4_{2^{144}}$ (the Erd\H{o}s--Rado
  arrow notation, see \cite{ErRa56}),
\item[(c)] $ A\subseteq \bbH$, $| A|\geq m$ and 
\item[(d)] $A-A\subseteq \bigcup\big\{W-W':W,W'\in\cW\big\}$, and 
\item[(e)] if $a,b\in A$, $a\neq b$, then $\rho(a,b)>{\rm diam}_\rho(Q)$
($={\rm diam}_\rho(-Q)$) for all $Q\in \cQ$. 
\end{enumerate}
Then exactly one of (A), (B) below holds true:
 \begin{enumerate}
 \item[(A)]   There is a $c\in \bbH$ such that $ A+c\subseteq \bigcup\cQ$.
 \item[(B)]   There is a $c\in \bbH$ such that $c- A\subseteq \bigcup\cQ$.    
 \end{enumerate}
\end{theorem}

\begin{proof}
Let $\langle a_i:i<m\rangle$ be a sequence of pairwise distinct elements of 
$ A$. Since $ A- A\subseteq \bigcup\big\{W-W':W,W'\in\cW\big\}$, we may
choose functions $\bb_0,\bb_1:m\times m\longrightarrow \bigcup\cW$ and
$\bar{W}_0,\bar{W}_1:m\times m\longrightarrow \cW$ such that  for all
$i,j<m$ 
\[a_i-a_j=\bb_0(i,j)-\bb_1(i,j),\ \bb_0(i,j)\in \bar{W}_0(i,j),\
  \bb_1(i,j)\in \bar{W}_1(i,j),\]
and $\bb_0(i,j)=\bb_1(j,i)$, and $\bb_1(i,j)=\bb_0(j,i)$. Let $\langle
\varphi_\ell(i_0,i_1,i_2,i_3):\ell<144\rangle$ list all formulas of the form  
\[\bar{W}_j(i_x,i_y)=\bar{W}_{j'}(i_{x'},i_{y'})\]
for $j,j'<2$ and $x,y,x',y'<4$, $x<y$, $x'<y'$.

Let $\mu:\big[ m\big]^{\textstyle 4}\longrightarrow {}^{144}2$ be a
coloring of quadruples from $m$ such that if $i_0<i_1<i_2<i_3<m$, then
\[\mu\big(\{i_0,i_1,i_2,i_3\}\big)\big(\ell)=1\quad \mbox{ if and only if
  }\quad \varphi_\ell(i_0,i_1,i_2,i_3)\mbox{ holds true.}\]
Since $m\longrightarrow \big(10\big)^4_{2^{144}}$, we may choose $u\in
[m]^{\textstyle 10}$ homogeneous for $\mu$. Without loss of generality,
$u=\{0,1,2,3,4,5,6,7,8,9\}$.  

\begin{claim}
  \label{cl3}
  Let $i,j,k<10$ be pairwise distinct. Then
  \begin{enumerate}
  \item $\bar{W}_0(i,j)\neq \bar{W}_1(i,j)$ and 
  \item $\bb_0(i,k)-\bb_1(i,k)=\bb_0(i,j)-\bb_1(i,j) +\bb_0(j,k)-
    \bb_1(j,k)$ and hence
\[\Big(\bar{W}_0(i,k)-\bar{W}_1(i,k)\Big)\cap\Big(\big(
  \bar{W}_0(i,j)-\bar{W}_1(i,j)\big) +\big(\bar{W}_0(j,k)-
  \bar{W}_1(j,k)\big) \Big)\neq \emptyset.\]
  \end{enumerate}
\end{claim}

\begin{proof}[Proof of the Claim]
(1)\quad Follows from assumption (e) of the Theorem (remember every set
from $\cW$ is a subset of a member of $\cQ$).
\smallskip

\noindent (2)\quad This  follows by the equality
$(a_i-a_j)+(a_j-a_k)=a_i-a_k$ and the choice of $\bb_0(i,j), \bar{W}_0(i,j), 
\bb_1(i,j),\bar{W}_1(i,j)$.
\end{proof}
\medskip

\begin{claim}
  \label{cl1}
 If $\{\bar{W}_0(i,j):i<j<10\}\cap \{\bar{W}_1(i,j):i<j<10\}\neq \emptyset$, then
 either $\bar{W}_0(0,1)=\bar{W}_1(1,2)$, or $\bar{W}_1(0,1)=\bar{W}_0(1,2)$.
\end{claim}

\begin{proof}[Proof of the Claim]
Suppose $i_0<j_0<10$ and $i_1<j_1<10$ are such that $\bar{W}_0(i_0,j_0) =
\bar{W}_1(i_1,j_1)$. We shall consider all possible orders of $i_0,j_0,i_1,j_1$
and use the homogeneity to conclude one of the clauses in the assertion.
\medskip

\noindent (a)\quad If $i_0<j_0<i_1<j_1$, then (by the homogeneity)
$\bar{W}_0(0,1)=\bar{W}_1(2,3) =\bar{W}_1(4,5) =\bar{W}_0(2,3)$, so 
$\bar{W}_0(2,3)=\bar{W}_1(2,3)$, contradicting Claim \ref{cl3}(1). 
\medskip

\noindent (b)\quad If $i_0<j_0=i_1<j_1$ then also
$\bar{W}_0(0,1)=\bar{W}_1(1,2)$ (giving the conclusion of Claim \ref{cl1}).
\medskip

\noindent (c)\quad If $i_0<i_1<j_0<j_1$, then
$\bar{W}_0(1,4)=\bar{W}_1(2,5)= \bar{W}_0(0,3)=\bar{W}_1(1,4)$,
contradicting Claim \ref{cl3}(1). 
\medskip

\noindent (d)\quad If $i_0<i_1<j_0=j_1$, then $\bar{W}_0(0,3)=
\bar{W}_1(1,3)= \bar{W}_1(2,3)= \bar{W}_0(1,3)$, contradicting Claim
\ref{cl3}(1). 
\medskip

\noindent (e)\quad If $i_0<i_1<j_1<j_0$, then $\bar{W}_0(1,4)=\bar{W}_1(2,3)
=\bar{W}_0(0,5) =\bar{W}_1(1,4)$, contradicting Claim \ref{cl3}(1).
\medskip

\noindent (f)\quad If $i_0=i_1<j_0<j_1$, then $\bar{W}_0(0,1)=\bar{W}_1(0,2)
= \bar{W}_1(0,3)=\bar{W}_0(0,2)$, contradicting Claim \ref{cl3}(1).  
\medskip

\noindent (g)\quad The configuration $i_0=i_1<j_0=j_1$ contradicts Claim
\ref{cl3}(1).   
\medskip

\noindent (h)\quad If $i_0=i_1<j_1<j_0$, then $\bar{W}_0(0,2)=\bar{W}_1(0,1) =
\bar{W}_0(0,3)= \bar{W}_1(0,2)$, contradicting Claim \ref{cl3}(1).  
\medskip

\noindent (i)\quad The configuration $i_1<i_0<j_0<j_1$  is not possible
similarly to (e) (just interchange $\bar{W}_0$ and $\bar{W}_1$).  
\medskip

\noindent (j)\quad The configuration $i_1<i_0<j_0=j_1$ is not possible
similarly to (d).
\medskip

\noindent (k)\quad The configuration $i_1<i_0<j_1<j_0$  is not possible
similarly to (c). 
\medskip

\noindent (l)\quad If $i_1<i_0=j_1<j_0$, then
$\bar{W}_1(0,1)=\bar{W}_0(1,2)$  (giving the conclusion of Claim \ref{cl1}). 
\medskip

\noindent (m)\quad The configuration $i_1<j_1<i_0<j_0$, is not possible 
similarly to (a).
\end{proof}

\medskip

Now, we will consider three cases, showing that the first one is not
possible. In the second case we will find $c\in \bbH$ such that 
$\{c-a_i:i<8\}\subseteq \bigcup\cV$. Then by Lemma \ref{fromKupMin} we will
also  have $c- A\subseteq \bigcup\cQ$. Finally in the last case we will find
$c\in \bbH$ such that $\{a_i+c:i<8\}\subseteq \bigcup\cV$, so by Lemma
\ref{fromKupMin} we will also have $ A+c\subseteq \bigcup\cQ$.
\medskip

For $\ell<2$ and $i<j<10$ let $\bar{V}_\ell(i,j)\in\cV$ be the unique
set such that $\bar{W}_\ell(i,j)\subseteq\bar{V}_\ell(i,j)$. Also, let
$\cV_\ell=\{\bar{V}_\ell(i,j):i<j<10\}$. 
\medskip

\noindent {\sc Case 1}:\quad $\{\bar{W}_0(i,j):i<j<10\}\cap
\{\bar{W}_1(i,j):i<j<10\}=\emptyset$.\\
By Claim \ref{cl3}(2) we have
\[\bb_0(0,2)-\bb_1(0,2)=\bb_0(0,1)-\bb_1(0,1) +\bb_0(1,2)-\bb_1(1,2)\] 
or
\[\big(\bb_0(0,1) -\bb_0(0,2) +\bb_0(1,2)\big)+\big(\bb_1(0,2)-\bb_1(0,1)-
\bb_1(1,2)\big)=0.\]
If $\big|\{\bar{W}_0(0,1), \bar{W}_0(0,2), \bar{W}_0(1,2)\}\big|\leq 2$,
then 
\begin{enumerate}
\item[either] $\bar{W}_0(0,1) =\bar{W}_0(1,2)$ and by the homogeneity
  $\bar{W}_0(0,1)=\bar{W}_0(i,j)$ for all $i<j<9$, so $\bb_0(0,1)+\bb_0(1,2) 
  -\bb_0(0,2)\in \bar{W}_0(0,1)+(\bar{W}_0(0,1)-\bar{W}_0(0,1)) \subseteq  
 \bar{V}_0(0,1)$, 
\item[or] $\bar{W}_0(0,2)=\bar{W}_0(0,1)$ and then
  $\bb_0(0,1)-\bb_0(0,2)+\bb_0(1,2)\in (\bar{W}_0(0,1)-
  \bar{W}_0(0,1))+ \bar{W}_0(1,2) \subseteq \bar{V}_0(1,2)$, 
\item[or] $\bar{W}_0(0,2)=\bar{W}_0(1,2)$ and then $\bb_0(1,2)-
  \bb_0(0,2)+ \bb_0(0,1)\in (\bar{W}_0(0,2)-\bar{W}_0(0,2))+
  \bar{W}_0(0,1) \subseteq \bar{V}_0(0,1)$.  
\end{enumerate}
Therefore, if $\big|\{\bar{W}_0(0,1), \bar{W}_0(0,2), \bar{W}_0(1,2)\}\big|\leq
2$ then $\bb_0(0,1) -\bb_0(0,2) +\bb_0(1,2)\in \bigcup\cV_0$. If
elements of $\{\bar{W}_0(0,1), \bar{W}_0(0,2), 
\bar{W}_0(1,2)\}$ are all distinct, then they are respectively
included in disjoint sets $\bar{V}_0(0,1), \bar{V}_0(0,2),
\bar{V}_0(1,2)$. Hence we may conclude that in any case $\bb_0(0,1)
-\bb_0(0,2) +\bb_0(1,2)$ equals to a nontrivial $(2,\cV_0,3)$--
combination. 

Similarly, if  $\big|\{\bar{W}_1(0,1), \bar{W}_1(0,2),
\bar{W}_1(1,2)\}\big|\leq 2$, then 
\begin{enumerate}
\item[either] $\bar{W}_1(0,1) =\bar{W}_1(1,2)$ and then 
$-\big((\bb_1(0,1)-\bb_1(0,2))+\bb_1(1,2)\big)\in -\bar{V}_1(1,2)$, 
\item[or] $\bar{W}_1(0,1)=\bar{W}_1(0,2)$ and then 
$-\big((\bb_1(0,1)-\bb_1(0,2))+\bb_1(1,2)\big)\in -\bar{V}_1(1,2)$, 
\item[or] $\bar{W}_1(0,2)=\bar{W}_1(1,2)$ and then 
$-\big((\bb_1(1,2) -\bb_1(0,2)) +\bb_1(0,1))\big)\in -\bar{V}_1(0,1)$.
\end{enumerate}
Therefore easily in any case $\bb_1(0,2)-\bb_1(0,1)- \bb_1(1,2)$
equals to a nontrivial $(2,\cV_1,3)$--combination. 
\medskip

Now, in the current case we have $\cV_0\cap \cV_1=\emptyset$, so we
may conclude that $0=\big(\bb_0(0,1) -\bb_0(0,2)
+\bb_0(1,2)\big)+\big(\bb_1(0,2)-\bb_1(0,1)-\bb_1(1,2)\big)$ is equal
to a nontrivial $(2,\cV,8)$--combination, contradicting the
assumption that $\cV$ is an 8--good qif.  
\medskip

Thus Case 1 cannot happen and by Claim \ref{cl1} either
$\bar{W}_0(0,1)=\bar{W}_1(1,2)$, or $\bar{W}_1(0,1)=\bar{W}_0(1,2)$. 
\medskip

\noindent {\sc Case 2}:\quad $\bar{W}_0(0,1)=\bar{W}_1(1,2)$.\\
By the homogeneity, $\bar{W}_0(j,8)=\bar{W}_1(8,9)$ for each $j<8$.
By Claim \ref{cl3}(2), for every $j<8$,  $a_j-a_9=\bb_0(j,8)
-\bb_1(j,8) +\bb_0(8,9)-\bb_1(8,9)$, so 
\[(a_9+\bb_0(8,9))-a_j=\big(\bb_1(8,9) -\bb_0(j,8)\big)+\bb_1(j,8)\in 
\big(\bar{W}_0(j,8)-\bar{W}_0(j,8)\big)+\bar{W}_1(j,8)\]
Since $\cW$ is immersed in $\cV$, the set on the far right above is
included in $\bar{V}_1(j,8)$. Hence for $c=a_9+\bb_0(8,9)$ and $
A'=\{a_j:j<8\}$ we have $c- A'\subseteq \bigcup\cV$. Using Lemma
\ref{fromKupMin}(2) we may conclude that $c- A\subseteq \bigcup\cQ$.  
\medskip

\noindent {\sc Case 3}:\quad $\bar{W}_1(0,1)=\bar{W}_0(1,2)$. \\
By the homogeneity, $\bar{W}_1(j,8)=\bar{W}_0(8,9)$ for each $j<8$. As
before we use Claim \ref{cl3}(2) to get 
\[(\bb_1(8,9)-a_9)+ a_j=\big(\bb_0(8,9)-\bb_1(j,8)\big) +\bb_0(j,8)
\in \big(\bar{W}_0(8,9)-\bar{W}_0(8,9)\big)+\bar{W}_0(j,8)\]
Since $\cW$ is immersed in $\cV$, the set on the far right above is
included in $\bar{V}_0(j,8)$. Thus for $c=\bb_1(8,9)-a_9$ and $
A'=\{a_j:j<8\}$ we have $ A'+c\subseteq \bigcup\cV$. By Lemma
\ref{fromKupMin}(1) we get $ A+c\subseteq \bigcup \cQ$. 
\bigskip

Finally, to show that only one of (A) and (B) may take place, suppose
$ A+c\subseteq \bigcup\cQ$ and $d- A\subseteq \bigcup\cQ$ for some  
$c,d\in \bbH$. For $a\in A$ let $Q_a,Y_a\in\cQ$ be such that 
$a+c\in Q_a$ and $d-a\in Y_a$. 

Fix any $a\in A$ and choose $b\in A\setminus \big(\{a\}\cup
(Y_a-c)\cup (d-Q_a)\big)$ (it is possible as by the assumption
\ref{gettranslMin}(e), $|A\cap (Y_a-c)|<2$ and $|A\cap
(d-Q_a)|<2$). Now, 
\[(a+c)+(d-a) =c+d=(b+c)+(d-b),\]
so $0\in Q_a+Y_a-Q_b-Y_b$. By the choice of $b$ we have $Q_b\neq Y_a$,
$Q_a\neq Y_b$ and also (by \ref{gettranslMin}(e)) $Q_a\neq  Q_b$ and
$Y_a\neq Y_b$. Therefore some nontrivial $(2,\cQ,4)$--combination is
equal to 0, contradicting $\cQ$ is a good qif. 
\end{proof}

\section{Quasi independence and sums}
In a special case when $\cQ,\cV,\cW$ are all families consisting of
singletons (and $\rho$ is the discrete metric on $\bbH$), Theorem
\ref{gettranslMin} gives the following result of its own interest. 

\begin{corollary}
 Suppose that $(\bbH,+,0)$ is an Abelian group and  $\bB\subseteq
 \bbH$ is quasi independent. Assume also that
 \begin{enumerate}
 \item[(a)] $m\longrightarrow \big(10\big)^4_{2^{144}}$,
 \item[(b)] $ A\subseteq \bbH$, $| A|\geq m$ and $ A- A\subseteq
   \bB-\bB$. 
 \end{enumerate}
 Then exactly one of (A), (B) below holds true:
 \begin{enumerate}
 \item[(A)]   There is a unique $c\in \bbH$ such that $ A+c\subseteq \bB$.
 \item[(B)]   There is a unique $c\in \bbH$ such that $c- A\subseteq \bB$.    
 \end{enumerate}
\end{corollary}

The above Corollary inspired our interest in its dual version when $A-A$
and $\bB-\bB$ are replaced by $A+A$ and $\bB+\bB$. This dual result (given
in Theorem \ref{gettransl} below) is not used in the proof of our
independence theorem, but we find it interesting. 

\begin{lemma}
\label{fromKup}
Suppose that $(\bbH,+,0)$ is an Abelian group and  $\bB\subseteq \bbH$ is quasi
independent. Assume that $ A'\subseteq A\subseteq \bbH$ and $c\in \bbH$ are such
that  
\begin{enumerate}
\item[(a)] $ A+ A\subseteq \bB+\bB$,  
\item[(b)] $ A'+c\subseteq \bB$ and $| A'|=4$.
\end{enumerate}
Then $ A-c\subseteq \bB$.
\end{lemma}

\begin{proof}
Suppose that $ A'\subseteq A\subseteq \bbH$ satisfy the assumptions (a) and 
(b).   Assume $a\in A$ and let us argue that $a-c\in \bB$.

Let $\langle a_i:i<4\rangle$ list the elements of $ A'$. For $i<4$ let
$b_i=a_i+c\in\bB$ and note that all $b_i$'s are pairwise distinct. Since
$a_i+a\in\bB+\bB$ we  may also choose $b_i',b_i''\in
\bB$ such that $a_i+a=b_i'+b_i''$. Then, for each $i<4$, we have 
\[a-c=a-(b_i-a_i)=b_i'+b_i''-a_i-(b_i-a_i)= b_i'+b_i''-b_i.\]
Thus for $i<j<4$ we have 
\begin{enumerate}
\item[$(*)_1$] \quad $0=(b_i'+b_i''+b_j)-(b_j'+b_j''+b_i)$.
\end{enumerate}
If for some $i<j<4$ both sets $\{b_i',b_i'',b_j\}$ and $\{b_j',b_j'',b_i\}$
had at least 2 elements, then the right hand side of $(*)_1$ would give a
$(2,8)$--combination from $\bB$ with the value 0, so the combination would
have to be a trivial one. Therefore
\begin{enumerate}
\item[$(*)_2$] for each $i<j<4$,
  \begin{enumerate}
  \item[either (i)] $b_i'=b_i''=b_j$,
  \item[or (ii)] $b_j'=b_j''=b_i$,
  \item[or (iii)] $\{b_i',b_i'',b_j\}=\{b_j',b_j'',b_i\}$.
  \end{enumerate}
\end{enumerate}
Suppose that $i<j<4$ are such that $(*)_2$(iii) holds true. Since $b_i\neq
b_j$, we get $b_i\in\{b_i',b_i''\}$ and hence $a-c=b_i'+b_i''-b_i\in
 \{b_i',b_i''\} \subseteq\bB$, and we are done.

Assume towards contradiction that 
\begin{enumerate}
\item[$(*)_3$]  for each $i<j<4$, either $(*)_2$(i) or $(*)_2$(ii) holds true.
\end{enumerate}
Then for some $i_0<4$, $b_j'=b_j''$ whenever $j\neq i_0$. Necessarily,
\[\Big(j_0\neq j_1\ \wedge\ i_0\notin\{j_0,j_1\}\Big)\ \Rightarrow\
  b_{j_0}'\neq b_{j_1}'\] 
(as $a+a_{j_0} \neq a+a_{j_1}$). Since there are no repetitions among
$b_j$'s, we may now choose $j\neq i_0$ such that $b_j\neq b_{i_0}'$,
$b_j'\neq b_{i_0}$ getting immediate contradiction with our assumption
$(*)_3$. 
\end{proof}

\begin{lemma}
\label{4up}
Suppose that $(\bbH,+,0)$ is an Abelian group and  $\bB\subseteq \bbH$ is quasi 
independent. Assume that $ A'\subseteq A\subseteq \bbH$ are such that  
\begin{enumerate}
\item[(a)] $ A+ A\subseteq \bB+\bB$,  
\item[(b)] $| A'|\geq 4$, and $ A'+c\subseteq \bB$ for some $c\in \bbH$.
\end{enumerate}
Then $ A+c\subseteq \bB$ and the order of $c$ is $\leq 2$. 
\end{lemma}

\begin{proof}
Let $ A'+c\subseteq \bB$. It follows from Lemma \ref{fromKup} that
$ A-c\subseteq \bB$. Applying that lemma again for $ A', A,\bB$ and $-c$
we get $ A+c\subseteq \bB$.

Concerning the second part of the assertion, suppose towards contradiction
that $c+c\neq 0$. Let $a_0,a_1,a_2,a_3$ be distinct elements of $ A$. Then
for distinct $i,j<4$ we have 
\[a_i+c\neq a_i-c,\quad a_i+c\neq a_j+c,\quad \mbox{and}\quad a_i-c\neq
  a_j-c,\]
and consequently we may find $i<4$ such that $\{a_0+c,a_0-c\}\cap \{a_i+c, 
a_i-c\}=\emptyset$. Then, by the first paragraph of this proof, 
$a_0+c,a_0-c,a_i+c, a_i-c \in\bB$ are all distinct and $(a_0+c)-(a_0-c)-
(a_i+c)+(a_i-c)=0$, contradicting the quasi independence of $\bB$.
\end{proof}

\begin{theorem}
\label{gettransl}
 Suppose that $(\bbH,+,0)$ is an Abelian group and  $\bB\subseteq \bbH$ is quasi
 independent. Assume also that
 \begin{enumerate}
 \item[(a)] $m\longrightarrow \big(6\big)^4_{2^{144}}$,
 \item[(b)] $ A\subseteq \bbH$, $| A|\geq m$ and $ A+ A\subseteq
   \bB+\bB$. 
 \end{enumerate}
Then there is a unique $c\in \bbH$ of order $\leq 2$ such that $ A+c\subseteq 
\bB$. 
\end{theorem}

\begin{proof}
Let $\langle a_i:i<m\rangle$ be a sequence of pairwise distinct elements of
$ A$. Since $ A+ A\subseteq \bB+\bB$, we may choose symmetric functions 
$\bb_0,\bb_1:m\times m\longrightarrow \bB$ such that 
\[a_i+a_j=\bb_0(i,j)+\bb_1(i,j)\qquad \mbox{ for all }\ i,j<m.\]

Let $\langle \varphi_\ell(i_0,i_1,i_2,i_3):\ell<144\rangle$ list all
formulas of the form 
\[\bb_j(i_x,i_y)=\bb_{j'}(i_{x'},i_{y'})\]
for $j,j'<2$ and $x<y<4$, $x'<y'<4$. 

Let $\mu:\big[ m\big]^{\textstyle 4}\longrightarrow {}^{144}2$ be a
coloring of quadruples from $m$ such that if $i_0<i_1<i_2<i_3<m$, then
\[\mu\big(\{i_0,i_1,i_2,i_3\}\big)\big(\ell\big)=1\quad \mbox{ if and only
    if }\quad \varphi_\ell(i_0,i_1,i_2,i_3)\mbox{ holds true.}\]
Since $m\longrightarrow \big(6\big)^4_{2^{144}}$, we may choose $u\in
[m]^{\textstyle 6}$ homogeneous for $\mu$. Without loss of generality,
$u=\{0,1,2,3,4,5\}$.  

\begin{claim}
  \label{cl5}
 If $\{\bb_0(i,j):i<j<6\}\cap \{\bb_1(i,j):i<j<6\}\neq \emptyset$, then
 either $\bb_0(0,1)=\bb_1(1,2)$, or $\bb_1(0,1)=\bb_0(1,2)$, or
 $\bb_0(0,1)=\bb_1(0,1)$. 
\end{claim}

\begin{proof}[Proof of the Claim]
Suppose $i_0<j_0<6$ and $i_1<j_1<6$ are such that $\bb_0(i_0,j_0) =
\bb_1(i_1,j_1)$. We shall consider all possible orders of $i_0,j_0,i_1,j_1$
and use the homogeneity to conclude one of the clauses in the assertion.
\medskip

\noindent (a)\quad If $i_0<j_0<i_1<j_1$, then (by the homogeneity)
$\bb_0(0,1)=\bb_1(2,3) =\bb_1(4,5) =\bb_0(2,3)$, so also
$\bb_0(0,1)=\bb_1(0,1)$. 
\medskip

\noindent (b)\quad If $i_0<j_0=i_1<j_1$ then also $\bb_0(0,1)=\bb_1(1,2)$.
\medskip

\noindent (c)\quad If $i_0<i_1<j_0<j_1$, then $\bb_0(0,3)=\bb_1(2,4)=
\bb_1(1,4)=\bb_0(0,2)$ and also $\bb_0(0,1)=\bb_1(1,2)$. 
\medskip

\noindent (d)\quad If $i_0<i_1<j_0=j_1$, then $\bb_0(0,3)=\bb_1(1,3)=
\bb_1(2,3)= \bb_0(1,3)$ and also $\bb_0(0,1)=\bb_1(0,1)$.
\medskip

\noindent (e)\quad If $i_0<i_1<j_1<j_0$, then $\bb_0(0,5)=\bb_1(3,4)=
\bb_1(1,2) =\bb_0(0,3)$ and also $\bb_0(0,1)=\bb_1(1,2)$.
\medskip

\noindent (f)\quad If $i_0=i_1<j_0<j_1$, then $\bb_0(0,1)=\bb_1(0,2) =
\bb_1(0,3)=\bb_0(0,2)$, so also $\bb_0(0,1)=\bb_1(0,1)$.
\medskip

\noindent (g)\quad If $i_0=i_1<j_0=j_1$ then $\bb_0(0,1)=\bb_1(0,1)$.
\medskip

\noindent (h)\quad If $i_0=i_1<j_1<j_0$, then $\bb_0(0,2)=\bb_1(0,1) =
\bb_0(0,3)= \bb_1(0,2)$, so also $\bb_0(0,1)=\bb_1(0,1)$.
\medskip

\noindent (i)\quad If $i_1<i_0<j_0<j_1$, then $\bb_1(0,1)=\bb_0(1,2)$
similarly to (e), just interchange $\bb_0$ and $\bb_1$.
\medskip

\noindent (j)\quad If $i_1<i_0<j_0=j_1$, then $\bb_0(0,1)=\bb_1(0,1)$
similarly to (d).
\medskip

\noindent (k)\quad If $i_1<i_0<j_1<j_0$, then $\bb_1(0,1)=\bb_0(1,2)$
similarly to (c).
\medskip

\noindent (l)\quad If $i_1<i_0=j_1<j_0$, then $\bb_1(0,1)=\bb_0(1,2)$.
\medskip

\noindent (m)\quad If $i_1<j_1<i_0<j_0$, then $\bb_0(0,1)=\bb_1(0,1)$
similarly to (a).
\end{proof}

\begin{claim}
  \label{cl6}
  If $\bb_0(0,3)=\bb_0(1,2)$, then $\bb_0(0,1)=\bb_0(1,2)=\bb_0(2,3)=
  \bb_0(0,3)$. \\
  Similarly if $\bb_0$  is replaced by $\bb_1$.
\end{claim}

\begin{proof}[Proof of the Claim]
  Straightforward by the homogeneity of $u$.
\end{proof}

\begin{claim}
  \label{cl7}
\[\bb_0(0,1)+\bb_1(0,1)-\bb_0(1,2)-\bb_1(1,2) +\bb_0(2,3)+\bb_1(2,3) =
  \bb_0(0,3)+ \bb_1(0,3).\]
\end{claim}

\begin{proof}[Proof of the Claim]
 Follows by the choice of $\bb_0(i,j), \bb_1(i,j)$ and 

 $(a_0+a_1)-(a_1+a_2)+(a_2+a_3)=a_0+a_3$.
\end{proof}

\medskip

Now, we will consider six cases, showing that the first four of them are not
possible. In the remaining two cases we will find $c\in \bbH$ such that
$\{a_i+c:i<4\}\subseteq \bB$. Then by Lemma \ref{4up} we will also have
$ A+c\subseteq \bB$.
\medskip

\noindent {\sc Case 1}:\quad $\{\bb_0(i,j):i<j<6\}\cap
\{\bb_1(i,j):i<j<6\}=\emptyset$ and  $\bb_1(0,3)\notin \{\bb_1(0,1),
\bb_1(1,2), \bb_1(2,3)\}$.\\
Then $\bb_1(0,3)\notin \{\bb_0(0,1),\bb_1(0,1),\bb_0(1,2), \bb_1(1,2),
\bb_0(2,3),\bb_1(2,3), \bb_0(0,3)\}$ and by Claim \ref{cl7}
\[\bb_1(0,3)=\bb_0(0,1)+\bb_1(0,1)-\bb_0(1,2)-\bb_1(1,2) +\bb_0(2,3)+\bb_1(2,3)
  -\bb_0(0,3),\]
contradicting quasi independence of $\bB$.
\medskip

\noindent {\sc Case 2}:\quad $\{\bb_0(i,j):i<j<6\}\cap
\{\bb_1(i,j):i<j<6\}=\emptyset$ and  $\bb_0(0,3)\notin \{\bb_0(0,1),
\bb_0(1,2), \bb_0(2,3)\}$.\\
By an argument similar to Case 1, one shows that this case is not possible
as well. 
\medskip

\noindent {\sc Case 3}:\quad $\{\bb_0(i,j):i<j<6\}\cap
\{\bb_1(i,j):i<j<6\}=\emptyset$ and  $\bb_0(0,3)\in \{\bb_0(0,1),
\bb_0(1,2), \bb_0(2,3)\}$ and $\bb_1(0,3)\in \{\bb_1(0,1),
\bb_1(1,2), \bb_1(2,3)\}$.
\smallskip

\noindent {\sc Subase 3A}:\quad $\bb_0(0,3)=\bb_0(1,2)$.\\
Then by Claim \ref{cl6}, $\bb_0(0,3)=\bb_0(0,1)=\bb_0(1,2) = \bb_0(2,3)$.\\
If $\bb_1(0,3)=\bb_1(0,1)$, then $a_0+a_3=a_0+a_1$ and $a_3=a_1$, a 
contradiction.\\  
If $\bb_1(0,3)=\bb_1(2,3)$, then $a_0+a_3=a_2+a_3$ and $a_0=a_2$, a 
contradiction.\\  
If $\bb_1(0,3)=\bb_1(1,2)$, then Claim \ref{cl6} implies
$\bb_1(0,3)=\bb_1(0,1)$ and we already know that this leads to a
contradiction.\\ 
Consequently Subcase 3A is not possible.
\smallskip

\noindent {\sc Subase 3B}:\quad $\bb_1(0,3)=\bb_1(1,2)$.\\
Similarly as in  Subcase 3A one argues that this is not possible.
\smallskip

\noindent {\sc Subase 3C}:\quad $\bb_0(0,3)=\bb_0(0,1)$ and 
$\bb_1(0,3)=\bb_1(0,1)$.\\ 
Then $a_0+a_1=a_0+a_3$ and $a_1=a_3$ giving a contradiction. 
\smallskip 

\noindent {\sc Subase 3D}:\quad $\bb_0(0,3)=\bb_0(2,3)$ and
$\bb_1(0,3)=\bb_1(2,3)$.\\ 
Like Subcase 3C, this is not possible.
\smallskip

\noindent {\sc Subase 3E}:\quad $\bb_0(0,3)=\bb_0(0,1)$ and
$\bb_1(0,3)=\bb_1(2,3)$.\\ 
If we had $\bb_1(0,1)=\bb_1(1,2)$, then also (by the homogeneity)
$\bb_1(1,2)=\bb_1(2,3)$ and we  get a contradiction like in Subcase 3C.\\ 
If we had $\bb_0(1,2)=\bb_0(2,3)$ then also $\bb_0(2,3)=\bb_0(0,1)$ and we
get a contradiction like in Subcase 3D.\\
Consequently, there must be no repetitions in $\{\bb_0(1,2),\bb_0(2,3),
\bb_1(0,1), \bb_1(1,2)\}$. By Claim \ref{cl7} and the assumption of the
current subcase we have 
\[\bb_1(0,1)+\bb_0(2,3)-\bb_0(1,2)-\bb_1(1,2)=0,\]
a contradiction with the quasi independence of $\bB$. 
\smallskip

\noindent {\sc Subase 3F}:\quad $\bb_0(0,3)=\bb_0(2,3)$ and $\bb_1(0,3)=
\bb_1(0,1)$.\\ 
Like Subcase 3E, this is not possible.
\medskip

The next three cases cover the possibility when $\{\bb_0(i,j):i<j<6\}\cap
\{\bb_1(i,j):i<j<6\}\neq \emptyset$. By Claim \ref{cl5}, this implies that 
either $\bb_0(0,1)=\bb_1(1,2)$, or $\bb_1(0,1)=\bb_0(1,2)$, or
$\bb_0(0,1)=\bb_1(0,1)$.  
\medskip

\noindent {\sc Case 4}:\quad $\bb_0(0,1)=\bb_1(0,1)$\\
Then for all $i<j<6$ we have $\bb_0(i,j)=\bb_1(i,j)$.

If for some $i_0<j_0\leq i_1<j_1$ we had $\bb_0(i_0,j_0)= \bb_0(i_1,j_1)$,
then by the homogeneity we would have had $\bb_0(0,1)=
\bb_0(i,j)=\bb_1(i,j)$ for all $i<j<5$ and 
\[4b_0(0,1)=(a_0+a_1)+(a_1+a_2)=2a_1+2b_0(0,1).\]
Hence $2b_0(0,1)+(a_0+a_1)=4b_0(0,1)=2a_1+2b_0(0,1)$ and $a_0=a_1$, a
contradiction.

Therefore, $\bb_0(i_0,j_0)\neq \bb_0(i_1,j_1)$ whenever $i_0<j_0\leq 
i_1<j_1\leq 3$. Now, by Claim \ref{cl7},
\[\begin{array}{l}
  2\bb_0(0,3)=  \bb_0(0,3)+ \bb_1(0,3)=\\
  \bb_0(0,1)+\bb_1(0,1)-\bb_0(1,2)-\bb_1(1,2) +\bb_0(2,3)+\bb_1(2,3) =\\  
    2\bb_0(0,1)-2\bb_0(1,2)+2\bb_0(2,3).
    \end{array}\]
If we had $\bb_0(0,3)=\bb_0(1,2)$, then by the homogeneity $\bb_0(1,2) =
\bb_0(0,5)=\bb_0(2,3)$, contradicting what we said above. Therefore,
$\bb_0(0,3) \neq \bb_0(1,2)$ and $\bb_0(0,1),\bb_0(1,2),\bb_0(2,3)$ are
pairwise distinct. Hence
\[2\bb_0(0,1)-2\bb_0(1,2)+2\bb_0(2,3)-2\bb_0(0,3)\]
is a nontrivial $(2,8)$--combination with value $0$, a contradiction with
the quasi independence of $\bB$. 

Consequently, Case 4 is also impossible. 
\medskip

\noindent {\sc Case 5}:\quad $\bb_0(0,1)=\bb_1(1,2)$.\\
By the homogeneity, for each $j<4$ we have then
$\bb_0(j,4)=\bb_1(4,5)$. Hence for every $j<4$ we have
\[a_j+a_4=\bb_0(j,4)+\bb_1(j,4)=\bb_1(4,5)+\bb_1(j,4),\]
and consequently
\[a_j+(a_4-\bb_1(4,5))=\bb_1(j,4)\in\bB.\]
Thus letting $c=a_4-\bb_1(4,5)$ we will have $\{a_i+c:i<4\}\subseteq
\bB$. By Lemma \ref{4up} we also have  $ A+c\subseteq \bB$. 
\medskip

\noindent {\sc Case 6}:\quad $\bb_1(0,1)=\bb_0(1,2)$.\\
Similarly to Case 5, for each $j<4$ we have $\bb_1(j,4)=\bb_0(4,5)$ and
\[a_j+a_4=\bb_0(j,4)+\bb_1(j,4)=\bb_0(j,4)+\bb_0(4,5).\]
Hence $a_j+(a_4-\bb_0(4,5))=\bb_0(j,4)\in\bB$ and the rest is clear.
\bigskip

Concerning the uniqueness of $c$,  suppose towards contradiction
that $c\neq d$ are such that $ A+c\subseteq \bB$ and $ A+d\subseteq
\bB$. Let $a_0,a_1,a_2,a_3$ be distinct elements of $ A$. Then for
distinct $i,j<4$ we have 
\[a_i+c\neq a_i+d,\quad a_i+c\neq a_j+c,\quad \mbox{and}\quad
  a_i+d\neq a_j+d,\] 
and we may find $i<4$ such that $\{a_0+c,a_0+d\}\cap
\{a_i+c, a_i+d\}=\emptyset$. Then the elements  $a_0+c,a_0+d,a_i+c,
a_i+d$ belong to $\bB$, they are all distinct and
$(a_0+c)-(a_0+d)-(a_i+c)+(a_i+d)=0$, contradicting the quasi
independence  of $\bB$.  

Finally, Lemma \ref{4up} gives that $c$ must be of order at most 2.
\end{proof}

\section{Forcing for Abelian groups with few elements of order two}
In this and the next section, we will keep the following
notation/assumptions concerning our group $\bbH$.

\begin{hypothesis}
\label{hypno}
  \begin{enumerate}
  \item $(\bbH,+,0)$ is an Abelian perfect Polish group with the topology
    generated by a complete metric $\rho^*$. 
  \item $\bD\subseteq \bbH$ is a countable dense subset and $\rho: \bbH\times
    \bbH\longrightarrow [0,\infty)$ is a {\em translation invariant\/}  metric
    compatible with the topology of $\bbH$. (The metric $\rho$ does not have to
    be complete; it exists by the Birkhoff--Kakutani theorem.) 
  \item The open ball in the metric $\rho$ with radius $2^{-n}$ and center
    at $0$ is denoted $\bB_n$ and we let $\cU=\big\{d+\bB_n: d\in\bD\
    \wedge\ n<\omega\big\}$. By the invariance of the metric $\rho$, the
    family $\cU$ is a  countable base of the topology of $\bbH$.
  \end{enumerate}
\end{hypothesis}

Note that if $P\subseteq B\subseteq \bbH$ then $x+y\in (B+x)\cap (B+y)$ for
each $x,y\in P$. Consequently, if $P\subseteq B$ is a perfect set, then it
witnesses that $B$ has a perfect set of pairwise non-disjoint
translations. But for $k\geq 2$ we may and will introduce a forcing notion
adding a Borel set $B\subseteq \bbH$ which has many pairwise
$k$--overlapping translations but no perfect set of such translations.

The technical details force us to break up the construction into two
cases. First, we will assume that the group $\bbH$ has only a few elements
of rank 2. So, in addition to the assumptions and notation specified in
\ref{hypno}, in this section we assume the following:
\begin{hypothesis}
\label{hypno2}
  \begin{enumerate}
\item The set of elements of $\bbH$ of order larger than 2 is dense
    in $\bbH$.
\item $1<k<\omega$.
\item $\vare$ is a countable ordinal and $\lambda$ is an uncountable
  cardinal such that ${\rm NPr}^\vare(\lambda)$ holds true. The model
  $\bbM(\vare,\lambda)$ and functions $\rksp,\bj$ and $\bk$ on
  $[\lambda]^{<\omega}\setminus \{\emptyset\}$ are as fixed in Definition
  \ref{hypo2}. 
  \end{enumerate}
\end{hypothesis}

We will define a forcing notion $\bbP$ adding $\lambda$ many (distinct)
elements $\langle \eta_\alpha:\alpha<\lambda\rangle$ of the group $\bbH$ as
well as a sequence $\langle F_m:m<\omega\rangle$ of closed subsets of
$\bbH$. The $\Sigma^0_2$ subset $S=\bigcup\limits_{m<\omega} F_m$ of $\bbH$
will have the property that (in the forcing extension) 
\begin{enumerate}
\item[$(\heartsuit)_1$] there is {\em no perfect\/} set $P\subseteq \bbH$
  satisfying
  \[\big(\forall x,y\in P\big) \big(\big|(x+S)\cap (y+S)\big|\geq k\big).\] 
\end{enumerate}
At the same we will make sure that
\begin{enumerate}
\item[$(\heartsuit)_2$]  $\big|(-\eta_\alpha+S)\cap
  (-\eta_\beta+S)\big|\geq k$ for all $\alpha,\beta<\lambda$.
\end{enumerate}
To ensure $(\heartsuit)_2$ holds, the forcing will also add witnesses for
it: group elements $\nu_{i,\alpha,\beta}=\nu_{i,\beta,\alpha}\in \bbH$ and
integers $h_{\alpha,\beta}<\omega$ such that $\eta_\alpha+
\nu_{i,\alpha,\beta}\in F_{h_{\alpha,\beta}}$ (for $i<k$,
$\alpha,\beta<\lambda$). 

A condition $p\in\bbP$ will give a ``finite information'' on objects
mentioned above. Thus for some finite $w^p\subseteq \lambda$, for all distinct
$\alpha,\beta\in w^p$, the condition $p$ provides a basic open neighborhood
$U^p_\alpha(n^p)$ of $\eta_\alpha$, basic open neighborhood
$W^p_{i,\alpha,\beta}$ of $\nu_{i,\alpha,\beta}$ and the values of
$h_{\alpha,\beta} = h^p(\alpha,\beta)$. An approximation to the closed set
$F_m\subseteq \bbH$ will be given by its open neighborhood 
\[F(p,m)= \bigcup\big\{U^p_\alpha(n^p)+W^p_{i,\alpha,\beta}: 
     (\alpha,\beta)\in (w^p)^{\langle 2\rangle} \  \wedge\ i<k \  \wedge\ 
     h^p(\alpha,\beta)=m\big\}.\]     
Clause $(\heartsuit)_1$ as well as the ccc of the forcing $\bbP$ will
result from the involvement of the rank $\rksp$ and additional technical
pieces of information carried by conditions $p\in \bbP$: basic open sets
$Q^p_{i,\alpha,\beta}, V^p_{i,\alpha,\beta}$ and integers $r^p_m$. 

\begin{definition}
  \label{fordef}
{\bf (A)}\quad Let $\bbP$ be the collection of all tuples  
\[p=\big(w^p,M^p,\bar{r}^p,n^p,\bar{\Upsilon}^p,\bar{V}^p,h^p\big)=   
  \big(w,M,\bar{r},n,\bar{\Upsilon},\bar{V},h\big)\]      
such that the following demands $(\boxtimes)_1$--$(\boxtimes)_8$ are
satisfied.  
\begin{enumerate}
\item[$(\boxtimes)_1$] $w\in [\lambda]^{<\omega}$, $|w|\geq 4$,
  $0<M<\omega$, $3\leq n<\omega$ and $\bar{r}=\langle r_m:m<M\rangle
  \subseteq \omega$ with $r_m\leq n-2$  for $m<M$.  
\item[$(\boxtimes)_2$] $\bar{\Upsilon}= \langle \bar{U}_\alpha:\alpha\in
  w\rangle$ where each $\bar{U}_\alpha=\langle U_\alpha(\ell):\ell\leq n\rangle$
  is a $\subseteq$--decreasing sequence of elements of the basis $\cU$.
\item[$(\boxtimes)_3$]  $\bar{V}=\langle Q_{i,\alpha,\beta},
    V_{i,\alpha,\beta}, W_{i,\alpha,\beta}: i<k,\ (\alpha,\beta)\in
    w^{\langle 2\rangle}\rangle\subseteq \cU$ and  $Q_{i,\alpha,\beta}=
    Q_{i,\beta,\alpha}\supseteq V_{i,\alpha,\beta}=
  V_{i,\beta,\alpha} \supseteq W_{i,\alpha,\beta}= W_{i,\beta,\alpha}$ for
  all $i<k$ and $(\alpha,\beta)\in w^{\langle 2\rangle}$. 
\item[$(\boxtimes)_4$] 
  \begin{enumerate}
  \item[(a)] The indexed family $\langle U_\alpha(n-2):\alpha\in
    w\rangle\conc \langle Q_{i,\alpha,\beta}: i<k,\ \alpha,\beta\in   w,\
    \alpha<\beta \rangle$ is an 8--good qif (so in particular the sets in
    this system are pairwise disjoint), and  
  \item[(b)] $\langle U_\alpha(n):\alpha\in  w\rangle\conc \langle
W_{i,\alpha,\beta}: i<k,\ \alpha,\beta\in   w,\ \alpha<\beta \rangle$ is
immersed in $\langle U_\alpha(n-1):\alpha\in  w\rangle\conc \langle
V_{i,\alpha,\beta}: i<k,\ \alpha,\beta\in   w,\  \alpha<\beta \rangle$ and
$\langle U_\alpha(n-1):\alpha\in  w\rangle\conc \langle  V_{i,\alpha,\beta}:
i<k,\ \alpha,\beta\in   w,\ \alpha<\beta \rangle$ is immersed in $\langle
U_\alpha(n-2):\alpha\in w\rangle\conc \langle Q_{i,\alpha,\beta}:  
  i<k,\ \alpha,\beta\in   w,\ \alpha<\beta \rangle$; see Definition
  \ref{quasidef}(4) (so all these families are 8--good qifs). 
  \end{enumerate}
\item[$(\boxtimes)_5$] 
  \begin{enumerate}
\item[(a)] If $\alpha,\beta\in w$, $\ell\leq n$ and $U_\alpha(\ell)\cap 
U_\beta(\ell)\neq \emptyset$, then $U_\alpha(\ell)=U_\beta(\ell)$, and 
\item[(b)] if $\alpha,\beta,\gamma\in w$, $\ell\leq n$, $U_\alpha(\ell)\neq
  U_\beta(\ell)$ and $a\in U_\alpha(\ell)$, $b\in U_\beta(\ell)$, then
  $\rho(a,b)>{\rm diam}_\rho\big(U_\gamma(\ell)\big)$ ($= {\rm
    diam}_\rho\big(-U_\gamma(\ell)\big)$).
\end{enumerate}
\item[$(\boxtimes)_6$] $h: w^{\langle 2\rangle}\stackrel{\rm
    onto}{\longrightarrow} M$ is such that $h(\alpha,\beta)=h(\beta,\alpha)$
  for $(\alpha,\beta)\in w^{\langle 2\rangle}$.  
\item[$(\boxtimes)_7$] Assume that $u,u'\subseteq w$, $\pi$ and $\ell\leq n$
  are such that  
  \begin{itemize}
\item $4\leq |u|=|u'|$ and $\pi:u\longrightarrow u'$ is a bijection, 
\item $r_{h(\alpha,\beta)}\leq \ell$ for all $(\alpha,\beta)\in u^{\langle
    2\rangle}$,  
\item $U_\alpha(\ell)\cap U_\beta(\ell)=\emptyset$ and
  $h(\alpha,\beta)=h(\pi(\alpha),\pi(\beta))$ for all distinct
  $\alpha,\beta\in u$,
\item for some $c\in \bbH$, 
  \begin{enumerate}
  \item[either] for all $\alpha\in u$, we have 
$\big(U_\alpha(\ell)+c\big)\cap U_{\pi(\alpha)}(\ell)\neq \emptyset$
\item[or] for all $\alpha\in u$, we have 
$\big(c-U_\alpha(\ell)\big)\cap U_{\pi(\alpha)}(\ell)\neq \emptyset$.
  \end{enumerate}
\end{itemize}
Then $\rksp(u)=\rksp(u')$, $\bj(u)=\bj(u')$, $\bk(u)=\bk(u')$ and for
$\alpha \in u$ 
\[|\alpha\cap u|=\bk(u)\quad\Leftrightarrow\quad |\pi(\alpha)\cap u'|=
  \bk(u).\]  
\item[$(\boxtimes)_8$] Assume that  
  \begin{itemize}
  \item $\emptyset\neq u\subseteq w$, $\rksp(u)=-1$, $\ell\leq n$ and
\item $\alpha\in u$ is such that $|\alpha\cap u|=\bk(u)$, and 
  \item  $r_{h(\beta,\beta')}\leq \ell$  and $U_\beta(\ell)\cap
    U_{\beta'}(\ell)=\emptyset$ for all $(\beta,\beta')\in u^{\langle
      2\rangle}$.   
  \end{itemize}
Then there is {\bf no}\/ $\alpha'\in w\setminus u$ such that $U_\alpha(\ell)
=U_{\alpha'}(\ell)$ and $h(\alpha,\beta)= h(\alpha', \beta)$ for all
$\beta\in u\setminus\{\alpha\}$.   
\end{enumerate}

\noindent {\bf (B)}\quad For $p\in\bbP$ and $m<M^p$ we define
\[F(p,m)= \bigcup\big\{U^p_\alpha(n^p)+W^p_{i,\alpha,\beta}: 
     (\alpha,\beta)\in (w^p)^{\langle 2\rangle} \  \wedge\ i<k \  \wedge\ 
     h^p(\alpha,\beta)=m\big\}.\]     

\noindent {\bf (C)}\quad For $p,q\in \bbP$ we declare that $p\leq q$\quad if
and only if  
\begin{itemize}
\item $w^p\subseteq w^q$, $M^p\leq M^q$, $\bar{r}^q\rest M^p=\bar{r}^p$,  
  $n^p\leq n^q$, $h^q\rest (w^p)^{\langle 2\rangle} = h^p$, and 
\item if $\alpha\in w^p$ and $\ell\leq n^p$ then $U^q_\alpha(\ell) =
  U^p_\alpha(\ell)$, and 
\item if $(\alpha,\beta)\in (w^p)^{\langle 2\rangle}$, $i<k$, then
$Q^q_{i,\alpha, \beta}\subseteq Q^p_{i,\alpha,\beta}$, $V^q_{i,\alpha,
\beta} \subseteq V^p_{i,\alpha,\beta}$, and $W^q_{i,\alpha, \beta}\subseteq
W^p_{i,\alpha,\beta}$, and  
\item if $m<M^p$, then $F(q,m)\subseteq F(p,m)$. 
\end{itemize}
\end{definition}

\begin{lemma}
\label{dense}
\begin{enumerate}
\item $(\bbP,\leq)$ is a partial order of size $\lambda$.
\item The following sets are dense in $\bbP$:
  \begin{enumerate}
  \item[(i)] $D^0_{\gamma,M,n}=\big\{p\in\bbP:\gamma\in u^p\ \wedge \ M^p>M
    \ \wedge\ n^p>n \big\}$ for $\gamma<\lambda$ and $M,n<\omega$.
\item[(ii)] $D^1_N=\big\{p\in\bbP:\diam(U^p_\alpha(n^p-2))<2^{-N}\ \wedge\ 
  \diam(Q^p_{i,\alpha,\beta})<2^{-N} \ \wedge$

\qquad  $\diam(U^p_\alpha(n^p-2)+Q^p_{i,\alpha,\beta})< 2^{-N} \mbox{ 
  for all }i<k,\ (\alpha, \beta)\in (w^p)^{\langle 2\rangle}\big\}$

for $N<\omega$.  
\item[(iii)] $D^2_N=\big\{p\in\bbP:$ for all $i,j<k$ and 
$(\alpha,\beta), (\gamma,\delta)\in (w^p)^{\langle 2\rangle}$ it
  holds that

\qquad\quad ${\rm diam}_\rho(U^p_\alpha(n^p-2))<2^{-N}$\ and\ ${\rm
  diam}_\rho (Q^p_{i,\alpha,\beta})<2^{-N}$\ and  

\qquad\quad  ${\rm diam}_\rho (U^p_\alpha(n^p-2)+Q^p_{i,\alpha,\beta})<
2^{-N}$~and  

\qquad\quad if $(i,\alpha^*,\alpha,\beta)\neq (j,\gamma^*,\gamma,\delta)$   
 then 

\qquad\quad $\big(U^p_{\alpha^*}(n^p)+W^p_{i,\alpha,\beta}\big)\cap 
\big(U^p_{\gamma^*}(n^p)+W^p_{i,\gamma,\delta}\big)= \emptyset\big\}$. 

for $N<\omega$.  
  \end{enumerate}
\item Assume $p\in\bbP$. Then there is $q\geq p$ such that $n^q\geq n^p+3$,
  $w^q=w^p$ and  
  \begin{itemize}
\item for all $\alpha\in w^p$, $\cl\big(U^q_\alpha(n^q-2)\big)\subseteq 
  U^p_\alpha(n^p)$, and
\item for all $i<k$ and $(\alpha,\beta)\in (w^p)^{\langle 2\rangle}$, 
\[\cl\big(U^q_\alpha(n^q-2)+Q^q_{i,\alpha,\beta}\big)\subseteq
  U^p_\alpha(n^p)+W^p_{i,\alpha,\beta}\quad\mbox{ and }\quad
  \cl\big(Q^q_{i,\alpha,\beta}\big) \subseteq W^p_{i,\alpha,\beta}.\]
  \end{itemize}
\end{enumerate}
\end{lemma}

\begin{proof}
 (2)(i)\quad  Suppose $p\in\bbP$ and $\gamma\in \lambda\setminus w^p$. Let  
 $\alpha^*=\min(w^p)$ and let $w=w^p\cup \{\gamma\}$ and $n=n^p+3$. Using
 Proposition \ref{existence} we may choose $U_\alpha(n-2)\in \cU$ (for
 $\alpha\in w$) and $Q_{i,\alpha,\beta}\in \cU$ (for $i<k$,  $\alpha<\beta$,
 $\alpha,\beta\in w$) such that  
 \begin{itemize}
\item $U_\alpha(n-2)\subseteq U^p_\alpha(n^p)$ and
  $Q_{i,\alpha,\beta}\subseteq  W_{i,\alpha,\beta}^p$ when $\alpha,\beta\in
  w^p$,  
\item $U_\gamma(n-2)\subseteq U^p_{\alpha^*}(n^p)$, 
\item $\langle U_\alpha(n-2):\alpha\in w\rangle\conc \langle
  Q_{i,\alpha,\beta}:  i<k,\ \alpha<\beta,\ \alpha,\beta\in w\rangle$ is an
  8--good qif,  
\item ${\rm diam}_\rho\big(U_\delta(n-2)\big) = {\rm diam}_\rho\big(
  -U_\delta(n-2)\big)<\rho(a,b)$ for all $\delta\in w$, $(\alpha,\beta)
  \in w^{\langle 2\rangle}$, $a\in U_\alpha(n-2)$ and $b\in U_\beta(n-2)$.
 \end{itemize}
Then by Observation \ref{obs3.2}(3) we may choose $U_\alpha(n-1),
U_\alpha(n), V_{i,\alpha,\beta}, W_{i,\alpha,\beta}\in \cU$ (for
$\alpha<\beta$ from $w$ and $i<k$) such that  $U_\alpha(n)\subseteq
U_\alpha(n-1)\subseteq U_\alpha(n-2)$, $W_{i,\alpha,\beta}\subseteq  
V_{i,\alpha,\beta}\subseteq Q_{i,\alpha,\beta}$ and
\begin{itemize}
\item $\langle U_\alpha(n-1):\alpha\in  w\rangle\conc \langle 
  V_{i,\alpha,\beta}: i<k,\ \alpha,\beta\in   w,\ \alpha<\beta \rangle$ is  
  immersed in $\langle U_\alpha(n-2):\alpha\in  w\rangle\conc \langle 
  Q_{i,\alpha,\beta}: i<k,\ \alpha,\beta\in   w,\  \alpha<\beta 
  \rangle$, and 
\item $\langle U_\alpha(n):\alpha\in  w\rangle\conc \langle 
  W_{i,\alpha,\beta}: i<k,\ \alpha,\beta\in   w,\ \alpha<\beta \rangle$ is  
  immersed in $\langle U_\alpha(n-1):\alpha\in  w\rangle\conc \langle 
  V_{i,\alpha,\beta}: i<k,\ \alpha,\beta\in   w,\  \alpha<\beta 
  \rangle$.
\end{itemize}
Put $\bar{\Upsilon}=\langle \bar{U}_\alpha:\alpha\in w\rangle$, where 
 $\bar{U}_\alpha= \bar{U}^p_\alpha\conc \langle U_\alpha(n-2),
 U_\alpha(n-1), U_\alpha(n)\rangle$ if $\alpha\in w^p$ and
 $\bar{U}_\gamma=\bar{U}^p_{\alpha^*}\conc \langle U_\gamma(n-2),
 U_\gamma(n-1), U_\gamma(n)\rangle$. Let $Q_{i,\beta, \alpha} =
 Q_{i,\alpha,\beta}$, $V_{i,\beta, \alpha} =V_{i,\alpha,\beta}$ and
 $W_{i,\beta, \alpha} =W_{i,\alpha,\beta}$ (for $i<k$, $\alpha<\beta$ from
 $w$), and let $\bar{V}= \langle Q_{i,\alpha,\beta},V_{i,\alpha,\beta}, 
 W_{i,\alpha,\beta}: i<k,\ (\alpha,\beta)\in w^{\langle 2\rangle} 
 \rangle$. Let $M=M^p+|w^p|$ and let $h: w^{\langle 2\rangle}
 \longrightarrow M$ be such that   
\begin{itemize}
\item $h(\alpha,\beta)=h^p(\alpha,\beta)$ when $(\alpha,\beta)\in
  (w^p)^{\langle 2\rangle}$,
\item $h(\alpha,\gamma)=h(\gamma,\alpha)=M^p+j$ when $\alpha\in w^p$ and
  $j=|w^p\cap \alpha|$.   
\end{itemize}
We also define $\bar{r}:M\longrightarrow (n-1)$ so that $\bar{r}\rest
M^p=\bar{r}^p$ and $r_m=n-2$ for $m\in [M^p,M)$

Put $q=(w,M,r,n,\bar{\Upsilon},\bar{V},h)$. Let us argue that $q\in\bbP$. To
this end we have to verify conditions $(\boxtimes)_1$--$(\boxtimes)_8$ of
Definition \ref{fordef}. Of these the first six demands follow immediately
by our choices. To show $(\boxtimes)_7$, suppose $u,u'\subseteq w$ and
$\pi:u \longrightarrow u'$ and $\ell\leq n$ and $c\in \bbH$ satisfy the
assumptions there. If $\alpha\in w^p$  then $h(\gamma,\alpha)\geq M^p$ and
therefore $\gamma\in u$ if and only if $\gamma\in u'$. If $\gamma\notin u$,
then $u\cup u'\subseteq w^p$ and clause $(\boxtimes)_7$ for $p$ (applied to
$\min(n^p,\ell)$ instead of $\ell$) gives the needed conclusion. If
$\gamma\in u$, then $\gamma\in u'$ too and we look at $h(\beta,\gamma)$ for
$\beta \in u\cap w^p$. Each of these values is taken by $h$ exactly one time, so
$h(\pi(\beta),\pi(\gamma))=h(\beta,\gamma)$ for all $\beta\in w^p$ implies
that $\pi(\gamma)=\gamma$ and $\pi(\beta)=\beta$ for $\beta\in u\cap
w^p$. Hence $u=u'$ and $\pi$ is the identity, so the desired  conclusion
follows.   

Now suppose $\ell\leq n$, $\alpha\in u\subseteq w$ are as in the assumptions
of $(\boxtimes)_8$ (so by \ref{hypo2}$(\circledast)_b$ also $u\geq 2$). If
$\gamma\notin u$, then applying $(\boxtimes)_8$ for $p$ to $\alpha,u$ and
$\ell'=\min(\ell,n^p)$ we see that there is no $\alpha'\in w^p\setminus u$
with $U_\alpha(\ell)=U_{\alpha'}(\ell)$, and $h(\alpha,\beta) =
h(\alpha',\beta)$ for all $\beta\in u\setminus \{\alpha\}$. The  values of
$h(\beta,\gamma)$ (for $\beta\in u$) are above $M^p$, so they cannot be
equal to $h(\beta,\alpha)$ either. Consequently, the conclusion of
$(\boxtimes)_8$ holds in this case.  So assume now that $\gamma\in
u\setminus\{\alpha\}$. The value of $h(\gamma,\alpha)$ is taken exactly
once, so no $\alpha'\in w \setminus \{\gamma, \alpha\}$ satisfies
$h(\gamma,\alpha)=h(\gamma,\alpha')$ and the desired conclusion should be
clear now. Finally, assume $\gamma=\alpha$. As we said, $|u|\geq 2$ so we
may take $\beta\in u\setminus\{\gamma\}$ and look at
$h(\gamma,\beta)$. There is no $\alpha'\in w\setminus\{\gamma\}$ with 
$h(\alpha',\beta)=h(\gamma,\beta)$, so desired conclusion follows,
finishing the proof of $(\boxtimes)_8$.

Now one easily deduces (2)(i).
\medskip

\noindent (ii)\quad Assume $p\in \bbP$ and $N<\omega$. For
$(\alpha,\beta)\in  (w^p)^{\langle 2\rangle}$ and $i<k$ first choose
$U_\alpha(n^p+1), Q_{i,\alpha,\beta}\in \cU$ such that $U_\alpha(n^p+1)
\subseteq U^p_\alpha(n^p)$, $Q_{i,\alpha,\beta}= Q_{i,\beta,\alpha}\subseteq
W^p_{i,\alpha,\beta}$ and $\rho^*$--diameters of $U_\alpha(n^p+ 1),
Q_{i,\alpha,\beta}$ and $U_\alpha(n^p+1)+ Q_{i,\alpha,\beta}$ are all
smaller than $2^{-N}$. Note that $\langle U_\alpha(n^p+1):\alpha\in
w^p\rangle\conc\langle Q_{i,\alpha,\beta}: i<k,\ \alpha<\beta,\
\alpha,\beta\in w^p\rangle$ is an 8--good qif.  Next, use Proposition
\ref{existence} to choose $U_\alpha(n^p+2), U_\alpha(n^p+3),
V_{i,\alpha,\beta}, W_{i,\alpha,\beta}\in\cU$ such that
$U_\alpha(n^p+3)\subseteq U_\alpha(n^p+2)\subseteq  U_\alpha(n^p+1)$,
and $W_{i,\alpha,\beta}= W_{i,\beta,\alpha}\subseteq
V_{i,\alpha,\beta}=V_{i,\beta,\alpha}\subseteq Q_{i,\alpha,\beta}$
(for $(\alpha,\beta)\in (w^p)^{\langle 2\rangle}$ and $i<k$), and  
\begin{itemize}
\item $\langle U_\alpha(n^p+3):\alpha\in w^p\rangle\conc\langle 
  W_{i,\alpha,\beta}: i<k,\ \alpha<\beta,\ \alpha,\beta\in w^p\rangle$ is 
  immersed in $\langle U_\alpha(n^p+2):\alpha\in w^p\rangle\conc\langle 
  V_{i,\alpha,\beta}: i<k,\ \alpha<\beta,\ \alpha,\beta\in w^p\rangle$, and    
\item $\langle U_\alpha(n^p+2):\alpha\in w^p\rangle\conc\langle 
  V_{i,\alpha,\beta}: i<k,\ \alpha<\beta,\ \alpha,\beta\in w^p\rangle$ is 
  immersed in $\langle U_\alpha(n^p+1):\alpha\in w^p\rangle\conc\langle 
  Q_{i,\alpha,\beta}: i<k,\ \alpha<\beta,\ \alpha,\beta\in
  w^p\rangle$.
\end{itemize}
Now, for $\alpha\in w^p$ let $\bar{U}_\alpha=\bar{U}^p_\alpha\conc \langle 
U_\alpha(n^p+1), U_\alpha(n^p+2), U_\alpha(n^p+3)\rangle$ and then let
$\bar{\Upsilon}=\langle \bar{U}_\alpha:\alpha\in w^p\rangle$ and
$\bar{V}=\langle Q_{i,\alpha,\beta},V_{i,\alpha,\beta},W_{i,\alpha,\beta}:
i<k,\ (\alpha,\beta)\in  (w^p)^{\langle 2\rangle}\rangle$. These choices
clearly determine a condition $q=(w^p,M^p,\bar{r}^p,n^p+3, \bar{\Upsilon}, 
\bar{V},h^p) \in D^1_N$ stronger than $p$. 
\medskip

\noindent (iii)\quad Similarly to (ii), we just make $U_\alpha(n^p+1)$,
$U_\alpha(n^p+2)$, $U_\alpha(n^p+3)$,  $V_{i,\alpha,\beta}$,
$Q_{i,\alpha,\beta}$ and $W_{i,\alpha,\beta}$ suitably small.  
\medskip

\noindent (3)\quad Analogous.
\end{proof}

\begin{lemma}
  \label{abc}
Suppose that $p\in\bbP$ and $\alpha,\beta,\gamma,\delta\in w^p$ are such
that $\alpha\neq \beta$. If
\[\Big(U^p_\alpha(n^p-2)-U^p_\beta(n^p-2)\Big)\cap
  \Big(U^p_\gamma(n^p-2)-U^p_\delta(n^p-2)\Big)\neq \emptyset,\]
then $\alpha=\gamma$ and $\beta=\delta$.
\end{lemma}

\begin{proof}
Let $n=n^p$. Suppose that $a\in U^p_\alpha(n-2)$, $b\in U^p_\beta(n-2)$,
$c\in U^p_\gamma(n-2)$ and $d\in U^p_\delta(n-2)$ are such that
$a-b=c-d$. Then $a+(c-a)=c$ and $b+(c-a)=d$, so as $\rho$ is invariant we
have $\rho(a,b)=\rho(c,d)$. Demand \ref{fordef}(A)$(\boxtimes)_5$(b)
implies that $\rho(a,b)>{\rm diam}_\rho(U^p_\gamma(n-2))$ and hence
$\gamma\neq \delta$. Now look at $a+d-b-c$: since $\alpha\neq\beta$ and
$\gamma\neq\delta$ it is a (2,4)--combination from an 8--good qif
$\langle U^p_\zeta(n-2): \zeta\in w\rangle$. Since the value of
the combination is 0, it has to be trivial. Hence immediately
$\alpha=\gamma$ and $\beta=\delta$.
\end{proof}

\begin{lemma}
  \label{Knaster}
  The forcing notion $\bbP$ has the Knaster property.
\end{lemma}

\begin{proof}
 Suppose $\langle p_\vare:\vare<\omega_1\rangle$ is a sequence of
 pairwise distinct conditions from $\bbP$. Applying standard
 $\Delta$--lemma based cleaning procedure we may find $w_0\subseteq
 \lambda$ and $A\in [\omega_1]^{\omega_1}$ such that for distinct
 $\xi,\zeta\in A$ the following demands $(*)_1+(*)_2$ are satisfied. 
 \begin{enumerate}
 \item[$(*)_1$]  $|w^{p_\xi}|=|w^{p_\zeta}|$, $w_0=w^{p_\xi}\cap
   w^{p_\zeta}$, $M^{p_\xi}=M^{p_\zeta}$, $n^{p_\xi}=n^{p_\zeta}$,
   $\bar{r}^{p_\xi}= \bar{r}^{p_\zeta}$.
 \item[$(*)_2$] If $\pi^*:w^{p_\zeta}\longrightarrow w^{p_\xi}$ is the
   order isomorphism, then
   \begin{itemize}
   \item $\pi^*\rest w_0$ is the identity,
\item $\bar{U}^{p_\zeta}_\alpha(\ell)=
 \bar{U}^{p_\xi}_{\pi^*(\alpha)}(\ell)$ whenever  $\alpha\in w^{p_\zeta}$,
 $\ell\leq n^{p_\zeta}$,  
\item if $(\alpha,\beta)\in \big(w^{p_\zeta}\big)^{\langle 2\rangle}$,
  $i<k$, then $h^{p_\zeta}(\alpha,\beta)= h^{p_\xi}( \pi^*(\alpha),\pi^* 
(\beta))$, and 
\[Q^{p_\zeta}_{i,\alpha,\beta} =  Q^{p_\xi}_{i,\pi^*(\alpha), 
\pi^*(\beta)},\quad V^{p_\zeta}_{i,\alpha,\beta} =  V^{p_\xi}_{i,\pi^*(\alpha),
\pi^*(\beta)}\ \mbox{ and }\ W^{p_\zeta}_{i,\alpha,\beta} =
W^{p_\xi}_{i,\pi^*(\alpha), \pi^*(\beta)},\] 
\item if $\emptyset\neq u\subseteq w^{p_\zeta}$, then $\rksp(u)=
  \rksp(\pi^*[u])$, $\bj(u)=\bj(\pi^*[u])$ and $\bk(u)=\bk(\pi^*[u])$. 
   \end{itemize}
 \end{enumerate}
 Note that then for all $\xi\in A$ we have
 \begin{enumerate}
 \item[$(*)_3$] if $u\subseteq w_0$, $\alpha\in w^{p_\xi}\setminus w_0$ and
   $\rksp\big( u\cup\{\alpha\}\big)=-1$, then $\bk\big(u\cup\{\alpha\}\big)
   \neq |u\cap\alpha|$. 
 \end{enumerate}
 [Why? Suppose towards contradiction that $\bk\big(u\cup\{\alpha\}\big) =
 |u\cap\alpha|$. For $\zeta\in A$ let $\alpha_\zeta\in
 w^{p_\zeta}$ be such that $|\alpha_\zeta\cap w^{p_\zeta}|=|\alpha\cap
 w^{p_\xi}|$. By $(*)_2$ we have
 \[j\stackrel{\rm def}{=}\bj\big(u\cup\{\alpha\}\big) =
\bj\big(u\cup\{\alpha_\zeta\}\big) 
\ \mbox{ and }\  \bk\big( u\cup\{\alpha_\zeta\}\big)=
\bk\big(u\cup\{\alpha\}\big)= |u\cap \alpha| = |u\cap \alpha_\zeta|
\stackrel{\rm def}{=}k.\] 
Therefore, letting $u\cup\{\alpha\}=\{\alpha_0,\ldots,\alpha_{\ell-1}\}$ be  
the increasing enumeration, we have $\alpha_k=\alpha$ and
\[\bbM\models R_{\ell,j}[\alpha_0,\ldots,\alpha_{k-1},\alpha_\zeta,
  \alpha_{k+1}, \ldots,\alpha_{\ell-1}]\qquad \mbox{ for all }\zeta\in A.\]  
However, this contradicts the choice of $\bj,\bk$ in Definition \ref{hypo2}
and the assumption $\rksp\big(u\cup\{\alpha\}\big)=-1$.] 
\medskip

 We will argue now that for $\xi,\zeta\in A$ the conditions $p_\xi,p_\zeta$
 are compatible. So let $\xi<\zeta$ be from $A$ and let $\pi^*:w^{p_\zeta}
 \longrightarrow w^{p_\xi}$ be the order isomorphism. Set $w=w^{p_\xi}\cup 
 w^{p_\zeta}$, $M=M^{p_\xi}+\big | w^{p_\xi}\setminus
 w^{p_\zeta}\big|^2$, $n=n^{p_\xi}+3$ and let $\bar{r}=\langle r_m:
 m<M\rangle$ be such that $r_m=r^{p_\xi}_m$ if  $m<M^{p_{\xi}}$, and
 $r_m=n-2$ if $M^{p_{\xi}}\leq m<M$.  

 Use Proposition \ref{existence} and Observation \ref{obs3.2}(iii) to choose 
$U_\alpha(n-2)$, $U_\alpha(n-1)$, $U_\alpha(n)$, $Q_{i,\alpha,\beta}$,
$V_{i,\alpha,\beta}$ and  $W_{i,\alpha,\beta}$ from $\cU$ for $i<k$ and
$(\alpha,\beta)\in w^{\langle 2\rangle}$ so that  
\begin{enumerate}
\item[$(*)_4$] 
\begin{enumerate}
\item[(a)] demands \ref{fordef}$(\boxtimes)_3$--$(\boxtimes)_5$ are  
  satisfied and
\item[(b)] if $(\alpha,\beta)\in \big(w^{p_\xi}\big)^{\langle 2\rangle}$, $i<k$, 
  then $U_\alpha(n-2)\subseteq U^{p_\xi}_\alpha(n^{p_\xi})$ and
  $Q_{i,\alpha,\beta}\subseteq  W_{i,\alpha,\beta}^{p_\xi}$, and
\item[(c)] if $(\alpha,\beta)\in \big(w^{p_\zeta}\big)^{\langle 2\rangle}$, $i<k$, 
  then $U_\alpha(n-2)\subseteq U^{p_\zeta}_\alpha(n^{p_\zeta})$ and
  $Q_{i,\alpha,\beta}\subseteq  W_{i,\alpha,\beta}^{p_\zeta}$.
\end{enumerate}
\end{enumerate}
Let $\bar{U}_\alpha=\bar{U}^{p_\xi}_\alpha\conc \langle U_\alpha(n-2),
U_\alpha(n-1), U_\alpha(n)\rangle$ if $\alpha\in w^{p_\xi}$ and
$\bar{U}_\alpha= \bar{U}^{p_\zeta}_\alpha\conc \langle  U_\alpha(n-2),
U_\alpha(n-1), U_\alpha(n)\rangle$ if $\alpha\in w^{p_\zeta}$, and let
$\bar{\Upsilon},\bar{V}$ be defined naturally. Choose $h:w^{\langle
  2\rangle} \longrightarrow M$ extending both $h^{p_\xi}$ and
$h^{p_\zeta}$ in such a manner that $h(\alpha,\beta)=h(\beta,\alpha)$ for
$(\alpha,\beta)\in w^{\langle 2\rangle}$ and the mapping    
\[\big (w^{p_\xi}\setminus w_0\big)\times \big(w^{p_\zeta}\setminus w_0\big)
  \ni (\alpha,\beta)\mapsto h(\alpha,\beta)\]
is a bijection onto $[M^{p_\xi},M)$. Finally we set
$q=(w,M,\bar{r},n,\bar{\Upsilon},\bar{V},h)$. 

Let us argue that $q\in\bbP$ (once we are done with that, it should be clear
that $q$ is stronger than both $p_\xi$ and $p_\zeta$). The only potentially
unclear demands to verify are $(\boxtimes)_7$ and $(\boxtimes)_8$ of
\ref{fordef}. 

First, to demonstrate $(\boxtimes)_7$, suppose that $u,u'\subseteq w$ and
$\pi:u\longrightarrow u'$ and $\ell\leq n$ and $c\in \bbH$ are as in the
assumptions there. Let us consider the following three cases.
\medskip

\noindent {\sc Case 1:}\quad $u\subseteq w^{p_\xi}$.\\
Then for each $(\alpha,\beta)\in u^{\langle 2\rangle}$ we have
$h(\alpha,\beta)<M^{p_\xi}$, so this also holds for all $(\gamma,\delta)\in
(u')^{\langle 2\rangle}$. Consequently, either $u'\subseteq w^{p_\xi}$ or
$u'\subseteq w^{p_\zeta}$.  
\medskip

If $u'\subseteq w^{p_\xi}$, then let $\ell'=\min(\ell,n^{p_\xi})$ and
consider $u,u',\pi,\ell'$. Using clause $(\boxtimes)_7$ for $p_\xi$ we
immediately obtain the desired conclusion.  \medskip

If $u'\subseteq w^{p_\zeta}$, then we let $\ell'=\min(\ell,n^{p_\xi})$ and 
we consider $u,\pi^*[u'],\ell'$ and $\pi^*\circ\pi$ (where, remember,
$\pi^*:w^{p_\zeta}\longrightarrow w^{p_\xi}$ is the order
isomorphism). By $(*)_1+(*)_2$, clause $(\boxtimes)_7$ for $p_\xi$
applies to them and we get     
\begin{itemize}
\item $\rksp(u)=\rksp\big(\pi^*[u']\big)$, $\bj(u)=\bj\big(
  \pi^*[u']\big)$, $\bk(u)=\bk\big( \pi^*[u']\big)$ and 
\item for $\alpha\in u$,\quad $|\alpha\cap u|=\bk(u)\ \Leftrightarrow \ 
\big |(\pi^*\circ\pi)(\alpha)\cap \pi^*[u']\big |=\bk(u)$.
\end{itemize}
Now $(*)_1$, $(*)_2$ immediately imply the desired conclusion.  
\medskip

\noindent {\sc Case 2:}\quad $u\subseteq w^{p_\zeta}$.\\
Same as the previous case, just interchanging $\xi$ and $\zeta$.
\medskip

\noindent {\sc Case 3:}\quad $u\setminus w^{p_\xi}\neq \emptyset
\neq u\setminus w^{p_\zeta}$.\\
Choose $\alpha\in u\setminus w^{p_\xi}$ and $\beta\in u\setminus
w^{p_\zeta}$. Then $h(\alpha,\beta)\geq M^{p_\xi}$ and therefore 
$n-2=r_{h(\alpha,\beta)}\leq \ell$.  

We will argue that $\pi$ is the identity on $u$ and $u=u'$ (so the needed
assertion is immediate). Suppose towards contradiction that we got a $\gamma\in u$ 
such that $\pi(\gamma)\neq \gamma$. Since $|u|\geq 4$ we may also pick
$\gamma'\in u$ such that $\{\gamma,\pi(\gamma)\}\cap
\{\gamma',\pi(\gamma')\}=\emptyset$. Now we consider two subcases determined
by the  property of $c\in \bbH$.   

Suppose $\big(U_\delta(\ell)+c\big)\cap U_{\pi(\delta)}(\ell)\neq \emptyset$
for all $\delta\in u$. Then for some $b\in U_\gamma(\ell)$, $b'\in
U_{\pi(\gamma)}(\ell)$, $b''\in U_{\gamma'}(\ell)$ and $b'''\in
U_{\pi(\gamma')}(\ell)$ we have $b'-b=c=b'''-b''$. However, this (and the 
choice of $\gamma$ and $\gamma'$) gives immediate contradiction with
$\langle U_\delta(\ell): \delta\in w\rangle$ being a good qif  (remember
$\ell\geq n-2$).    

Assume now that $\big(c-U_\delta(\ell)\big)\cap U_{\pi(\delta)}(\ell)\neq 
\emptyset$ for all $\delta\in u$. Then for some $b\in U_\gamma(\ell)$,
$b'\in U_{\pi(\gamma)}(\ell)$, $b''\in U_{\gamma'}(\ell)$ and $b'''\in
U_{\pi(\gamma')}(\ell)$ we have $b'+b=c=b'''+b''$, getting immediate
contradiction with $\langle U_\delta(\ell): \delta\in w\rangle$ being a good 
qif.
\medskip

Now, concerning  $(\boxtimes)_8$, suppose that $u\subseteq w$, $\ell\leq n$ 
and $\alpha\in u$ are such that 
\begin{itemize}
\item $|\alpha\cap u|=\bk(u)$ and $\rksp(u)=-1$ and 
\item $r_{h(\beta,\beta')}\leq \ell$ and $U_\beta(\ell)\cap U_{\beta'}(\ell)
  = \emptyset$ for all $(\beta,\beta')\in u^{\langle 2\rangle}$. 
\end{itemize}
We want to argue that there is no $\alpha'\in w$ such that 
\begin{enumerate}
\item[$(\maltese)^{\alpha'}$] $\alpha'\notin u$, \quad $h(\alpha,\beta)=
  h(\alpha', \beta)$ for all $\beta\in u\setminus\{\alpha\}$, and
  $U_\alpha(\ell)=U_{\alpha'}(\ell)$. 
\end{enumerate}
This is immediate if $\ell\geq n-2$, so let us assume $\ell\leq
n^{p_\zeta}$. Then we must also have $r_{h(\beta,\beta')}\leq n^{p_\zeta}$ for all
$(\beta,\beta')\in u^{\langle 2\rangle}$, so either $u\subseteq u^{p_\xi}$
or $u\subseteq u^{p_\zeta}$. By the symmetry, we may assume that $u\subseteq 
u^{p_\xi}$.  

If $u\subseteq w^{p_\xi}\cap w^{p_\zeta}$ then we may first use
$(\boxtimes)_8$ for $p_\xi$ to assert that there is no $\alpha'\in
w^{p_\xi}$ satisfying $(\maltese)^{\alpha'}$ and then in the same manner
argue that  no $\alpha'\in w^{p_\zeta}$ satisfies $(\maltese)^{\alpha'}$. 

If $u\subseteq w^{p_\xi}$ but $u\setminus w^{p_\zeta}\neq \emptyset$ and
$\alpha\in w^{p_\xi}\cap w^{p_\zeta}$, then $(\boxtimes)_8$ for $p_\xi$
implies there is no $\alpha'\in w^{p_\xi}$ satisfying
$(\maltese)^{\alpha'}$. Also if $\alpha'\in w^{p_\zeta}\setminus w^{p_\xi}$
then for $\beta\in  u\setminus w^{p_\zeta}$ we have $h(\alpha,\beta)<
M^{p_\xi}\leq h(\alpha',\beta)$, so $(\maltese)^{\alpha'}$ fails then
too. 

Thus we are left only with the possibility that $\alpha\in
u^{p_\xi}\setminus u^{p_\zeta}$. Like before, $(\boxtimes)_8$ for
$p_\xi$ implies there is no $\alpha'\in w^{p_\xi}$ satisfying
$(\maltese)^{\alpha'}$. So suppose now $\alpha'\in w^{p_\zeta}\setminus
w^{p_\xi}$. By $(*)_3$ we know that $(u\setminus
\{\alpha\})\setminus u^{p_\zeta}\neq \emptyset$, so let  $\beta\in
u\setminus u^{p_\zeta}$, $\beta\neq\alpha$. Then we have $h(\alpha,\beta)<
M^{p_\xi}\leq h(\alpha',\beta)$, so $(\maltese)^{\alpha'}$ fails. The
proof of $(\boxtimes)_8$ is complete now. 
\end{proof}

\begin{lemma}
\label{genset}
For each $(\alpha,\beta)\in \lambda^{\langle 2\rangle}$ and
  $i<k$,
\[  \begin{array}{ll}
\forces_\bbP& \mbox{`` the sets }\\
&\displaystyle \bigcap \big\{U^p_\alpha(n^p): p\in \name{G}_\bbP\
    \wedge\ \alpha\in w^p\big\}\quad\mbox{ and }\quad 
\bigcap \big\{W^p_{i,\alpha,\beta}: p\in \name{G}_\bbP\
    \wedge\ \alpha,\beta\in w^p\big\}\\
&\mbox{ have exactly one element each. ''}
  \end{array}\]
\end{lemma}

\begin{proof}
 Follows from Lemma \ref{dense}(2)(ii), (3). 
\end{proof}

\begin{definition}
\label{namesdef}
\begin{enumerate}
\item For $(\alpha,\beta)\in\lambda^{\langle 2\rangle}$ and $i<k$ let
  $\name{\eta}_\alpha$, $\name{\nu}_{i,\alpha,\beta}$ and
  $\name{h}_{\alpha,\beta}$ be $\bbP$--names such that 
\[  \begin{array}{ll}
\forces_\bbP&\mbox{`` }\displaystyle \{\name{\eta}_\alpha\}= \bigcap
              \big\{U^p_\alpha(n^p): p\in \name{G}_\bbP\ \wedge\
              \alpha\in w^p\big\},\\
&\displaystyle\quad \{\name{\nu}_{i,\alpha,\beta}\}=\bigcap \big\{W^p_{i,\alpha,\beta}:
  p\in \name{G}_\bbP\ \wedge\ \alpha,\beta\in w^p\big\}\\
&\displaystyle\quad\name{h}_{\alpha,\beta}=h^p(\alpha,\beta)\mbox{ for some
  (all) $p\in\name{G}_\bbP$ such that $\alpha,\beta\in w^p$. ''}
  \end{array}\]
\item For $m<\omega$ let $\name{\bF}_m$ be a $\bbP$--name such that 
\[\forces_\bbP \mbox{`` }\name{\bF}_m=\bigcap \big\{ F(p,m): p\in
  \name{G}_\bbP \  \wedge \ m<M^p\big\}.\mbox{ ''}\] 
  (Remember $F(p,m)$ was defined in Definition \ref{fordef}{\bf (B)}.)
\end{enumerate}
\end{definition}

\begin{lemma}
\label{baspropnames}
  \begin{enumerate}
\item For each $m<\omega$, $\forces_\bbP$`` $\name{\bF}_m$ is a
  closed subset of $\bbH$. '' 
\item For $i<k$ and $(\alpha,\beta)\in \lambda^{\langle 2\rangle}$ we
    have 
\[\forces_\bbP\mbox{`` }
  \name{\eta}_\alpha,\name{\nu}_{i,\alpha,\beta}\in \bbH,\quad
  \name{h}_{\alpha,\beta}<\omega,\quad \name{\nu}_{i,\alpha,\beta}=
  \name{\nu}_{i,\beta, \alpha}\ \mbox{ and }\  \name{\eta}_\alpha +
  \name{\nu}_{i,\alpha,\beta}\in
  \name{\bF}_{\name{h}_{\alpha,\beta}} .\mbox{ ''}\]
\item $\forces_\bbP$`` $\langle\name{\eta}_\alpha,
  \name{\nu}_{i,\alpha,\beta}: i<k,\ \alpha<\beta<\lambda\rangle$ is
  quasi independent (so they are also distinct) .''
\item $\forces_\bbP$`` $\Big|\big(-\name{\eta}_\alpha+
  \bigcup\limits_{m<\omega} \name{\bF}_m\big) \cap
  \big(-\name{\eta}_\beta+ \bigcup\limits_{m<\omega} \name{\bF}_m\big)
  \Big|\geq k$. ''
  \end{enumerate}
\end{lemma}

\begin{proof}
Should be clear (remember Lemma \ref{dense}).  
\end{proof}

\begin{lemma}
  \label{addin}
  Let $p=(w,M,\bar{r},n,\bar{\Upsilon},\bar{V},h)\in D^2_1\subseteq
  \bbP$ (cf. \ref{dense}(iii)) and 
  $a_\ell,b_\ell\in\bbH$ and $U_\ell,W_\ell\in\cU$ (for $\ell<4$) be such
  that the following conditions are satisfied.
\begin{enumerate}
\item[$(\circledast)_1$] $U_\ell\in \{U_\alpha(n):\alpha\in w\}$, $W_\ell\in
  \{W_{i, \alpha,\beta}:i<k,\ (\alpha,\beta)\in w^{\langle 2\rangle}\}$ (for
  $\ell<4$).
\item[$(\circledast)_2$]
  \begin{itemize}
  \item $(U_0+W_0)\cap (U_1+W_1)=\emptyset$,
  \item $(U_1+W_1)\cap (U_3+W_3)=\emptyset$,
  \item $(U_2+W_2)\cap (U_3+W_3)=\emptyset$,
  \item $(U_0+W_0)\cap (U_2+W_2)=\emptyset$.   
  \end{itemize}
\item[$(\circledast)_3$] $a_\ell\in U_\ell$ and $b_\ell\in W_\ell$ and
  $a_\ell+b_\ell \in \bigcup\limits_{m<M} F(p,m)$ for $\ell<4$.
\item[$(\circledast)_4$] $(a_0+b_0)-(a_1+b_1)=(a_2+b_2)-(a_3+b_3)$.
\end{enumerate}
Then for some $(\alpha,\beta)\in w^{\langle 2\rangle}$ and distinct $i,j<k$
we have
\begin{enumerate}
\item[either] $U_0=U_2=U_\alpha(n)$, $U_1=U_3=U_\beta(n)$, 
  $W_0=W_1=W_{i,\alpha,\beta}$, and $W_2=W_3=W_{j,\alpha,\beta}$,
\item[or] $U_0=U_1=U_\alpha(n)$, $U_2=U_3=U_\beta(n)$, 
  $W_0=W_2=W_{i,\alpha,\beta}$, and $W_1=W_3=W_{j,\alpha,\beta}$. 
\end{enumerate}
\end{lemma}

\begin{proof}
  For $\ell<4$ let $U_\ell^-$ and $V_\ell$ be such that
  \begin{itemize}
  \item if $U_\ell=U_\alpha(n)$ then $U_\ell^-=U_\alpha(n-1)$,
  \item if $W_\ell=W_{i,\alpha,\beta}$ then $V_\ell=V_{i,\alpha,\beta}$.    
  \end{itemize}
  Also, let
  \[{\rm LHS}_a=a_0-a_1-a_2+a_3,\quad {\rm LHS}_b=b_0-b_1-b_2+b_3,\quad 
    \mbox{ and }\quad {\rm LHS}={\rm LHS}_a+{\rm LHS}_b=0.\]
Put $\cU^*=\{U_0,U_1,U_2,U_3\}$, $\cW^*=\{W_0,W_1,W_2,W_3\}$,
$\cU^*_-=\{U_0^-,U_1^-,U_2^-,U_3^-\}$, and  $\cV^*=\{V_0,V_1,V_2,V_3\}$.
\medskip

\begin{enumerate}
\item[(a)] $|\cU^*|>1$.
\end{enumerate}
Why? If not, then $U_0=U_1=U_2=U_3$ and by the assumption $(\circledast)_2$
of the Lemma we have $\{W_0,W_3\}\cap \{W_1,W_2\}=\emptyset$. By
\ref{fordef}(A)$(\boxtimes)_4$, the latter also means that $\{V_0,V_3\}\cap
\{V_1,V_2\}=\emptyset$. Now,
\[{\rm LHS}=\big((a_0-a_1)+b_0\big)+\big((a_3-a_2)+b_3\big)-b_1-b_2,\]
and using \ref{fordef}(A)$(\boxtimes)_4$(b) we have $(a_0-a_1)+b_0\in V_0$,
$(a_3-a_2)+b_3\in V_3$, $b_1\in V_1$ and $b_2\in V_2$. Since
$\{V_0,V_3\}\cap \{V_1,V_2\}=\emptyset$ we see that ${\rm LHS}$ is a
nontrivial $(2,4)$--combination from $\cV^*$, so it cannot be 0,
contradicting assumption $(\circledast)_4$ of the Lemma. 
\medskip

\begin{enumerate}
\item[(b)] $|\cW^*|>1$.
\end{enumerate}
Why? Fully parallel to (a).
\medskip

\begin{enumerate}
\item[(c)] If for some $W$ we have $|\{\ell<4:W_\ell=W\}|=3$, then ${\rm LHS}_b$
  is a nontrivial $(2,4)$--combination from $\cV^*$.
\end{enumerate}
Why? Suppose $W_0=W_1=W_2\neq W_3$. Then, by
\ref{fordef}(A)$(\boxtimes)_4$(b), we have $(b_1-b_0)+b_2\in V_2$ and
$b_3\in V_3$. Hence ${\rm LHS}_b= -\big( (b_1-b_0)+b_2\big)+b_3\in -V_2+V_3$
and $V_2\neq V_3$. 

Suppose $W_0=W_1=W_3\neq W_2$. Then, by \ref{fordef}(A)$(\boxtimes)_4$(b), 
we have $(b_0-b_1)+b_3\in V_3$ and $b_2\in V_2$, so ${\rm LHS}_b=
(b_0-b_1)+b_3-b_2\in V_3-V_2$ and $V_2\neq V_3$.

The other cases are fully parallel. 
\medskip

\begin{enumerate}
\item[(d)] If for some $U$ we have $|\{\ell<4:U_\ell=U\}|=3$, then ${\rm LHS}_a$
  is a nontrivial $(2,4)$--combination from $\cU^*_-$.
\end{enumerate}
Why? Same argument as for (c), just using $U_\ell$ instead of $W_\ell$. 

\begin{enumerate}
\item[(e)] For every $W$ we have $|\{\ell<4:W_\ell=W\}|<3$.
\end{enumerate}
Why? We already know that $|\{\ell<4:W_\ell=W\}|<4$ (by (b)), so suppose
$\{\ell<4:W_\ell=W\}$ has exactly 3 elements. It follows from (c) that then
${\rm LHS}_b$ is a nontrivial $(2,4)$--combination from $\cV^*$. By (a) we
know that $|\cU^*|>1$. If for some $U$ we have $|\{\ell<4:U_\ell=U\}|=3$,
then we may use (d) to claim that ${\rm LHS}_a$ is a (nontrivial)
$(2,4)$--combination from $\cU^*_-$ and then ${\rm LHS}$ is a nontrivial
$(2,8)$--combination from $\cU^*_-\cup\cV^*$, contradicting
$(\circledast)_4$ (remember \ref{fordef}(A)$(\boxtimes)_4$).  

So suppose that for each $U$ we have $|\{\ell<4:U_\ell=U\}|\leq 2$. Then
${\rm LHS}_a$ is a (possibly trivial) $(2,4)$--combination from $\cU^*$ and
consequently ${\rm LHS}$ is a nontrivial $(2,8)$--combination from
$\cU^*\cup \cV^*$, so also from $\cU^*_-\cup\cV^*$, again contradicting
$(\circledast)_4$.   
\medskip

\begin{enumerate}
\item[(f)] For each $U$, $|\{\ell<4:U_\ell=U\}|<3$.
\end{enumerate}
Why? Same argument as for (e), just using (a) and (c) instead of (b) and
(d). 
\medskip

Since $p\in D^2_1$, it follows from our assumption $(\circledast)_3$ that
for each $\ell<4$, for some $\alpha=\alpha(\ell),\beta=\beta(\ell)$, and
$i=i(\ell)$ we have $U_\ell=U_\alpha(n)$ and $W_\ell=W_{i,\alpha,\beta}$. It
follows from (e)+(f) that ${\rm LHS}$ is a $(2,8)$--combination from
$\cU^*\cup\cW^*$. Necessarily it is a trivial combination (as ${\rm LHS}=0$
by $(\circledast)_4$).  Consequently ,
\begin{enumerate}
\item[$(\odot)_1$] either $U_0=U_1\neq U_2=U_3$, or $U_0=U_2\neq U_1=U_3$,
and 
\item[$(\odot)_2$] either $W_0=W_1\neq W_2=W_3$, or $W_0=W_2\neq W_1=W_3$. 
\end{enumerate}
Suppose $U_0=U_1\neq U_2=U_3$. Then by $(\circledast)_2$ we must have
$W_0\neq W_1$, $W_2\neq W_3$ and by $(\odot)_2$ we get $W_0=W_2$ and
$W_1=W_3$. Thus for some $(\alpha,\beta)\in w^{\langle 2\rangle}$ and
$i,j<k$, $i\neq j$, we have
\[U_0=U_1=U_\alpha(n),\quad U_2=U_3=U_\beta(n),\quad
  W_0=W_2=W_{i,\alpha,\beta}, \quad  W_1=W_3=W_{j,\alpha,\beta}.\]
Suppose now that $U_0=U_2$ and $U_1=U_3$. By $(\circledast)_2$ we must have
then $W_0\neq W_2$ and $W_1\neq W_3$. Therefore, by $(\odot)_2$, we may
conclude that $W_0=W_1$ and $W_2=W_3$. Consequently, for some
$(\alpha,\beta)\in w^{\langle 2\rangle}$ and distinct $i,j<k$ we have
\[U_0=U_2=U_\alpha(n),\quad U_1=U_3=U_\beta(n),\quad
  W_0=W_1=W_{i,\alpha,\beta}, \quad  W_2=W_3=W_{j,\alpha,\beta}.\]
\end{proof}

\begin{lemma}
  \label{getdiff}
  Let $p=(w,M,\bar{r},n,\bar{\Upsilon},\bar{V},h)\in D^2_1$ and 
  $\bX\subseteq \bbH$, $|\bX|\geq 5$. Suppose that $a_i(x,y),b_i(x,y),
  U_i(x,y)$ and $W_i(x,y)$ for $x,y\in \bX$, $x\neq y$ and $i<k$ satisfy the
  following demands (i)--(iv) (for all $x\neq y$, $i\neq i'$).
\begin{enumerate}
\item[(i)] $U_i(x,y)\in \{U_\alpha(n):\alpha\in w\}$, $W_i(x,y)\in
  \{W_{j, \alpha,\beta}:j<k,\ (\alpha,\beta)\in w^{\langle 2\rangle}\}$.
\item[(ii)]
  \begin{itemize}
  \item $\big(U_i(x,y)+W_i(x,y) \big)\cap \big (U_i(y,x)+W_i(y,x)
    \big)=\emptyset$, 
  \item $\big (U_i(x,y)+W_i(x,y) \big)\cap \big (U_{i'}(x,y)+W_{i'}(x,y)
    \big)=\emptyset$.    
  \end{itemize}
\item[(iii)] $a_i(x,y)\in U_i(x,y)$ and $b_i(x,y)\in W_i(x,y)$, and

  $a_i(x,y)+b_i(x,y)\in \bigcup\limits_{m<M} F(p,m)$.
\item[(iv)] $x-y=\big (a_i(x,y)+b_i(x,y) \big)-\big (a_i(y,x)+b_i(y,x)
  \big)$. 
\end{enumerate}
Then
\begin{enumerate}
\item $\bX-\bX\subseteq \bigcup\big\{U_\alpha(n-2)-U_\beta(n-2): \alpha,\beta \in 
  w\big\}$.
\item If $(x,y)\in \bX^{\langle 2\rangle}$ and $x-y\in
  U_\alpha(n-2)-U_\beta(n-2)$, $\alpha,\beta\in w$, then $\alpha\neq\beta$
  and for each $i<k$ we have $a_i(x,y)+b_i(x,y), a_i(y,x)+b_i(y,x)\in
  F(p,h(\alpha,\beta))$. 
\end{enumerate}
\end{lemma}

\begin{proof}
(1)\quad  Fix $x,y\in\bX$, $x\neq y$, for a moment.

Let $i\ne i'$, $i,i'<k$. We may apply Lemma \ref{addin} for $U_i(x,y)$,
$W_i(x,y)$, $U_i(y,x)$, $W_i(y,x)$, $a_i(x,y)$, $b_i(x,y)$, $a_i(y,x)$,
$b_i(y,x)$ here as $U_0,W_0,U_1,W_1,a_0,b_0,a_1,b_1$ there and for similar
objects with $i'$ in place of $i$ as $U_2,W_2,U_3,W_3,a_2,b_2,a_3,b_3$
there. This will produce distinct $\alpha=\alpha(x,y,i,i'), \beta=
\beta(x,y,i,i')\in w$ and distinct $j=j(x,y,i,i'),j'=j'(x,y,i,i')<k$ such
that 
\begin{enumerate}
\item[either (A)] $U_i(x,y)=U_{i'}(x,y)=U_\alpha(n)$,
  $U_i(y,x)=U_{i'}(y,x)=U_\beta(n)$,

  $W_i(x,y)=W_i(y,x)=W_{j,\alpha,\beta}$,
  $W_{i'}(x,y)=W_{i'}(y,x)=W_{j',\alpha,\beta}$,
\item[or (B)] $U_i(x,y)=U_i(y,x)=U_\alpha(n)$,
  $U_{i'}(x,y)=U_{i'}(y,x)=U_\beta(n)$,

  $W_i(x,y)=W_{i'}(x,y)=W_{j,\alpha,\beta}$,
  $W_i(y,x)=W_{i'}(y,x)=W_{j',\alpha,\beta}$.
\end{enumerate}
Note that if for some $i\neq i'$, $i,i'<k$, the possibility (A) above holds,
then it holds for all $i,i'<k$ and 
\[\begin{array}{r}
  x-y=\big(a_i(x,y)+b_i(x,y)\big)-\big(a_i(y,x)+b_i(y,x)\big)=\\
    \Big(a_i(x,y)+ \big(b_i(x,y)-b_i(y,x)\big)\Big)-a_i(y,x)
    \end{array}\]
and $a_i(x,y)+ \big(b_i(x,y)-b_i(y,x)\big)\in
U_\alpha(n)+\big(W_{j,\alpha,\beta} -W_{j,\alpha,\beta}\big)\subseteq
U_\alpha(n-1)\subseteq U_\alpha(n-2)$. Hence $x-y\in U_\alpha(n-2)
-U_\beta(n-2)$. 

Now unfix $x,y$. By what we have said, the first assertion of the Lemma will 
follow once we show that 
\begin{enumerate}
\item[$(\heartsuit)$] for all $x,y\in \bX$, $x\neq y$, there are $i\neq i'$
  such that possibility (A) above holds for them. 
\end{enumerate}
Here the argument breaks into two cases: $k\geq 3$ and $k=2$, with the former
being somewhat simpler.
\bigskip

{\sc Case}\quad $k\geq 3$.\\
Let $x,y\in\bX$, $x\neq y$. Suppose towards contradiction that in the
previous considerations both for $x,y,0,1$ and for $x,y,1,2$ the second
(i.e., (B)) possibility takes place. This gives us $\alpha,\beta,j,j'$ such
that $\alpha\neq\beta$, $j\neq j'$ and 
\begin{enumerate}
\item[$(*)_1$] $U_0(x,y)=U_0(y,x)=U_\alpha(n)$,
\item[$(*)_2$] $U_1(x,y)=U_1(y,x)=U_\beta(n)$,
\item[$(*)_3$] $W_0(x,y)=W_1(x,y)=W_{j,\alpha,\beta}$, 
\item[$(*)_4$] $W_0(y,x)=W_1(y,x)=W_{j',\alpha,\beta}$,
\end{enumerate}
and we also get $\gamma,\delta,\ell,\ell'$ such that  $\gamma\neq \delta$
and $\ell\neq \ell'$ and 
\begin{enumerate}
\item[$(*)_5$] $U_1(x,y)=U_1(y,x)=U_\gamma(n)$,
\item[$(*)_6$] $U_2(x,y)=U_2(y,x)=U_\delta(n)$,
\item[$(*)_7$] $W_1(x,y)=W_2(x,y)=W_{\ell,\gamma,\delta}$, 
\item[$(*)_8$] $W_1(y,x)=W_2(y,x)=W_{\ell',\gamma,\delta}$.  
\end{enumerate}
It follows from $(*)_2+(*)_5$ that $\gamma=\beta$ and from $(*)_3+(*)_7$ we
have $\ell=j$ and $\delta=\alpha$. Finally, $(*)_4+(*)_8$ imply
$\ell'=j'$. Consequently,
\[U_0(x,y)=U_2(x,y),\ U_0(y,x)=U_2(y,x),\ W_0(x,y)=W_2(x,y), \ 
  W_0(y,x)= W_2(y,x),\]
contradicting assumption (ii). 
\bigskip

{\sc Case}\quad $k=2$.\\
We will argue that $(\heartsuit)$ holds true in this case as well. First,
however, we have to establish some auxiliary facts.
\medskip

For each $x,y\in\bX$, $x\neq y$, we may choose $\alpha=\alpha(x,y)$, 
$\beta=\beta(x,y)$ and $j=j(x,y)$ such that $\alpha\neq\beta$ and 
\begin{enumerate}
\item[either $(A)^{\alpha,\beta,j}_{x,y}$] $U_0(x,y)=U_1(x,y)=U_\alpha(n)$,  
$U_0(y,x)=U_1(y,x)=U_\beta(n)$,

$W_0(x,y)=W_0(y,x)=W_{j,\alpha,\beta}$,
$W_1(x,y)=W_1(y,x)=W_{1-j,\alpha,\beta}$,
\item[or $(B)^{\alpha,\beta,j}_{x,y}$] $U_0(x,y)=U_0(y,x)=U_\alpha(n)$,   
$U_1(x,y)=U_1(y,x)=U_\beta(n)$,

$W_0(x,y)=W_1(x,y)=W_{j,\alpha,\beta}$,
$W_0(y,x)=W_1(y,x)=W_{1-j,\alpha,\beta}$, 
\end{enumerate}
Note that for each $(x,y)\in w^{\langle 2\rangle}$, either there are
$\alpha,\beta,j$ such that $(A)^{\alpha,\beta,j}_{x,y}$ holds true or there
are $\alpha,\beta,j$ such that $(B)^{\alpha,\beta,j}_{x,y}$ is true, but not
both. Also, remembering \ref{fordef}(A)$(\boxtimes)_4$(b), 
\begin{enumerate}
\item [$(\triangle)_1$]  if $(A)^{\alpha,\beta,j}_{x,y}$ holds, then $x-y\in 
U_\alpha(n-1) -U_\beta(n-1)$ and if $(B)^{\alpha,\beta,j}_{x,y}$ is
satisfied, then $x-y\in V_{j,\alpha,\beta} -V_{1-j,\alpha,\beta}$. 
\end{enumerate}
Define functions $\chi:\bX^{\langle 2\rangle}\longrightarrow 2$ and $\Theta:
\bX^{\langle 2\rangle}\longrightarrow [w]^2\times 2$ as follows. Assuming
$(x,y)\in \bX^{\langle 2\rangle}$,
\begin{itemize}
\item if for some $\alpha,\beta,j$ the demand 
$(A)^{\alpha,\beta,j}_{x,y}$ holds, then $\chi(x,y)=1$ and 
$\Theta(x,y)=(\{\alpha,\beta\},j)$, 
\item if for some $\alpha,\beta,j$ the demand $(B)^{\alpha,\beta,j}_{x,y}$
  is satisfied, then $\chi(x,y)=0$ and $\Theta(x,y)=(\{\alpha,\beta\},j)$. 
\end{itemize}
\medskip

Our goal is to show that the function $\chi$ never takes value 0 (as this 
will imply that the assertion $(\heartsuit)$ holds true). Note that 
\begin{enumerate}
\item [$(\triangle)_2$] if $\chi(x,y)=0$ and  $\Theta(x,y)=(\{\alpha,
  \beta\},j)$, then $\chi(y,x)=0$ and $\Theta(y,x)=(\{\alpha,\beta\},1-j)$,
  so also  $\Theta(x,y)\neq \Theta(y,x)$.  
\end{enumerate}
Also, 
\begin{enumerate}
\item [$(\triangle)_3$] if $x,y,z\in \bX$ are pairwise distinct and
  $\chi(x,y)=\chi(y,z)=1$, then $\chi(x,z)=1$. 
\end{enumerate}
Why? Assume $\chi(x,z)=0$. Then, by $(\triangle)_1$, for some $j,\xi,\zeta$
we have $x-z\in V_{j,\xi,\zeta} -V_{1-j,\xi,\zeta}$. However, $x-y\in
U_\alpha(n-1)-U_\beta(n-1)$ and $y-z\in U_\gamma(n-1)-U_\delta(n-1)$ (for
some $\alpha\neq \beta$ and $\gamma\neq \delta$), so
\[x-z\in U_\alpha(n-1)-U_\beta(n-1)+U_\gamma(n-1)-U_\delta(n-1).\]
Thus for some $a\in U_\alpha(n-1)$, $b\in U_\beta(n-1)$, $c\in
U_\gamma(n-1)$, $d\in U_\delta(n-1)$, $e\in V_{j,\xi,\zeta}$, and $f\in
V_{1-j,\xi,\zeta}$ we have $a-b+c-d+f-e=0$. The left hand side of this
equation represents a nontrivial (2,8)--combination from  $\langle
U_\zeta(n-1):\zeta\in w\rangle\conc \langle V_{0,\zeta,\zeta'},
V_{1,\zeta,\zeta'} :(\zeta,\zeta')\in w^{\langle   2\rangle}\rangle$
(remember $\alpha\neq \beta$, $\gamma\neq \delta$, $\xi\neq\zeta$), a
contradiction. 
\medskip

\begin{enumerate}
\item [$(\triangle)_4$] If $x,y,z\in \bX$ are pairwise distinct and
  $\chi(x,y)=\chi(y,z)=0$, then $\Theta(x,y)=\Theta(y,z)=\Theta(z,x)$ and
  $\chi(x,z)=0$. 
\end{enumerate}
Why? Let $\Theta(x,y)=(\{\alpha,\beta\},i)$, $\Theta(y,z)=(\{\gamma,
\delta\},j)$, and $\Theta(x,z)=(\{\xi,\zeta\},\ell)$. If
$\{\alpha,\beta\}\neq \{\gamma,\delta\}$, then 
\[x-z=(x-y)+(y-z)\in V_{i,\alpha,\beta}-V_{1-i,\alpha,\beta}+
  V_{j,\gamma,\delta} -V_{1-j,\gamma,\delta}\]
and $\{V_{0,\alpha,\beta},V_{1,\alpha,\beta}\}\cap \{ V_{0,\gamma,\delta},
V_{1,\gamma,\delta}\}=\emptyset$. Since either $x-z\in V_{\ell,\xi,\zeta}-
V_{1-\ell,\xi,\zeta}$ or $x-z\in \Big(U_\xi(n-1)-U_\zeta(n-1)\Big)\cup \Big(
U_\zeta(n-1)-U_\xi (n-1)\Big)$, we easily get that some nontrivial
(2,8)--combination from  $\langle U_\zeta(n-1):\zeta\in
w\rangle\conc \langle V_{0,\zeta,\zeta'},
V_{1,\zeta,\zeta'}: (\zeta,\zeta')\in w^{\langle
  2\rangle}\rangle$ equals 0, a contradiction. Consequently,
$\{\alpha,\beta\}=\{\gamma,\delta\}$, i.e.,
$\Theta(x,y)=(\{\alpha,\beta\},i)$ and  $\Theta(y,z)=(\{\alpha,\beta\}, j)$ 
for some $\alpha,\beta,i,j$.

If $i\neq j$ then $x-z\in V_{i,\alpha,\beta}-V_{1-i,\alpha,\beta}+
V_{1-i,\alpha,\beta}-V_{i,\alpha,\beta}$. But also either $x-z\in U_\xi(n-1)
-U_\zeta(n-1)$, or $x-z\in U_\zeta(n-1)-U_\xi(n-1)$, or $x-z\in
V_{\ell,\xi,\zeta}-V_{1-\ell,\xi,\zeta}$. In the first case we get 
\[\begin{array}{r}
0\in \Big(\big(V_{i,\alpha,\beta}-V_{i,\alpha,\beta}\big)+ U_\xi(n-1)\Big)
  - \Big(\big(V_{1-i,\alpha,\beta}-V_{1-i,\alpha,\beta}\big)+
  U_\zeta(n-1)\Big) \subseteq \quad \\
U_\xi(n-2)-U_\zeta(n-2),
\end{array}\] 
and symmetrically in the second case. In the last case we have
\[\begin{array}{r}
0\in \Big(\big(V_{i,\alpha,\beta}-V_{i,\alpha,\beta}\big)+
V_{\ell,\xi,\zeta}\Big) - \Big(\big(V_{1-i,\alpha,\beta}-
V_{1-i,\alpha,\beta} \big)+  V_{1-\ell,\xi,\zeta} \Big) \subseteq \quad\\  
Q_{\ell,\xi,\zeta}-Q_{1-\ell,\xi,\zeta}.
\end{array}\] 
In any case this gives a contradiction with
\ref{fordef}(A)$(\boxtimes)_4$. Consequently $i=j$ and
$\Theta(x,y)=\Theta(y,z)=(\{\alpha,\beta\},i)$. 

By considerations as above we see that necessarily $\chi(x,z)=0$ and
$\Theta(x,z)=(\{\alpha,\beta\},\ell)$. If $\ell=i$, then
\[x-z\in V_{i,\alpha,\beta}-V_{1-i,\alpha,\beta}\quad\mbox{ and }\quad x-z\in 
  V_{i,\alpha,\beta} - V_{1-i,\alpha,\beta}+V_{i,\alpha,\beta}
  -V_{1-i,\alpha,\beta}.\] 
Hence
\[0\in\Big(\big(V_{i,\alpha,\beta}-V_{i,\alpha,\beta}\big)+
  V_{i,\alpha,\beta} \Big) -\Big(\big(V_{1-i,\alpha,\beta} -
  V_{1-i,\alpha,\beta} \big)+ V_{1-i,\alpha,\beta}\Big) \subseteq
  Q_{i,\alpha,\beta} -Q_{1-i,\alpha,\beta},\]
a contradiction.

Consequently, $\ell=1-i$ and $\Theta(z,x)=(\{\alpha,\beta\},i)=\Theta(x,y)$
(and $\chi(x,z)=0$).
\bigskip\medskip

Now, suppose towards contradiction that $(\heartsuit)$ is not true and
$x,y\in\bX$ are such that  $x\neq y$ and $\chi(x,y)=0$. Let $z\in
\bX\setminus\{x,y\}$. We cannot have $\chi(x,z)=\chi(y,z)=1$ (as then
$(\triangle)_3$ would give a contradiction with $\chi(x,y)=0$). So one of
them is 0, and then $(\triangle)_4$ implies that the other is 0 as well and
\[\chi(x,y)=\chi(y,z)=\chi(x,z)=0\quad \mbox{ and }\quad \Theta(x,y)=
  \Theta(y,z)=\Theta(z,x).\]
Taking $t\in \bX\setminus \{x,y,z\}$ by similar considerations we obtain 
\[\chi(x,t)=\chi(y,t)=0\quad \mbox{ and }\quad \Theta(x,y)=
  \Theta(y,t)=\Theta(t,x).\]
Now consider $x,z,t$: since $\chi(x,z)=\chi(x,t)=0$ we may use
$(\triangle)_4$ to conclude that
\[\chi(z,t)=0\quad \mbox{ and }\quad \Theta(x,z)=\Theta(z,t)=\Theta(t,x).\]
But we have established already that $\Theta(t,x)=\Theta(x,y)=\Theta(z,x)$,
a contradiction (remember $(\triangle)_2$). The proof of Lemma
\ref{getdiff}(1) is complete now.
\medskip

\noindent (2)\quad Suppose $(x,y)\in \bX^{\langle 2\rangle}$.  In the
previous part we showed that for all $i<i'<k$ possibility (A) holds true.
More precisely, there are distinct $\alpha,\beta\in w$ such that for all
$i<k$ for some $j<k$ we have $a_i(x,y)\in U_\alpha(n)$ and $a_i(y,x)\in
U_\beta(n)$, and  $b_i(x,y), b_i(y,x)\in W_{j,\alpha,\beta}$. Then also
\begin{itemize}
\item $a_i(x,y)+b_i(x,y)\in U_\alpha(n)+W_{j,\alpha,\beta}\subseteq 
  F(p,h(\alpha,\beta))$, 
\item $a_i(y,x)+b_i(y,x)\in U_\beta(n)+W_{j,\alpha,\beta}\subseteq 
  F(p,h(\alpha,\beta))$. 
\end{itemize}
We also know that for these $\alpha,\beta$ we have $x-y\in U_\alpha(n-2) -
U_\beta(n-2)$. To complete the proof we note that, by Lemma \ref{abc}, for
any $\alpha',\beta'$ 
\begin{quotation}
$\big(U_\alpha(n-2)-U_\beta(n-2)\big)\cap 
\big(U_{\alpha'}(n-2)-U_{\beta'}(n-2)\big)\neq \emptyset$ implies 
$\alpha=\alpha'$ and $\beta=\beta'$.   
\end{quotation}
\end{proof}

\begin{lemma}
  \label{noperf}
\[\begin{array}{ll}
\forces_\bbP&\mbox{`` there is no perfect set $P\subseteq \bbH$ such that}\\
 &\quad  \Big(\forall x,y\in P\Big)\Big(\Big|\big(x+
   \bigcup\limits_{m<\omega} \name{\bF}_m\big) \cap  \big(y+
   \bigcup\limits_{m<\omega} \name{\bF}_m\big) \Big|\geq k\Big)\mbox{. ''}  
\end{array}\] 
\end{lemma}

\begin{proof}
Suppose towards contradiction that  $G\subseteq \bbP$ is generic over
$\bV$ and in $\bV[G]$ the following assertion holds true:
\begin{quotation}
for some perfect set $P\subseteq \bbH$ we have
  \[\Big|\big(x+ \bigcup\limits_{m<\omega} \name{\bF}_m^G\big) \cap 
  \big(y+ \bigcup\limits_{m<\omega} \name{\bF}_m^G\big) \Big|\geq k\] 
for  all $x,y\in P$.  
\end{quotation}
Then for any distinct $x,y\in P$ there are $c_0,d_0,\ldots, c_{k-1},d_{k-1}
\in \bigcup\limits_{m<\omega} \name{\bF}_m^G$ such that $c_i\neq c_j$
whenever $i\neq j$ and   $x-y=c_i-d_i$ (for all $i<k$). 

For $\bar{\ell}=\langle \ell_i:i<k\rangle \subseteq \omega$, $\bar{m}
=\langle m_i:i<k\rangle\subseteq \omega$ and $N<\omega$ let    
\[\begin{array}{l}
Z^N_{\bar{\ell},\bar{m}}= \{(x,y)\in P^2:
\mbox{there are } c_i\in \name{\bF}_{\ell_i}^G, d_i\in 
\name{\bF}_{m_i}^G \mbox{ (for $i<k$) such that}\\
\qquad\quad x-y=c_i-d_i \ \mbox{ and } \ 2^{-N}<\min\big(\rho(c_i,c_j),  
    \rho(d_i,d_j)\big)\mbox{ for all distinct }i,j<k\}.     
\end{array}\] 
By our assumption on $P$ we know that 
\begin{enumerate}
\item[$(\boxdot)_0$] for all $x,y\in P$, $x\neq y$, there are $\bar{\ell},
  \bar{m}$ and $N$ such that $(x,y)\in Z_{\bar{\ell},\bar{m}}^N$.    
\end{enumerate}
The sets $Z_{\bar{\ell},\bar{m}}^N\subseteq P^2$ are $\Sigma^1_1$, so they
have the Baire property (in $P^2$). Therefore, for every open set
$U\subseteq \bbH\times\bbH$ with $U\cap P^2\neq\emptyset$ there is a basic
open set $(d_0+\bB_{n_0})\times (d_1+\bB_{n_1})\subseteq U$ such that
$\big[(d_0+\bB_{n_0})\times (d_1+\bB_{n_1})\big]\cap P^2\neq\emptyset$ and 
\begin{itemize}
\item either $Z_{\bar{\ell},\bar{m}}^N\cap \big[(d_0+\bB_{n_0})\times
  (d_1+\bB_{n_1})\big]$ is a meager subset of $P^2$,
\item or $\big[[(d_0+\bB_{n_0})\times (d_1+\bB_{n_1})]\cap P^2\big]
  \setminus  Z_{\bar{\ell},\bar{m}}^N$ is a meager subset of $P^2$. 
\end{itemize}

Now we may choose closed nowhere dense subsets $F_j$ of $P^2$ (for
$j<\omega$) such that for each $d_0,d_1\in\bD$ and $n_0,n_1<\omega$ and 
$N,\bar{\ell},\bar{m}$ as before we have
\begin{enumerate}
\item[$(\boxdot)^a_1$] if $Z_{\bar{\ell},\bar{m}}^N\cap \big[(d_0+\bB_{n_0})
  \times  (d_1+\bB_{n_1})\big]$ is meager in $P^2$, then 
  \[Z_{\bar{\ell},\bar{m}}^N\cap \big[(d_0+\bB_{n_0})\times
  (d_1+\bB_{n_1})\big] \subseteq \bigcup\limits_{j<\omega} F_j,\]
\item[$(\boxdot)^b_1$] if $\big[[(d_0+\bB_{n_0})\times (d_1+\bB_{n_1})]\cap
  P^2\big]  \setminus  Z_{\bar{\ell},\bar{m}}^N$ is meager in $P^2$, then 
  \[\big[[(d_0+\bB_{n_0})\times (d_1+\bB_{n_1})]\cap P^2\big] \setminus
  Z_{\bar{\ell},\bar{m}}^N \subseteq \bigcup\limits_{j<\omega} F_j.\]
\end{enumerate}
Let $\langle F_i:i<\omega\rangle$ be an enumeration of all sets
$E_j(d_0,d_1,n_0,n_1,N,\bar{\ell},\bar{m})$ (for all relevant
parameters). Then $\bigcup\limits_{j<\omega} F_j$ is a meager subset of
$P^2$. Let $B^*=P^2\setminus \bigcup\limits_{j<\omega} F_j$.
\medskip

We are going to choose now a sequence $0=n^*_0=n_0 <n^*_1<n_1< n^*_2<n_2<  
n^*_3<n_3<\ldots$ and a system $\langle  d_\sigma: \sigma \in {}^\iota 2,\
\iota<\omega\rangle\subseteq\bD$ such that the following demands
$(\boxdot)_2^a$--$(\boxdot)_2^e$ are satisfied.   
\begin{enumerate}
\item[$(\boxdot)_2^a$]  If $\iota<\omega$, $\sigma,\sigma' \in {}^\iota
  2$, $\sigma\neq \sigma'$, then $(d_\sigma+ \bB_{n_\iota}) \cap P\neq
  \emptyset$ and $\rho(d_\sigma,d_{\sigma'})>2^{2-n_\iota}$, and
  $\diam(d_\sigma+\bB_{n_\iota}) <2^{-\iota}$.
\item[$(\boxdot)_2^b$] If $\iota<\omega$, $\sigma\in {}^\iota 2$, then
  $\cl\big(d_{\sigma\conc\langle 0\rangle}+\bB_{n_{\iota+1}}\big)\cup
  \cl\big(d_{\sigma\conc\langle 1\rangle}+\bB_{n_{\iota+1}}\big) \subseteq
  (d_\sigma+ \bB_{n_\iota})$. 
\item[$(\boxdot)_2^c$]  If $\iota<\omega$ and  $\sigma,\sigma' \in {}^\iota
  2$, $\sigma\neq \sigma'$, and
  \[(x,y), (x',y')\in B^*\cap \big[(d_\sigma+\bB_{n_\iota})\times
    (d_{\sigma'}+ \bB_{n_\iota})\big],\]
then for all $\bar{\ell}\subseteq n_\iota^*$, $\bar{m}\subseteq n_\iota^*$
and  $N< n_\iota^*$ we  have 
\[(x,y)\in Z_{\bar{\ell},\bar{m}}^N\ \Leftrightarrow\ (x',y')\in
Z_{\bar{\ell},\bar{m}}^N.\]     
\item[$(\boxdot)_2^d$]  If $\iota<\omega$ and  $\sigma,\sigma' \in {}^\iota
  2$, $\sigma\neq \sigma'$, and  $(x,y)\in B^*\cap
  \big[(d_\sigma+\bB_{n_\iota}) \times (d_{\sigma'}+ \bB_{n_\iota})\big]$,
  then there are $\bar{\ell}\subseteq n_\iota^*$, $\bar{m}\subseteq
  n_\iota^*$ and  $N<n_\iota^*$ such that $(x,y)\in
  Z_{\bar{\ell},\bar{m}}^N$.
\item[$(\boxdot)_2^e$]  If $\iota<\omega$ and  $\sigma,\sigma' \in {}^\iota
  2$, $\sigma\neq \sigma'$, then  $\big[(d_\sigma+\bB_{n_\iota}) \times
  (d_{\sigma'}+ \bB_{n_\iota})\big] \cap \bigcup\limits_{j<\iota} F_j
  =\emptyset$. 
\end{enumerate}
The construction is by induction on $\iota<\omega$. We start with choosing
any $d_{\langle\rangle}\in\bD$ such that $(d_{\langle\rangle}+\bB_0)\cap
P\neq\emptyset$. We also set $n_0=n_0^*=0$. Let us describe in more detail
choices for $\iota=1$ as they have all the ingredients used later. So, first
find open sets $V^\dagger,V^{\dagger \dagger}$ such that $V^\dagger\cap 
P\neq\emptyset\neq V^{\dagger \dagger}\cap P$ and $\cl(V^\dagger)\cup
\cl(V^{\dagger \dagger})\subseteq (d_{\langle\rangle}+\bB_0)$,  $\cl(V^\dagger)\cap
\cl(V^{\dagger \dagger})=\emptyset$. Let $N_0,\bar{\ell}_0,\bar{m}_0$ be
such that the set $Z^{N_0}_{\bar{\ell}_0,\bar{m}_0}\cap \big [
V^\dagger\times V^{\dagger \dagger}\big]$ is not meager in $P^2$ and let
$n^*_1$ be such that $N_0<n^*_1$, $\bar{\ell}_0\subseteq n^*_1$ and
$\bar{m}_0\subseteq n^*_1$. Now we repeatedly use the Baire property of the
sets $Z^N_{\bar{\ell},\bar{m}}$ to find open sets $V'\subseteq V^\dagger$
and $V''\subseteq V^{\dagger \dagger}$ such that $V'\cap  P\neq\emptyset\neq
V''\cap P$ and 
\begin{enumerate}
\item[(A)] $\big[(V'\times V'')\cap P^2\big]\setminus
  Z^{N_0}_{\bar{\ell}_0,\bar{m}_0}$ is meager in $P^2$ (where
  $N_0,\bar{\ell}_0,\bar{m}_0$ are the ones fixed above), and 
\item[(B)] for every $\bar{\ell}\subseteq n^*_1$, $\bar{m}\subseteq n^*_1$
  and $N<n^*$, either $\big[(V'\times V'')\cap P^2\big] \setminus
  Z^N_{\bar{\ell},\bar{m}}$ is meager in $P^2$, or $(V'\times V'')\cap
  Z^N_{\bar{\ell},\bar{m}}$ is meager in $P^2$.
\end{enumerate}
Since $F_0$ is a nowhere dense subset of $P^2$, we may find open sets
$V^*\subseteq V'$ and $V^{**}\subseteq V''$ such that $V^*\cap 
P\neq\emptyset\neq V^{**}\cap P$ and $(V^*\times V^{**})\cap F_0=\emptyset$. 
Now, after fixing some $x\in V^*\cap P$ and $y\in V^{**}\cap P$ we choose
$n>n^*_1$ so large that $\rho(x,y)>2^{3-n}$ and 
\begin{itemize}
\item ${\rm diam}_{\rho^*}(x+\bB_n)<1/2$ and ${\rm
    diam}_{\rho^*}(y+\bB_n)<1/2$, and 
\item $x+\bB_n\subseteq V^*$ and $y+\bB_n\subseteq V^{**}$.
\end{itemize}
Then we set $n_1=n+2$ and choose $d_{\langle 0\rangle}\in (x+\bB_{n_1})\cap
\bD$ and  $d_{\langle 1\rangle}\in (y+\bB_{n_1})\cap \bD$. Note that  $x\in
d_{\langle 0\rangle} +\bB_{n_1}\subseteq x+\bB_n$ and $y\in
d_{\langle 1\rangle} +\bB_{n_1}\subseteq y+\bB_n$.

Assuming $n^*_\iota<n_\iota<\omega$ and $\langle d_\sigma:\sigma\in {}^\iota
2\rangle\subseteq \bD$ have been selected,  we first pick open sets $\langle
V^\dagger_\varsigma:\varsigma\in {}^{\iota+1}2\rangle$ such that for all
$\sigma\in {}^\iota 2$ we have  $V^\dagger_{\sigma\conc \langle 0\rangle}
\cap P\neq\emptyset\neq V^\dagger_{\sigma\conc\langle 1\rangle} \cap P$,
$\cl(V^\dagger_{\sigma\conc\langle 0\rangle})\cup
\cl(V^\dagger_{\sigma\conc\langle 1\rangle})\subseteq
(d_\sigma+\bB_{n_\iota})$,  $\cl(V^\dagger_{\sigma\conc\langle
  0\rangle})\cap \cl(V^\dagger_{\sigma\conc\langle 1\rangle})=\emptyset$.
Next, letting $\langle (\varsigma_j',\varsigma_j''):j<j^*\rangle$ be an
enumeration of $\big({}^{\iota+1} 2\big)^{\langle 2\rangle}$, we choose
inductively open sets
\[V^\dagger_\varsigma=V^0_\varsigma\supseteq V^1_\varsigma\supseteq \ldots
  V^{j^*}_\varsigma \]
and integers
\[n_\iota=N^0_\varsigma\leq N^1_\varsigma\leq \ldots\leq N^{j^*}_\varsigma\]
(for $\varsigma\in {}^{\iota+1} 2$), as well as $N_j,\bar{\ell}_j,\bar{m}_j$,
in such a manner that the following demands (a)--(d) are satisfied for all
$j<j^*$. 
\begin{enumerate}
\item[(a)] If $\varsigma\in {}^{\iota+1} 2\setminus \{\varsigma_j', 
  \varsigma_j''\}$, then $V^{j+1}_\varsigma=V^j_\varsigma$ and
  $N^{j+1}_\varsigma =N^j_\varsigma$.
\item[(b)] $N_j,\bar{\ell}_j,\bar{m}_j$ are such that the  the set
  $Z^{N_j}_{\bar{\ell}_j,\bar{m}_j}\cap \big[V^j_{\varsigma_j'}\times
  V^j_{\varsigma_j''} \big]$ is not meager in $P^2$ and $N^{j+1}_{\varsigma_j'}
  =N^{j+1}_{\varsigma_j''}$ is such that
  $N_j+N^j_{\varsigma_j'}+N^j_{\varsigma_j''}<N^{j+1}_{\varsigma_j'}$,
  $\bar{\ell}_j \subseteq N^{j+1}_{\varsigma_j'}$ and $\bar{m}_j\subseteq
  N^{j+1}_{\varsigma_j'}$. 
\item[(c)] Open sets $V^{j+1}_{\varsigma_j'}\subseteq  V^j_{\varsigma_j'}$
  and $V^{j+1}_{\varsigma_j''}\subseteq V^j_{\varsigma_j''}$  are such that  
 $V^{j+1}_{\varsigma_j'}\cap  P\neq\emptyset\neq V^{j+1}_{\varsigma_j''}\cap
 P$ and  
\begin{enumerate}
\item[(A)] $\big[(V^{j+1}_{\varsigma_j'}\times V^{j+1}_{\varsigma_j''})\cap
  P^2\big]\setminus Z^{N_j}_{\bar{\ell}_j,\bar{m}_j}$ is meager in $P^2$ (where
  $N_j,\bar{\ell}_j,\bar{m}_j$ are the ones fixed in (b) above).
\end{enumerate}
\item[(d)] $\big(V^{j+1}_{\varsigma_j'}\times
  V^{j+1}_{\varsigma_j''}\big)\cap \bigcup\limits_{i\leq \iota} F_i=\emptyset$.
\end{enumerate}
Then we set $n_{\iota+1}^*=\max\{N^{j^*}_\varsigma:\varsigma\in {}^{\iota+1}
2\}$ and we choose inductively open sets 
\[V^{j^*}_\varsigma\supseteq V^{j^*+1}_\varsigma\supseteq
  V^{j^*+2}_\varsigma\supseteq \ldots  V^{j^*+j^*}_\varsigma \]
(for $\varsigma\in {}^{\iota+1}2$) so that the following conditions (e)--(f)
are satisfed. 

\begin{enumerate}
\item[(e)] If $\varsigma\in {}^{\iota+1} 2\setminus \{\varsigma_j', 
  \varsigma_j''\}$, then $V^{j^*+j+1}_\varsigma=V^{j^*+j}_\varsigma$.
\item[(f)]  Open sets $V^{j^*+j+1}_{\varsigma_j'}\subseteq
  V^{j^*+j}_{\varsigma_j'}$  and $V^{j^*+j+1}_{\varsigma_j''}\subseteq
  V^{j^*+j}_{\varsigma_j''}$  are such that   
 $V^{j^*j+1}_{\varsigma_j'}\cap  P\neq\emptyset\neq
 V^{j^*+j+1}_{\varsigma_j''} \cap  P$ and  
\begin{enumerate}
\item[(B)] for every $\bar{\ell}\subseteq n^*_{\iota+1}$, $\bar{m}\subseteq
  n^*_{\iota+1}$  and $N<n^*_{\iota+1}$, either
  $\big[(V^{j^*+j+1}_{\varsigma_j'}\times V^{j^*+j+1}_{\varsigma_j''})\cap  P^2\big]
  \setminus  Z^N_{\bar{\ell},\bar{m}}$ is meager in $P^2$, or 
  $(V^{j^*+j+1}_{\varsigma_j'}\times V^{j^*+j+1}_{\varsigma_j''})\cap
  Z^N_{\bar{\ell},\bar{m}}$ is meager in $P^2$. 

\end{enumerate}
\end{enumerate}
Next, we fix $x_\varsigma\in V^{2j^*}_\varsigma\cap P$ for $\varsigma\in
{}^{\iota+1} 2$. Choose $n> n^*_{\iota+1}$ so large that
\begin{itemize}
\item $\rho(x_\varsigma,x_{\varsigma'}) >2^{3-n}$ for distinct $\varsigma,
  \varsigma' \in {}^{\iota+1} 2$,
\item ${\rm diam}_{\rho^*}(x_\varsigma+\bB_n)< 2^{-\iota-1}$ and
  $x_\varsigma+\bB_n \subseteq V^{j^*}_\varsigma$  for all $\varsigma
  \in {}^{\iota+1} 2$.
\end{itemize}
Then we set $n_{\iota+1}=n+2$ and choose $d_\varsigma \in (x_\varsigma+
\bB_{n_{\iota+1}}) \cap \bD$.

This completes the description of the inductive construction.
\bigskip

It follows from $(\boxdot)^a_2+(\boxdot)^b_2$ that for each $\eta\in \can$
the set $\bigcap\limits_{\ell<\omega} d_{\eta\rest \ell}+\bB_{n_\ell}$ is a
singleton included in $P$. By $(\boxdot)^e_2$ we know that for
$\eta\neq\eta'$ 
\[\bigcap_{\ell<\omega} (d_{\eta\rest\ell}+\bB_{n_\ell})\times
  \bigcap_{\ell<\omega} (d_{\eta'\rest\ell}+\bB_{n_\ell})\subseteq B^*.\]
For $\sigma\in {}^\iota 2$ and $\ell<\omega$
let $\sigma*_\ell 0= \sigma\conc \langle\underbrace{0,\ldots,0}_\ell\rangle$
and let $x_\sigma^*\in\bbH$ be such that 
\begin{enumerate}
\item[$(\boxdot)_3$] $\displaystyle \{x_\sigma^*\}=\bigcap_{\ell <\omega}
  \Big(d_{\sigma*_\ell 0}+\bB_{n_{\iota+\ell}}\Big)$; so $x_\sigma^*\in P$
  and if $\sigma\neq \sigma'$ are from ${}^\iota 2$ then
  $(x^*_\sigma,x^*_{\sigma'}) \in B^*$. 
\end{enumerate}
Let $\name{P}, \name{F}_j,\name{n}^*_\iota, \name{n}_\iota,\name{d}_\sigma,
\name{x}_\sigma^*$ be $\bbP$--names for the objects appearing in 
$(\boxdot)_2$--$(\boxdot)_3$. Still working in $\bV[G]$, we may choose a  
sequence $\langle p_\iota,q_\iota:\iota<\omega\rangle\subseteq G$ such that: 
\begin{enumerate}
\item[$(\boxdot)^a_4$]  $p_0\forces_\bbP$`` $\name{P}$ is a perfect subset
  of $\bbH$, $\name{F}_j$ are closed nowhere dense subsets of $P^2$, and
  $\name{n}^*_\iota, \name{n}_\iota,\name{d}_\sigma,  \name{x}_\sigma^*$
  have the properties stated in  $(\boxdot)_1^a$--$(\boxdot)_1^b$,
  $(\boxdot)_2^a$--$(\boxdot)_2^e$, $(\boxdot)_3$ '',  and   
\item[$(\boxdot)^b_4$] $p_\iota$ decides the values of $\name{n}_\iota^*,
  \name{n}_\iota$ and $\name{d}_\sigma$ for $\sigma\in {}^\iota 2$,
  $\iota>0$,   
\item[$(\boxdot)^c_4$] $p_\iota\leq q_\iota\leq p_{\iota+1}$ and $p_\iota,
  q_\iota\in D^2_{n_\iota}\cap D^0_{0, n_\iota, n_\iota}\cap G$ (see
  \ref{dense}(2)) and $n^{p_\iota}+10<n^{q_\iota}$ and $w^{p_\iota}=w^{q_\iota}$. 
\end{enumerate}
\medskip

The properties of conditions from $\bbP$ stated in \ref{fordef}(A) are
absolute, so they hold in $\bV[G]$ as well (with $\bB_\ell$ being
$\bB_\ell^G$ etc). Now, still working in $\bV[G]$, for $0<\iota<\omega$ let
$\bX_\iota=\{x^*_\sigma:\sigma\in {}^\iota 2\}$. Note that
$x^*_\sigma\neq x^*_{\sigma'}$ and $(x^*_\sigma,x^*_{\sigma'})\in B^*$ when
$\sigma,\sigma'\in {}^\iota 2$ are distinct, and $\bX_\iota \subseteq
\bX_{\iota'}$ when $\iota\leq \iota'<\omega$. It follows from
$(\boxdot)^d_2+(\boxdot)_3$ that for $x,y\in\bX_\iota$, $x\neq y$, we have
$(x,y)\in Z^N_{\bar{\ell},\bar{m}}$ for some $N=N(x,y)<n_\iota^*$, 
$\bar{\ell}=\bar{\ell}(x,y), \bar{m} =\bar{m}(x,y)\subseteq n_\iota^*$. By
clause $(\boxdot)^c_2$, these $N(x,y),\bar{\ell}(x,y),\bar{m}(x,y)$ may be
chosen in such a manner that
\begin{enumerate}
\item[$(\boxdot)_5$] if $\sigma,\sigma'\in {}^\iota 2$, $\iota^*<\iota$,
  $\sigma\rest \iota^*=\sigma'\rest \iota^*$ but $\sigma(\iota^*)\neq
  \sigma'(\iota^*)$, and $\varsigma=\sigma\rest (\iota^*+1)$,
  $\varsigma'=\sigma'\rest (\iota^*+1)$, then
  $\bar{\ell}(x_\sigma,x_{\sigma'})=\bar{\ell}(x_\varsigma,
  x_{\varsigma'})$,   $\bar{m}(x_\sigma,x_{\sigma'})
  =\bar{m}(x_\varsigma,x_{\varsigma'})$,  and $N(x_\sigma,x_{\sigma'}) 
=N(x_\varsigma,x_{\varsigma'})$. 
\end{enumerate}
Let $J<\omega$ be such that the arrow property $J\longrightarrow 
(10)^4_{2^{144}}$ holds true and fix a $\iota\geq J$ for a while.
\medskip

Fix $x,y\in\bX_\iota$, $x\neq y$, and let $N=N(x,y)<n_\iota^*$,
$\bar{\ell}=\bar{\ell}(x,y), \bar{m} =\bar{m}(x,y)\subseteq n_\iota^*$. Then 
there are $c_i\in \name{\bF}^G_{\ell_i}$ and $d_i\in \name{\bF}^G_{m_i}$
(for $i<k$) such that for $i\neq i'$ we have  
\[ x-y=c_i-d_i \ \mbox{ and } \ 2^{-n_\iota} <2^{-N}<\rho(c_i,c_{i'}),  \
  \mbox{ and } \   2^{-n_\iota} < 2^{-N}<\rho(d_i,d_{i'}).\]   
The reasons for the use of $q_\iota$ rather than $p_\iota$ in what
follows will become clear at the end. Since $n_\iota<M^{q_\iota}$ and
$\name{\bF}_m^G\subseteq F(q_\iota, m)$  for all $m<M^{q_\iota}$, we get
$c_i\in U^{q_\iota}_\alpha(n^{q_\iota})+ W^{q_\iota}_{j,\alpha,\beta}$ for
some $j<k$ and $(\alpha,\beta)\in (w^{q_\iota})^{\langle 2\rangle}= 
(w^{p_\iota})^{\langle 2\rangle}$ and similarly for $d_i$. Therefore, for
each $i<k$ we may pick      
\begin{itemize}
\item $U_i(x,y),U_i(y,x)\in\{U^{q_\iota}_\alpha(n^{q_\iota}):\alpha\in
  w^{q_\iota}\}$, and 
\item $W_i(x,y), W_i(y,x)\in\{W^{q_\iota}_{j,\alpha,\beta}:
  (\alpha,\beta) \in (w^{q_\iota})^{\langle 2\rangle}\}$, and 
\item $a_i(x,y)\in U_i(x,y)$, $a_i(y,x)\in U_i(y,x)$ and $b_i(x,y)\in
  W_i(x,y)$, $b_i(y,x)\in W_i(y,x)$    
\end{itemize}
such that 
\[x-y=\big(a_i(x,y)+b_i(x,y)\big)-\big(a_i(y,x)+b_i(y,x)\big)\]
and  for $i\neq i'$
\[  \begin{array}{l}
2^{-n_\iota}<\rho\Big(a_i(x,y)+b_i(x,y), a_{i'}(x,y)+b_{i'}(x,y)\Big)\\ 
2^{-n_\iota}<\rho\Big(a_i(y,x)+b_i(y,x), a_{i'}(y,x)+b_{i'}(y,x)\Big).
  \end{array}\]    
Since the metric $\rho$ is invariant (and by $(\boxdot)^a_2$), we also have
\[2^{-n_\iota}<\rho(x,y)=\rho\Big(a_i(x,y)+b_i(x,y),a_i(y,x)+b_i(y,x)
  \Big).\] 
Since $q_\iota\in D^2_{n_\iota}$ we know that for all relevant
$j,\alpha,\beta$,   
\[{\rm diam}_\rho\Big(U^{q_\iota}_\alpha(n^{q_\iota})+
  W^{q_\iota}_{j,\alpha,\beta} \Big) <2^{-n_\iota},\] 
and consequently each of the sets $U^{q_\iota}_\alpha(n^{q_\iota})+ 
W^{q_\iota}_{j,\alpha,\beta}$ contains at most one element from each of the
sets $\big\{a_i(x,y)+b_i(x,y), a_i(y,x)+b_i(y,x)\big\}$,
$\big\{a_i(x,y)+b_i(x,y), a_{i'}(x,y)+b_{i'}(x,y)\big\}$  and
$\big\{a_i(y,x)+b_i(y,x), a_{i'}(y,x)+b_{i'}(y,x)\big\}$. Since $q_\iota\in
D^2_{n_\iota}$, different sets of the form $U^{q_\iota}_\alpha(n^{q_\iota}) 
+W^{q_\iota}_{j,\alpha,\beta}$ are disjoint, and thus we see that the
assumptions (i)--(iv) of Lemma \ref{getdiff} are satisfied.  
\medskip

Unfixing $x,y$, we may use Lemma \ref{getdiff}(1) to conclude that 
\begin{enumerate}
\item[$(\boxdot)_6$] \quad $\bX_\iota-\bX_\iota\subseteq
  \bigcup\big\{U_\alpha^{q_\iota} (n^{q_\iota}-2)
  -U_\beta^{q_\iota}(n^{q_\iota}-2): \alpha,\beta \in w^{q_\iota}\big\}$
\end{enumerate}
and hence also 
\[\bX_\iota-\bX_\iota\subseteq \bigcup\big\{U_\alpha^{p_\iota}
  (n^{p_\iota}) -U_\beta^{p_\iota}(n^{p_\iota}): \alpha,\beta \in
  w^{p_\iota} \big\}.\]  
Moreover, by \ref{getdiff}(2), we also conclude that
\begin{enumerate}
\item[$(\boxdot)_7$] if $x,y\in\bX_\iota$ and  $0\neq x-y\in
  U_\alpha^{q_\iota}(n^{q_\iota}-2) -U_\beta^{q_\iota}(n^{q_\iota}-2)$, then
  $\alpha\neq \beta$ and $\bar{m}(x,y)(i)=\bar{\ell}(x,y)(i)=
  h^{q_\iota}(\alpha,\beta)=  h^{p_\iota}(\alpha,\beta)$ for all  $i<k$. 
\end{enumerate}
Since $\big\{U_\alpha^{p_\iota}(n^{p_\iota}): \alpha\in w^{p_\iota}\big\}$,
$\big\{U_\alpha^{p_\iota}(n^{p_\iota}-1): \alpha\in w^{p_\iota} \big\}$,
$\big\{U_\alpha^{p_\iota}(n^{p_\iota}-2): \alpha\in w^{p_\iota} \big\}$ and
$\bX_\iota$ satisfy the assumptions of Theorem \ref{gettranslMin}, we get
that exactly one of $(A)_\iota$, $(B)_\iota$ below holds true. 
 \begin{enumerate}
 \item[$(A)_\iota$]   There is a $c_\iota\in \bbH$ such that $\bX_\iota+
   c_\iota\subseteq \bigcup\big\{U_\alpha^{p_\iota}(n^{p_\iota}-2):
   \alpha\in w^{p_\iota} \big\}$.  
 \item[$(B)_\iota$]   There is a $c_\iota\in \bbH$ such that $c_\iota-
   \bX_\iota\subseteq \bigcup\big\{U_\alpha^{p_\iota}(n^{p_\iota}-2):
   \alpha\in w^{p_\iota} \big\}$.  
 \end{enumerate}

Unfixing $\iota<\omega$ we let 
\[ \begin{array}{l}
A=\{\iota<\omega: J\leq \iota\mbox{ and case $(A)_\iota$ holds true }\}\\
B=\{\iota<\omega: J\leq \iota\mbox{ and case $(B)_\iota$ holds true }\}.
\end{array}\] 
One of the sets $A,B$ is infinite and this leads us to two very similar
cases.
\medskip

\noindent{\sc Case:} The set $A$ is infinite.\\
For $\iota\in A$ let $\bX_\iota,c_\iota$ be as before. Let $w_\iota=
\{\alpha\in w^{p_\iota}: U_\alpha^{p_\iota}(n^{p_\iota}-2)\cap (\bX_\iota
+c_\iota) \neq \emptyset\}$. Since ${\rm diam}_\rho\big(U_\alpha^{p_\iota}
(n^{p_\iota}-2)\big)<2^{-n_\iota}< \rho(x,y)$ for $\alpha\in w_\iota$ and distinct
$x,y\in\bX_\iota$, we get $\big|U_\alpha^{p_\iota}(n^{p_\iota}-2) \cap
(\bX_\iota +c_\iota)\big|=1$ for $\alpha\in w_\iota$. Consequently, we have
a natural bijection $\varphi_\iota:\bX_\iota\longrightarrow w_\iota$ such
that $x+c_\iota\in U_{\varphi_\iota(x)}^{p_\iota}(n^{p_\iota}-2)$. 

For $\iota<\iota'$ from $A$ we have $\bX_\iota \subseteq \bX_{\iota'}$ and
the mapping $\pi_{\iota,\iota'}=\varphi_{\iota'}\circ \varphi_\iota^{-1}:
w_\iota\longrightarrow w_{\iota'}$ is an injection. Clearly, if $x\in
\bX_\iota$, $\alpha=\varphi_\iota(x)\in w_\iota$ then  
\begin{enumerate}
\item[$(\boxdot)_8$] \quad $x+c_{\iota'}\in \Big(U^{p_{\iota'}}_\alpha( 
  n^{p_\iota}-2) +(c_{\iota'}-c_\iota)\Big)\cap
  U^{p_{\iota'}}_{\pi_{\iota,\iota'}(\alpha)} (n^{p_\iota}-2)\neq\emptyset$.   
\end{enumerate}
Suppose now that $x,y\in \bX_\iota$, $x\neq y$. By $(\boxdot)_6$, there are
$\alpha,\beta\in w^{q_\iota}$ such that $x-y\in
U^{q_\iota}_\alpha(n^{q_\iota}-2) - U^{q_\iota}_\beta(n^{q_\iota}-2)$ (and,
by $(\boxdot)_7$, $\alpha\neq\beta$). Then also 
\[x-y\in \Big(U^{p_\iota}_\alpha(n^{p_\iota}-2) - U^{p_\iota}_\beta
(n^{p_\iota}-2) \Big) \cap\Big(U^{p_\iota}_{\varphi_\iota(x)}
(n^{p_\iota}-2) -U^{p_\iota}_{\varphi_\iota(y)} (n^{p_\iota}-2) \Big).\] 
By Lemma \ref{abc} we conclude that $\alpha=\varphi_\iota(x)$ and
$\beta=\varphi_\iota(y)$. Together with $(\boxdot)_7$ this gives us that 
\begin{enumerate}
\item[$(\boxdot)_9^\iota$] if $(x,y)\in (\bX_\iota)^{\langle 2\rangle}$,
  then $\bar{m}(x,y)(i)=\bar{\ell}(x,y)(i)= h^{p_\iota}(\varphi_\iota(x),
  \varphi_\iota(y))$ for all $i<k$. 
\end{enumerate}
Putting together $(\boxdot)_9^\iota$ and $(\boxdot)_9^{\iota'}$ we see that 
\begin{enumerate}
\item[$(\boxdot)_{10}$] if $\iota<\iota'$ are from $A$ and $(x,y)\in
  (\bX_\iota)^{\langle 2\rangle}$,  then 
\[h^{p_\iota}(\varphi_\iota(x), \varphi_\iota(y))=
  h^{p_{\iota'}}(\varphi_{\iota'}(x), \varphi_{\iota'}(y)).\] 
In other words, if $(\alpha,\beta)\in w_\iota$ then 
\[h^{p_\iota}(\alpha,\beta)=h^{p_{\iota'}}(\alpha,\beta)= 
  h^{p_{\iota'}}(\pi_{\iota,\iota'}(\alpha), \pi_{\iota,\iota'}(\beta)).\] 
\end{enumerate}
It follows from $(\boxdot)_8+(\boxdot)_{10}$ and
\ref{fordef}(A)$(\boxtimes)_7$ that for $\iota<\iota'$ from $A$ we have 
\begin{enumerate}
\item[$(\boxdot)_{11}$] $\rksp(w_\iota)= \rksp\big(\pi_{\iota,\iota'}
  [w_\iota]\big)$, $\bj(w_\iota)=\bj\big(\pi_{\iota,\iota'} [w_\iota]\big)$,
  $\bk(w_\iota)=\bk\big(\pi_{\iota,\iota'} [w_\iota]\big)$ and 
\[|\alpha\cap w_\iota|=\bk(w_\iota)\quad\Leftrightarrow\quad
  |\pi_{\iota,\iota'}(\alpha)\cap \pi_{\iota,\iota'} [w_\iota]|=
  \bk(w_\iota)\qquad \mbox{ for all }\alpha \in w_\iota.\]   
\end{enumerate}
(Note that $r^{p_{\iota'}}_m\leq n^{p_\iota}-2$ when $m= h^{p_{\iota'}}( 
  \alpha,\beta)$, $\alpha,\beta\in w_\iota\subseteq w^{p_\iota}$,
  $\alpha\neq\beta$.) 
\medskip

Choose a strictly increasing sequence $\langle\iota(\ell): \ell<\omega 
\rangle \subseteq A$ such that 
\begin{enumerate}
\item[$(\boxdot)_{12}$] for each $\ell<\omega$, 
\[ 2^{2-n_{\iota(\ell+1)-1}}<{\rm diam}_\rho\Big(U^{p_{\iota(\ell)}}_0
  (n^{p_{\iota(\ell)}}-2)\Big)={\rm diam}_\rho\Big(
  U^{p_{\iota(\ell+1)}}_0   (n^{p_{\iota(\ell)}}-2)\Big)\]  
(remember $n_\iota$'s were chosen in $(\boxdot)_2$ and $p_{\iota(\ell)}\in
D^0_{0, n_{\iota(\ell)}, n_{\iota(\ell)}}$ so also $0\in w^{p_{\iota(\ell)}}$). 
\end{enumerate}
\medskip

Fix $\ell<\omega$ for a moment and suppose $\varsigma\in
{}^{\iota(\ell)}2$ is such that
\[\big|\varphi_{\iota(\ell)}(x^*_\varsigma)\cap w_{\iota(\ell)}\big|= 
  \bk(w_{\iota(\ell))}).\]
Let $\varsigma^*\in {}^{\iota(\ell+1)}2$ be such that $\varsigma\vtl
\varsigma^*$, $\varsigma^*(n)=0$ for $n\in [\iota(\ell),
\iota(\ell+1))$, and let $\sigma\in {}^{\iota(\ell+1)}2$ be such
that $\varsigma^*\rest (\iota(\ell+1)-1)\vtl \sigma$ and
$\sigma(\iota(\ell+1)-1)=1$. Then $x^*_{\varsigma^*}=x^*_\varsigma$
and $\rho(x^*_{\varsigma^*},x^*_\sigma)<
2^{1-n_{\iota(\ell+1)-1}}$. By $(\boxdot)_{12}$ we have then  
\[\rho(x^*_\varsigma+ c_{\iota(\ell+1)}, x^*_\sigma+
  c_{\iota(\ell+1)}) =\rho(x^*_{\varsigma^*},x^*_\sigma)<{\rm 
    diam}_\rho\Big( U^{p_{\iota(\ell+1)}}_0   (n^{p_{\iota(\ell)}}-2)\Big).\]  
Consequently,
\begin{enumerate}
\item[$(\boxdot)_{13}$]  $\displaystyle
  U^{p_{\iota(\ell+1)}}_{\varphi_{\iota(\ell+1)}  (x^*_\varsigma)}
  (n^{p_{\iota(\ell)}}-2) = U^{p_{\iota(\ell+1)}}_{
    \varphi_{\iota(\ell+1)} (x^*_\sigma)} (n^{p_{\iota(\ell)}}-2) $ 
\end{enumerate}
(remember \ref{fordef}(A)$(\boxtimes)_5$(b)). It follows from
$(\boxdot)_2^{c,d}+(\boxdot)_5$ that for each $x\in X_{\iota(\ell)}\setminus 
\{x^*_\varsigma\}$ we have $\bar{\ell}(x,x^*_\varsigma)=
\bar{\ell}(x,x^*_\sigma)$ and $\bar{m}(x,x^*_\varsigma)=
\bar{m}(x,x^*_\sigma)$, so by $(\boxdot)_9^{\iota(\ell+1)}$ we also have
\begin{enumerate}
\item[$(\boxdot)_{14}$]  $\displaystyle
  h^{p_{\iota(\ell+1)}}\big(\varphi_{\iota(\ell+1)}(x), 
  \varphi_{\iota(\ell+1)}(x^*_\varsigma) \big) = h^{p_{\iota(\ell+1)}}
  \big(\varphi_{\iota(\ell+1)}(x), \varphi_{\iota(\ell+1)}(x^*_\sigma)
  \big) $. 
\end{enumerate}
Condition \ref{fordef}(A)$(\boxtimes)_7$ for $p_{\iota(\ell+1)}$
together with $(\boxdot)_{11}$ imply now that, letting
$w^{\iota(\ell),\sigma}=\big(\pi_{\iota(\ell),\iota(\ell+1)}
\big[w_{\iota(\ell)}\big] \setminus\{\varphi_{\iota(\ell+1)}(x^*_\varsigma)\})\cup 
\{\varphi_{\iota(\ell+1)}(x^*_\sigma)\}$, we have 
\begin{enumerate}
\item[$(\boxdot)_{15}$] $\rksp\big(w^{\iota(\ell),\sigma}\big)
  =\rksp\big(\pi_{\iota(\ell), \iota(\ell+1)} [w_{\iota(\ell)}]\big)= \rksp(w_{\iota(\ell)})$, 

$\bj\big(w^{\iota(\ell),\sigma}\big) =\bj\big(\pi_{\iota(\ell),
  \iota(\ell+1)} [w_{\iota(\ell)}]\big)= \bj(w_{\iota(\ell)})$, and

$\bk\big(w^{\iota(\ell),\sigma}\big) =\bk\big(\pi_{\iota(\ell),
  \iota(\ell+1)} [w_{\iota(\ell)}]\big)= \bk(w_{\iota(\ell)}) = \big|
\varphi_{\iota(\ell+1)}(x^*_\sigma)\cap w^{\iota(\ell),\sigma}\big|$. 
\end{enumerate}
(Remember, $r^{p_{\iota(\ell+1)}}_m\leq n^{p_{\iota(\ell)}}-2$ when $m=
h^{p_{\iota(\ell+1)}}(\alpha,\beta)$, $\alpha,\beta\in
\pi_{\iota(\ell),\iota(\ell+1)}\big[w_{\iota(\ell)}\big]$ are distinct.) 
Consequently, if $\rksp(w_{\iota(\ell)})\geq 0$ then
\[\rksp(w_{\iota(\ell+1)})\leq \rksp\Big(\pi_{\iota(\ell),\iota(\ell+1)}
  \big[w_{\iota(\ell)}\big] \cup
  \{\varphi_{\iota(\ell+1)}(x^*_\sigma)\}\Big) <
\rksp(w_{\iota(\ell)})\] 
(remember Definition \ref{hypo2}$(\circledast)_e$).
\medskip

Unfixing $\ell<\omega$, we see that for some $\ell^*$ we have
$\rksp(w_{\iota(\ell^*)})=-1$. However, applying to $\ell^*$ the
procedure described above we get $\sigma\in {}^{\iota(\ell^*+1)}2$
such that $\varphi_{\iota(\ell^*+1)}(x^*_\sigma)$ contradicts clause
\ref{fordef}(A)$(\boxtimes)_8$ for $p_{\iota(\ell^*+1)}$ (remember
$(\boxdot)_{13}+(\boxdot)_{15}$). 
\medskip

\noindent{\sc Case:} The set $B$ is infinite.\\
Almost identical to the previous case. Defining $\varphi_\iota$ we use the
condition $c_\iota-x\in U^{p_\iota}_{\varphi_\iota(x)}(n^{p_\iota}-2)$, but
then not much other changes is needed. Even in $(\boxdot)_8$ we have 
\[c_{\iota'}-x=(c_\iota-x)+(c_{\iota'}-c_\iota) \in \Big(U^{p_{\iota'}}_\alpha( 
  n^{p_\iota}-2) +(c_{\iota'}-c_\iota)\Big)\cap
  U^{p_{\iota'}}_{\pi_{\iota,\iota'}(\alpha)}
  (n^{p_\iota}-2)\neq\emptyset\]
(where $\alpha=\varphi_\iota(x)\in w_\iota$).    
\end{proof}

The following theorem is the consequence of results presented in this 
section. 

\begin{theorem}
\label{mainA}  
Assume that 
  \begin{enumerate}
  \item $(\bbH,+,0)$ is an Abelian perfect Polish group,
  \item the set of elements of $\bbH$ of order larger than 2 is dense
    in $\bbH$,
  \item $2\leq k<\omega$ and 
\item $\vare<\omega_1$ and $\lambda$ is an uncountable
  cardinal such that ${\rm NPr}^\vare(\lambda)$ holds true. 
  \end{enumerate}
Then there is a ccc forcing notion $\bbP$ of cardinality $\lambda$ such that 
\[\begin{array}{ll}
\forces_\bbP &\mbox{`` for some $\Sigma^0_2$ subset $B$ of $\bbH$ we
               have:}\\  
&\quad \mbox{ there is a set }X\subseteq \bbH\mbox{ of cardinality }
  \lambda\mbox{ such that }\\
&\qquad\big(\forall x,y\in X\big)\big(\big| (x+B)\cap
  (y+B)\big|\geq k\big)\\
&\quad \mbox{ but there is no perfect set }P\subseteq \bbH \mbox{ such that 
  } \\
&\qquad\big(\forall x,y\in P\big)\big(\big| (x+B)\cap (y+B)\big|\geq
  k\big)\mbox{ ''.}
  \end{array}\]
\end{theorem}

\section{Forcing for groups with all elements of order $\leq 2$} 
Let us consider the situation when the main (algebraic) assumption of the
previous section fails: the set of elements of $\bbH$ of order larger than 2 is
NOT dense in $\bbH$. Let $H_2=\{a\in \bbH:a+a=0\}$, so $H_2$ is a
closed subgroup of $\bbH$ and its complement $\bbH\setminus H_2$ is
not dense in $\bbH$. Consequently, the interior of $H_2$ is not
empty and thus also $H_2$ is an open subset of $\bbH$. If $\bbH$ is a
perfect Polish group, so is $H_2$. Each coset of $H_2$ is clopen and  
consequently $\bbH/H_2$ is countable. 

Suppose that $\bT\subseteq H_2$ is a Borel set with $\lambda$ many
$k$--overlapping translations but without a perfect set of such
translations. Then $\bT$ is also a Borel subset of $\bbH$ and it still has
$\lambda$ many $k$--overlapping translations. If $P\subseteq \bbH$ is a
perfect set, then (as $|\bbH/H_2|\leq \omega$) for some $a\in \bbH$ the
intersection $P\cap (H_2+a)$ is uncountable. Consider $Q=\big(P\cap
(H_2+a)\big)-a\subseteq H_2$ --- it is a closed uncountable subset of $H_2$
(so contains a perfect set) and by the assumptions on $\bT$ there are
$c,d\in Q$ such that $\big|(\bT+c)\cap (\bT+d)\big|<k$. Then
$c+a,d+a\in P$ and $\big|(\bT+(c+a))\cap (\bT+(d+a))\big|=
\big|\big((\bT+c)\cap (\bT+d)\big)+a\big|<k$. 

Consequently, to completely answer the problem of Borel sets with
non--disjoint translations it is enough to deal with the case of all
elements of $\bbH$ being of order $\leq 2$. The arguments in this case are
similar to those from Section 5, but they are simpler. However, there is one
substantial difference. If $\bbH$ is a Polish group with all elements of
order $\leq 2$ and $B\subseteq \bbH$ is an uncountable Borel set, then $B$
has a perfect set of pairwise $2$--overlapping translations. Namely, choosing a
perfect set $P\subseteq B$ we will have $x+y,0\in (B+x)\cap (B+y)$ for each
$x,y\in P$. Moreover, if $x+b_0=y+b_1$, then also $x+b_1=y+b_0$. Therefore,
if $x\neq y$ and $(B+x)\cap (B+y)$ is finite, then $|(B+x)\cap (B+y)|$ must
be even. For that reason the meaning of $k$ in our forcing here will be
slightly different: the translations of the new Borel set will have at least
$2k$ elements.

\begin{hypothesis}
\label{hypall2}
  In the rest of the section we assume the following:
  \begin{enumerate}
\item $(\bbH,+,0)$, $\bD$, $\rho,\rho^*$ and $\cU$ are as in Assumption
   \ref{hypno}.
\item All elements of $\bbH$ have orders at most 2. 
\item $1<k<\omega$.
\item $\vare$ is a countable ordinal and $\lambda$ is an uncountable
  cardinal such that ${\rm NPr}^\vare(\lambda)$ holds true. The model
  $\bbM(\vare,\lambda)$ and functions $\rksp,\bj$ and $\bk$ on
  $[\lambda]^{<\omega}\setminus \{\emptyset\}$ are as fixed in Definition
  \ref{hypo2}. 
  \end{enumerate}
\end{hypothesis}

In groups with all elements of order two we should use a weaker notion of
independence. 

\begin{definition}
  \label{quasi-in}
  Let $(\bbH,+,0)$ be an Abelian group
  \begin{enumerate}
\item  A set $\bB\subseteq \bbH$ is  {\em  quasi$^-$  independent in
    $\bbH$\/} if $|\bB|\geq 8$ and if for all distinct
  $b_0,b_1,b_2,\ldots,b_7\in\bB$ and any $e_0,e_1,e_2,\ldots, e_7\in
  \{0,1\}$ not all equal $0$, we have  
  \[e_0b_0+e_1b_1+e_2b_2+e_3b_3+e_4b_4+e_5b_5+e_6b_6+e_7b_7\neq 0.\]
\item A family $\{V_i:i\leq n\}$ of disjoint subsets of $\bbH$ is a qif$^-$
  if for each choice of $b_i\in V_i$, $i\leq n$, the set $\{b_i:i\leq n\}$
  is quasi$^-$ independent. 
\end{enumerate}
\end{definition} 

\begin{proposition}
  \label{exis2ord}
  Assume that
  \begin{enumerate}
\item[(i)] $(\bbH,+,0)$ is a perfect Abelian Polish group,
\item[(ii)] $U_0,\ldots,U_n$ are nonempty open subsets of $\bbH$, $n\geq 7$. 
\end{enumerate}
Then there are non-empty open sets $V_i\subseteq U_i$ (for $i\leq n$) such
that $\{V_i:i\leq n\}$ is a qif$^-$. 
\end{proposition}

\begin{proof}
Similar to Proposition \ref{existence}.
\end{proof}

The forcing notion used in the case of groups with all elements of order
$\leq 2$ is almost the same as the one introduced in Definition
\ref{fordef}. The only difference is that instead of 8--good qifs we use the
weaker concept of qifs$^-$. (There are no 8--good qifs in the current case.)
Since in the current case, $a-b=a+b$ for $a,b\in\bbH$, we still can repeat 
all needed ingredients of Section 5. To stress the importance of this
property we will consistently use the addition $+$ rather than subtraction
$-$. 

\begin{definition}
  \label{5.2cp}
{\bf (A)}\quad Let $\bbQ$ be the collection of all tuples  
\[p=\big(w^p,M^p,\bar{r}^p,n^p,\bar{\Upsilon}^p,\bar{V}^p,h^p\big)=   
  \big(w,M,\bar{r},n,\bar{\Upsilon},\bar{V},h\big)\]      
such that the following demands $(\otimes)_1$--$(\otimes)_8$ are
satisfied.  
\begin{enumerate}
\item[$(\otimes)_1$] $w\in [\lambda]^{<\omega}$, $|w|\geq 4$,
  $0<M<\omega$, $3\leq n<\omega$ and $\bar{r}=\langle r_m:m<M\rangle$ where  
  $r_m\leq n-2$  for $m<M$.  
\item[$(\otimes)_2$] $\bar{\Upsilon}= \langle \bar{U}_\alpha:\alpha\in
  w\rangle$ where each $\bar{U}_\alpha=\langle U_\alpha(\ell):\ell\leq n\rangle$
  is a $\subseteq$--decreasing sequence of elements of $\cU$.
\item[$(\otimes)_3$]  $\bar{V}=\langle Q_{i,\alpha,\beta},
    V_{i,\alpha,\beta}, W_{i,\alpha,\beta}: i<k,\ (\alpha,\beta)\in
    w^{\langle 2\rangle}\rangle\subseteq \cU$ and  $Q_{i,\alpha,\beta}=
    Q_{i,\beta,\alpha}\supseteq V_{i,\alpha,\beta}=
  V_{i,\beta,\alpha} \supseteq W_{i,\alpha,\beta}= W_{i,\beta,\alpha}$ for
  all $i<k$ and $(\alpha,\beta)\in w^{\langle 2\rangle}$. 
\item[$(\otimes)_4$] 
  \begin{enumerate}
  \item[(a)] The indexed family $\langle U_\alpha(n-2):\alpha\in
    w\rangle\conc \langle Q_{i,\alpha,\beta}: i<k,\ \alpha,\beta\in   w,\
    \alpha<\beta \rangle$ is a qif$^-$ (so in particular the sets in
    this system are pairwise disjoint), and  
  \item[(b)] $\langle U_\alpha(n):\alpha\in  w\rangle\conc \langle
W_{i,\alpha,\beta}: i<k,\ \alpha,\beta\in   w,\ \alpha<\beta \rangle$ is
immersed in $\langle U_\alpha(n-1):\alpha\in  w\rangle\conc \langle
V_{i,\alpha,\beta}: i<k,\ \alpha,\beta\in   w,\  \alpha<\beta \rangle$ and
$\langle U_\alpha(n-1):\alpha\in  w\rangle\conc \langle  V_{i,\alpha,\beta}:
i<k,\ \alpha,\beta\in   w,\ \alpha<\beta \rangle$ is immersed in $\langle
U_\alpha(n-2):\alpha\in w\rangle\conc \langle Q_{i,\alpha,\beta}:  
  i<k,\ \alpha,\beta\in   w,\ \alpha<\beta \rangle$.
  \end{enumerate}
\item[$(\otimes)_5$] 
  \begin{enumerate}
\item[(a)] If $\alpha,\beta\in w$, $\ell\leq n$ and $U_\alpha(\ell)\cap 
U_\beta(\ell)\neq \emptyset$, then $U_\alpha(\ell)=U_\beta(\ell)$, and 
\item[(b)] if $\alpha,\beta,\gamma\in w$, $\ell\leq n$, $U_\alpha(\ell)\neq
  U_\beta(\ell)$ and $a\in U_\alpha(\ell)$, $b\in U_\beta(\ell)$, then
  $\rho(a,b)>{\rm diam}_\rho\big(U_\gamma(\ell)\big)$. 
\end{enumerate}
\item[$(\otimes)_6$] $h: w^{\langle 2\rangle}\stackrel{\rm
    onto}{\longrightarrow} M$ is such that $h(\alpha,\beta)=h(\beta,\alpha)$
  for $(\alpha,\beta)\in w^{\langle 2\rangle}$.  
\item[$(\otimes)_7$] Assume that $u,u'\subseteq w$, $\pi$ and $\ell\leq n$
  are such that  
  \begin{itemize}
\item $4\leq |u|=|u'|$ and $\pi:u\longrightarrow u'$ is a bijection, 
\item $r_{h(\alpha,\beta)}\leq \ell$ for all $(\alpha,\beta)\in u^{\langle
    2\rangle}$,  
\item $U_\alpha(\ell)\cap U_\beta(\ell)=\emptyset$ and
  $h(\alpha,\beta)=h(\pi(\alpha),\pi(\beta))$ for all distinct
  $\alpha,\beta\in u$,
\item for some $c\in \bbH$,  for all $\alpha\in u$, we have 
$\big(U_\alpha(\ell)+c\big)\cap U_{\pi(\alpha)}(\ell)\neq \emptyset$.
\end{itemize}
Then $\rksp(u)=\rksp(u')$, $\bj(u)=\bj(u')$, $\bk(u)=\bk(u')$ and for
$\alpha \in u$ 
\[|\alpha\cap u|=\bk(u)\quad\Leftrightarrow\quad |\pi(\alpha)\cap u'|=
  \bk(u).\]  
\item[$(\otimes)_8$] Assume that  
  \begin{itemize}
  \item $\emptyset\neq u\subseteq w$, $\rksp(u)=-1$, $\ell\leq n$ and
\item $\alpha\in u$ is such that $|\alpha\cap u|=\bk(u)$, and 
  \item  $r_{h(\beta,\beta')}\leq \ell$  and $U_\beta(\ell)\cap
    U_{\beta'}(\ell)=\emptyset$ for all $(\beta,\beta')\in u^{\langle
      2\rangle}$.   
  \end{itemize}
Then there is {\bf no}\/ $\alpha'\in w\setminus u$ such that $U_\alpha(\ell)
=U_{\alpha'}(\ell)$ and $h(\alpha,\beta)= h(\alpha', \beta)$ for all
$\beta\in u\setminus\{\alpha\}$.   
\end{enumerate}

\noindent {\bf (B)}\quad For $p\in\bbQ$ and $m<M^p$ we define
\[F(p,m)= \bigcup\big\{U^p_\alpha(n^p)+W^p_{i,\alpha,\beta}: 
     (\alpha,\beta)\in (w^p)^{\langle 2\rangle} \  \wedge\ i<k \  \wedge\ 
     h^p(\alpha,\beta)=m\big\}.\]     

\noindent {\bf (C)}\quad For $p,q\in \bbQ$ we declare that $p\leq q$\quad if
and only if  
\begin{itemize}
\item $w^p\subseteq w^q$, $M^p\leq M^q$, $\bar{r}^q\rest M^p=\bar{r}^p$,  
  $n^p\leq n^q$, $h^q\rest (w^p)^{\langle 2\rangle} = h^p$, and 
\item if $\alpha\in w^p$ and $\ell\leq n^p$ then $U^q_\alpha(\ell) =
  U^p_\alpha(\ell)$, and 
\item if $(\alpha,\beta)\in (w^p)^{\langle 2\rangle}$, $i<k$, then
$Q^q_{i,\alpha, \beta}\subseteq Q^p_{i,\alpha,\beta}$, $V^q_{i,\alpha,
\beta} \subseteq V^p_{i,\alpha,\beta}$, and $W^q_{i,\alpha, \beta}\subseteq
W^p_{i,\alpha,\beta}$, and  
\item if $m<M^p$, then $F(q,m)\subseteq F(p,m)$. 
\end{itemize}
\end{definition}

\begin{lemma}
\label{5.3cp}
\begin{enumerate}
\item $(\bbQ,\leq)$ is a partial order of size $\lambda$.
\item The following sets are dense in $\bbQ$:
  \begin{enumerate}
  \item[(i)] $D^0_{\gamma,M,n}=\big\{p\in\bbQ:\gamma\in u^p\ \wedge \ M^p>M
    \ \wedge\ n^p>n \big\}$ for $\gamma<\lambda$ and $M,n<\omega$.
\item[(ii)] $D^1_N=\big\{p\in\bbQ:\diam(U^p_\alpha(n^p-2))<2^{-N}\ \wedge\ 
  \diam(Q^p_{i,\alpha,\beta})<2^{-N} \ \wedge$

\qquad  $\diam(U^p_\alpha(n^p-2)+Q^p_{i,\alpha,\beta})< 2^{-N} \mbox{ 
  for all }i<k,\ (\alpha, \beta)\in (w^p)^{\langle 2\rangle}\big\}$

for $N<\omega$.  
\item[(iii)] $D^2_N=\big\{p\in\bbQ:$ for all $i,j<k$ and 
$(\alpha,\beta), (\gamma,\delta)\in (w^p)^{\langle 2\rangle}$ it
  holds that

\qquad\quad ${\rm diam}_\rho(U^p_\alpha(n^p-2))<2^{-N}$\ and\ ${\rm
  diam}_\rho (Q^p_{i,\alpha,\beta})<2^{-N}$\ and  

\qquad\quad  ${\rm diam}_\rho (U^p_\alpha(n^p-2)+Q^p_{i,\alpha,\beta})<
2^{-N}$~and  

\qquad\quad if $(i,\alpha^*,\alpha,\beta)\neq (j,\gamma^*,\gamma,\delta)$   
 then 

\qquad\quad $\big(U^p_{\alpha^*}(n^p)+W^p_{i,\alpha,\beta}\big)\cap 
\big(U^p_{\gamma^*}(n^p)+W^p_{i,\gamma,\delta}\big)= \emptyset\big\}$. 

for $N<\omega$.
\item[(iv)] $D^3_N=\big\{p\in D^2_N:$ for some $\langle Q^*_{i,\alpha,\beta}:
  i<k, \alpha,\beta\in w^p, \alpha<\beta\rangle\subseteq \cU$

\qquad\quad the system 

\qquad\quad $\langle U^p_\alpha(n-3):\alpha\in
  w^p\rangle\conc \langle Q^*_{i,\alpha,\beta}:
  i<k, \alpha,\beta\in w^p, \alpha<\beta\rangle$ 

\qquad\quad is a qif$^-$ and

\qquad\quad   $\langle U^p_\alpha(n-2):\alpha\in w^p\rangle\conc \langle
Q_{i,\alpha,\beta}:  i<k, \alpha,\beta\in w^p, \alpha<\beta\rangle$ 

\qquad\quad  is immersed in it $\ \big\}$.  
  \end{enumerate}
\item Assume $p\in\bbQ$. Then there is $q\geq p$ such that $n^q\geq n^p+3$,
  $w^q=w^p$ and  
  \begin{itemize}
\item for all $\alpha\in w^p$, $\cl\big(U^q_\alpha(n^q-2)\big)\subseteq 
  U^p_\alpha(n^p)$, and
\item for all $i<k$ and $(\alpha,\beta)\in (w^p)^{\langle 2\rangle}$, 
\[\cl\big(U^q_\alpha(n^q-2)+Q^q_{i,\alpha,\beta}\big)\subseteq
  U^p_\alpha(n^p)+W^p_{i,\alpha,\beta}\quad\mbox{ and }\quad
  \cl\big(Q^q_{i,\alpha,\beta}\big) \subseteq W^p_{i,\alpha,\beta}.\]
  \end{itemize}
\end{enumerate}
\end{lemma}

\begin{proof}
Same as for \ref{dense} (just using Proposition \ref{exis2ord}).
\end{proof}

\begin{lemma}
  \label{5.4cp}
Suppose that $p\in D^3_1$ and $\alpha,\beta,\gamma,\delta\in w^p$
are such that $\alpha\neq \beta$. If
\[\Big(U^p_\alpha(n^p-2)+U^p_\beta(n^p-2)\Big)\cap
  \Big(U^p_\gamma(n^p-2)+U^p_\delta(n^p-2)\Big)\neq \emptyset,\]
then $\{\alpha, \beta\}=\{\gamma,\delta\}$.
\end{lemma}

\begin{proof}
  Similar to \ref{abc}, remembering $\langle U^p_\alpha(n-2):\alpha\in
  w^p\rangle$ is immersed in a qif$^-$ $\langle U^p_\alpha(n-3):\alpha\in 
  w^p\rangle$; see \ref{5.3cp}(2)(iv).
\end{proof}

\begin{lemma}
  \label{5.5cp}
  The forcing notion $\bbQ$ has the Knaster property.
\end{lemma}

\begin{proof}
  Same as Lemma \ref{Knaster}, but when defining a bound $q$ to
  $p_\xi,p_\zeta$ modify the demands to have $n^q=n^{p_\zeta}+4$ and $q\in
  D^3_1$. 
\end{proof}

\begin{lemma}
\label{5.6cp}
For each $(\alpha,\beta)\in \lambda^{\langle 2\rangle}$ and
  $i<k$,
\[  \begin{array}{ll}
\forces_\bbQ& \mbox{`` the sets }\\
&\displaystyle \bigcap \big\{U^p_\alpha(n^p): p\in \name{G}_\bbQ\
    \wedge\ \alpha\in w^p\big\}\quad\mbox{ and }\quad 
\bigcap \big\{W^p_{i,\alpha,\beta}: p\in \name{G}_\bbQ\
    \wedge\ \alpha,\beta\in w^p\big\}\\
&\mbox{ have exactly one element each. ''}
  \end{array}\]
\end{lemma}

\begin{proof}
 Follows from Lemma \ref{5.3cp}. 
\end{proof}

\begin{definition}
\label{5.7cp}
\begin{enumerate}
\item For $(\alpha,\beta)\in\lambda^{\langle 2\rangle}$ and $i<k$ let
  $\name{\eta}_\alpha$, $\name{\nu}_{i,\alpha,\beta}$ and
  $\name{h}_{\alpha,\beta}$ be $\bbQ$--names such that 
\[  \begin{array}{ll}
\forces_\bbQ&\mbox{`` }\displaystyle \{\name{\eta}_\alpha\}= \bigcap
              \big\{U^p_\alpha(n^p): p\in \name{G}_\bbQ\ \wedge\
              \alpha\in w^p\big\},\\
&\displaystyle\quad \{\name{\nu}_{i,\alpha,\beta}\}=\bigcap \big\{W^p_{i,\alpha,\beta}:
  p\in \name{G}_\bbQ\ \wedge\ \alpha,\beta\in w^p\big\}\\
&\displaystyle\quad\name{h}_{\alpha,\beta}=h^p(\alpha,\beta)\mbox{ for some
  (all) $p\in\name{G}_\bbQ$ such that $\alpha,\beta\in w^p$. ''}
  \end{array}\]
\item For $m<\omega$ let $\name{\bF}_m$ be a $\bbQ$--name such that 
  \[\forces_\bbQ \mbox{`` }\name{\bF}_m=\bigcap \big\{ F(p,m): m<M^p \
    \wedge \ p\in \name{G}_\bbQ\big\}.\mbox{ ''}\]
\end{enumerate}
\end{definition}

\begin{lemma}
\label{5.8cp}
  \begin{enumerate}
\item For each $m<\omega$, $\forces_\bbQ$`` $\name{\bF}_m$ is a
  closed subset of $\bbH$. '' 
\item For $i<k$ and $(\alpha,\beta)\in \lambda^{\langle 2\rangle}$ we
    have 
\[\forces_\bbQ\mbox{`` }
  \name{\eta}_\alpha,\name{\nu}_{i,\alpha,\beta}\in \bbH,\quad
  \name{h}_{\alpha,\beta}<\omega,\quad \name{\nu}_{i,\alpha,\beta}=
  \name{\nu}_{i,\beta, \alpha}\ \mbox{ and }\  \name{\eta}_\alpha +
  \name{\nu}_{i,\alpha,\beta}\in
  \name{\bF}_{\name{h}_{\alpha,\beta}} .\mbox{ ''}\]
\item $\forces_\bbQ$`` $\langle\name{\eta}_\alpha,
  \name{\nu}_{i,\alpha,\beta}: i<k,\ \alpha<\beta<\lambda\rangle$ is
  quasi$^-$ independent.''
\item $\forces_\bbQ$`` $\name{\nu}_{0,\alpha,\beta},\ldots,
\name{\nu}_{k-1,\alpha,\beta}, (\eta_\alpha+\eta_\beta+
\name{\nu}_{0,\alpha,\beta}), \ldots,(\eta_\alpha+\eta_\beta+ 
\name{\nu}_{k-1,\alpha,\beta})$ are distinct elements of  
$\big(\name{\eta}_\alpha+ \bigcup\limits_{m<\omega} \name{\bF}_m\big)
\cap\big(\name{\eta}_\beta+ \bigcup\limits_{m<\omega} \name{\bF}_m\big)$. ''  
  \end{enumerate}
\end{lemma}

\begin{proof}
Should be clear.  
\end{proof}

\begin{lemma}
  \label{5.9cp}
  Let $p=(w,M,\bar{r},n,\bar{\Upsilon},\bar{V},h)\in D^2_1\subseteq
  \bbQ$ (cf. \ref{5.3cp}(iii)) and 
  $a_\ell,b_\ell\in\bbH$ and $U_\ell,W_\ell\in\cU$ (for $\ell<4$) be such
  that the following conditions are satisfied.
\begin{enumerate}
\item[$(\circledast)_1$] $U_\ell\in \{U_\alpha(n):\alpha\in w\}$, $W_\ell\in
  \{W_{i, \alpha,\beta}:i<k,\ (\alpha,\beta)\in w^{\langle 2\rangle}\}$ (for
  $\ell<4$).
\item[$(\circledast)_2$] $(U_\ell+W_\ell)\cap
  (U_{\ell'}+W_{\ell'})=\emptyset$ for $\ell<\ell'<4$.
\item[$(\circledast)_3$] $a_\ell\in U_\ell$ and $b_\ell\in W_\ell$ and
  $a_\ell+b_\ell \in \bigcup\limits_{m<M} F(p,m)$ for $\ell<4$.
\item[$(\circledast)_4$] $(a_0+b_0)+(a_1+b_1)=(a_2+b_2)+(a_3+b_3)$.
\end{enumerate}
Then for some $(\alpha,\beta)\in w^{\langle 2\rangle}$ and distinct $i,j<k$
one of the following three conditions holds.
\begin{enumerate}
\item[(A)] $\big\{\{U_0+W_0,U_1+W_1\}, \{U_2+W_2, U_3+W_3\}\big\}=$\\
  $\big\{\{U_\alpha(n)+W_{i,\alpha,\beta},U_\beta(n)+W_{i,\alpha,\beta}\},
  \{U_\alpha(n)+W_{j,\alpha,\beta},U_\beta(n)+W_{j,\alpha,\beta}\}\big\}$. 
\item[(B)] $\big\{\{U_0+W_0,U_1+W_1\}, \{U_2+W_2, U_3+W_3\}\big\}=$\\
  $\big\{\{U_\alpha(n)+W_{i,\alpha,\beta},U_\alpha(n)+W_{j,\alpha,\beta}\},
  \{U_\beta(n)+W_{i,\alpha,\beta},U_\beta(n)+W_{j,\alpha,\beta}\}\big\}$. 
\item[(C)] $\big\{\{U_0+W_0,U_1+W_1\}, \{U_2+W_2, U_3+W_3\}\big\}=$\\
  $\big\{\{U_\alpha(n)+W_{i,\alpha,\beta},U_\beta(n)+W_{j,\alpha,\beta}\},
  \{U_\alpha(n)+W_{j,\alpha,\beta},U_\beta(n)+W_{i,\alpha,\beta}\}\big\}$. 
\end{enumerate}
\end{lemma}

\begin{proof}
The arguments here are very similar to those in Lemma \ref{addin}. Note that
the assumption $(\circledast)_2$ here is slightly stronger than there (to
compensate for weaker qifs). Also our conclusion here is arguably weaker,
but this is a necessity caused by the fact that $a-b=a+b$ in $\bbH$.

For $\ell<4$ let $U_\ell^-, U_\ell^{--}$ and $V_\ell,Q_\ell$ be such that 
  \begin{itemize}
  \item if $U_\ell=U_\alpha(n)$ then $U_\ell^-=U_\alpha(n-1)$, $U_\ell^{--}=
    U_\alpha(n-2)$, 
  \item if $W_\ell=W_{i,\alpha,\beta}$ then $V_\ell=V_{i,\alpha,\beta}$,
    $Q_\ell=Q_{i,\alpha,\beta}$.    
  \end{itemize}
Using steps as in \ref{addin}, one can show that
\begin{enumerate}
\item[(*)]  For every $W,U$ we have
  \[|\{\ell<4:W_\ell=W\}|<3\quad\mbox{ and }\quad |\{\ell<4:U_\ell=U\}|<3.\] 
\end{enumerate}

Now,  since $p\in D^2_1$, it follows from our assumption $(\circledast)_3$ that
\begin{enumerate}
\item[(**)] for each $\ell<4$, for some $\alpha=\alpha(\ell),\beta=\beta(\ell)$, and 
$i=i(\ell)$ we have $U_\ell=U_\alpha(n)$ and $W_\ell=W_{i,\alpha,\beta}$.
\end{enumerate}
By assumption $(\circledast)_4$ we know that
\[0\in U_0+U_1+U_2+U_3+W_0+W_1+W_2+W_3.\]

If all of $U_i$'s are distinct, then $0\in U_0^-+U_1^-+U_2^-+U_3^-+X$, where
$X=\{0\}$ or $X=W_i+W_j$ for some $i<j<4$ with $W_i\neq W_j$ or
$X=W_0+W_1+W_2+W_3$ with all $W_i$'s distinct (remember (*)). This
contradicts \ref{5.2cp}(A)$(\otimes)_4$. Similarly if all $W_i$'s are
distinct.

So suppose $|\{U_0,U_1,U_2,U_3\}|=3$. Then for some $\ell<\ell'<4$,
$U_\ell\neq U_{\ell'}$ and 
\[0\in U_\ell^-+U_{\ell'}^-+W_0+W_1+W_2+W_3\subseteq U_\ell^{--} +
U_{\ell'}^{--}+X,\] 
where $X=\{0\}$ or $X=W_i+W_j$ for some $i<j<4$ with $W_i\neq W_j$ or
$X=W_0+W_1+W_2+W_3$ with all $W_i$'s distinct (remember (*)). This
again contradicts \ref{5.2cp}(A)$(\otimes)_4$. Similarly if
$|\{W_0,W_1,W_2,W_3\}|=3$. 

Consequently, $|\{U_0,U_1,U_2,U_3\}|=2=|\{W_0,W_1,W_2,W_3\}|$. Moreover for
some distinct $\alpha,\beta\in w$ we have 
\[|\{\ell<4:U_\ell=U_\alpha(n)\}|=|\{\ell<4:U_\ell=U_\beta(n)\}|=2\]
and for some $(i,\gamma,\delta)\neq (j,\vare,\zeta)$ we have 
\[|\{\ell<4:W_\ell=W_{i,\gamma,\delta}\}|=|\{\ell<4:W_\ell=W_{j,\vare,\zeta}\}|=2.\]
Now we consider all possible configurations .
\medskip

\noindent {\sc Case 1}\quad $U_0=U_1$, $U_2=U_3$, say they are respectively 
$U_\alpha(n)$ and $U_\beta(n)$.\\
Necessarily $W_0\neq W_1$ and $W_2\neq W_3$.

If $W_0=W_2$, $W_1=W_3$ then recalling (**) above, we also get 
$\{\gamma,\delta\}= \{\vare,\zeta\}=\{\alpha,\beta\}$ and (possibly after 
interchanging $i$ and $j$) $W_0= W_{i,\alpha,\beta}$, 
$W_1=W_{j,\alpha,\beta}$. This gives conclusion (B).

If $W_0=W_3$, $W_1=W_2$ then again by (**) above, we get $\{\gamma,\delta\}=
\{\vare,\zeta\}=\{\alpha,\beta\}$ and (possibly after interchanging $i$ and
$j$) $W_0= W_{i,\alpha,\beta}$,   $W_1=W_{j,\alpha,\beta}$. This also gives
conclusion (B). 
\medskip

\noindent {\sc Case 2}\quad $U_0=U_2$, $U_1=U_3$, say they are respectively
$U_\alpha(n)$ and $U_\beta(n)$.\\
Necessarily $W_0\neq W_2$ and $W_1\neq W_3$.

If $W_0=W_1$, $W_2=W_3$ then recalling (**) above, we also get 
$\{\gamma,\delta\}= \{\vare,\zeta\}=\{\alpha,\beta\}$. After possibly 
interchanging $i$ and $j$,  $W_0= W_{i,\alpha,\beta}$,
$W_2=W_{j,\alpha,\beta}$ and we get conclusion (A). 

If $W_0=W_3$, $W_1=W_2$ then again by (**) , we have $\{\gamma,\delta\}=
\{\vare,\zeta\}=\{\alpha,\beta\}$. After possibly interchanging $i$ and $j$,
$W_0= W_{i,\alpha,\beta}$,   $W_1=W_{j,\alpha,\beta}$. This leads to
conclusion (C).  
\medskip

\noindent {\sc Case 3}\quad $U_0=U_3$, $U_1=U_2$, say they are respectively
$U_\alpha(n)$ and $U_\beta(n)$.\\
Necessarily $W_0\neq W_3$ and $W_1\neq W_2$.

If $W_0=W_1$, $W_2=W_3$ then like above we get $\{\gamma,\delta\}=
\{\vare,\zeta\}=\{\alpha,\beta\}$. After possibly  
interchanging $i$ and $j$,  $W_0= W_{i,\alpha,\beta}$,
$W_2=W_{j,\alpha,\beta}$ and we get conclusion (A). 

If $W_0=W_2$, $W_1=W_3$ then we also have $\{\gamma,\delta\}=
\{\vare,\zeta\}=\{\alpha,\beta\}$ and after possibly interchanging $i$ and
$j$, $W_0= W_{i,\alpha,\beta}$,   $W_1=W_{j,\alpha,\beta}$. This leads to
conclusion (C).  
\end{proof}

\begin{lemma}
  \label{5.10cp}
  Let $p=(w,M,\bar{r},n,\bar{\Upsilon},\bar{V},h)\in D^3_1$ and 
  $\bX\subseteq \bbH$, $|\bX|\geq 5$. Suppose that $a_i(x,y),b_i(x,y),
  U_i(x,y)$ and $W_i(x,y)$ for $x,y\in \bX$, $x\neq y$ and $i<k$ satisfy the
  following demands (i)--(iv) (for all $x\neq y$, $i\neq i'$).
\begin{enumerate}
\item[(i)] $U_i(x,y)\in \{U_\alpha(n):\alpha\in w\}$, $W_i(x,y)\in
  \{W_{j, \alpha,\beta}:j<k,\ (\alpha,\beta)\in w^{\langle 2\rangle}\}$.
\item[(ii)]
  \begin{itemize}
  \item $\big(U_i(x,y)+W_i(x,y) \big)\cap \big (U_i(y,x)+W_i(y,x)
    \big)=\emptyset$, 
  \item $\big (U_i(x,y)+W_i(x,y) \big)\cap \big (U_{i'}(x,y)+W_{i'}(x,y) 
    \big)=\emptyset$,
  \item $\big (U_i(x,y)+W_i(x,y) \big)\cap \big (U_{i'}(y,x)+W_{i'}(y,x) 
    \big)=\emptyset$.        
  \end{itemize}
\item[(iii)] $a_i(x,y)\in U_i(x,y)$ and $b_i(x,y)\in W_i(x,y)$, and

  $a_i(x,y)+b_i(x,y)\in \bigcup\limits_{m<M} F(p,m)$.
\item[(iv)] $x+y=\big (a_i(x,y)+b_i(x,y) \big)+\big (a_i(y,x)+b_i(y,x)
  \big)$. 
\end{enumerate}
Then
\begin{enumerate}
\item $\bX+\bX\subseteq \bigcup\big\{U_\alpha(n-2)+U_\beta(n-2): \alpha,\beta \in 
  w\big\}$.
\item If $(x,y)\in \bX^{\langle 2\rangle}$ and $x+y\in
  U_\alpha(n-2)+U_\beta(n-2)$, $\alpha,\beta\in w$, then $\alpha\neq\beta$
  and for each $i<k$ we have $a_i(x,y)+b_i(x,y), a_i(y,x)+b_i(y,x)\in
  F(p,h(\alpha,\beta))$. 
\end{enumerate}
\end{lemma}

\begin{proof}
(1)\quad  Fix $x,y\in\bX$, $x\neq y$, for a moment.

Let $i\ne i'$, $i,i'<k$. We may apply Lemma \ref{5.9cp} for $U_i(x,y)$,
$W_i(x,y)$, $U_i(y,x)$, $W_i(y,x)$, $a_i(x,y)$, $b_i(x,y)$, $a_i(y,x)$,
$b_i(y,x)$ here as $U_0,W_0,U_1,W_1,a_0,b_0,a_1,b_1$ there and for similar
objects with $i'$ in place of $i$ as $U_2,W_2,U_3,W_3,a_2,b_2,a_3,b_3$
there. This will produce distinct $\alpha=\alpha(x,y,i,i'), \beta=
\beta(x,y,i,i')\in w$ and distinct $j=j(x,y,i,i'),j'=j'(x,y,i,i')<k$ such
that 
\[\begin{array}{l}
\mbox{either }(A)^{\alpha,\beta,j,j'}_{x,y,i,i'}:\\
\big\{\{U_i(x,y)+W_i(x,y),U_i(y,x)+W_i(y,x)\}, \{U_{i'}(x,y)+W_{i'}(x,y),
 U_{i'}(y,x)+W_{i'}(y,x)\}\big\}=\\ 
\big\{\{U_\alpha(n)+W_{j,\alpha,\beta},U_\beta(n)+W_{j,\alpha,\beta}\},
  \{U_\alpha(n)+W_{j',\alpha,\beta},U_\beta(n)+W_{j',\alpha,\beta}\}\big\},\\ 
\mbox{or }(B)^{\alpha,\beta,j,j'}_{x,y,i,i'}:\\
\big\{\{U_i(x,y)+W_i(x,y),U_i(y,x)+W_i(y,x)\}, \{U_{i'}(x,y)+W_{i'}(x,y),
 U_{i'}(y,x)+W_{i'}(y,x)\}\big\}=\\ 
\big\{\{U_\alpha(n)+W_{j,\alpha,\beta},U_\alpha(n)+W_{j',\alpha,\beta}\},
  \{U_\beta(n)+W_{j,\alpha,\beta},U_\beta(n)+W_{j',\alpha,\beta}\}\big\},\\ 
\mbox{or }(C)^{\alpha,\beta,j,j'}_{x,y,i,i'}:\\
\big\{\{U_i(x,y)+W_i(x,y),U_i(y,x)+W_i(y,x)\}, \{U_{i'}(x,y)+W_{i'}(x,y),
 U_{i'}(y,x)+W_{i'}(y,x)\}\big\}=\\ 
\big\{\{U_\alpha(n)+W_{j,\alpha,\beta},U_\beta(n)+W_{j',\alpha,\beta}\},
  \{U_\alpha(n)+W_{j',\alpha,\beta},U_\beta(n)+W_{j,\alpha,\beta}\}\big\}. 
\end{array}\]
Plainly,
\begin{enumerate}
\item[$(\odot)^{x,y}_1$] if for some $i\neq i'$ and $\alpha,\beta,j,j'$ the
  clause  $(A)^{\alpha,\beta,j,j'}_{x,y,i,i'}$ holds true,\\
  then $x+y\in U_\alpha(n-2)+U_\beta(n-2)$. 
\end{enumerate}
It should be also clear that for each $x\neq y$ and distinct $i,i',i''$,
\begin{enumerate}
\item[$(\odot)_2$] if $(B)^{\alpha,\beta,j,j'}_{x,y,i,i'}$ holds true, then
  also $(B)^{\alpha,\beta,j,j'}_{x,y,i,i''}$ holds true,
\end{enumerate}
and 
\begin{enumerate}
\item[$(\odot)_3$] if $(C)^{\alpha,\beta,j,j'}_{x,y,i,i'}$ holds true, then
  also $(C)^{\alpha,\beta,j,j'}_{x,y,i,i''}$ holds true,
\end{enumerate}
Consequently, if $k\geq 3$ then by argument similar to \ref{getdiff} (case
$k\geq 3$) for any $x\neq y$ from $\bX$ neither of possibilities
$(B)^{\alpha,\beta,j,j'}_{x,y,i,i'}$ nor
$(C)^{\alpha,\beta,j,j'}_{x,y,i,i'}$ can hold. Therefore we may easily
finish the proof of Lemma \ref{5.10cp} (when $k\geq 3$).
\medskip

So assume $k=2$. For each $x\neq y$ from $\bX$ we fix $\alpha=\alpha(x,y)$
and $\beta=\beta(x,y)$ such that either $(A)^{\alpha,\beta,0,1}_{x,y,0,1}$
or $(B)^{\alpha,\beta,0,1}_{x,y,0,1}$ or
$(C)^{\alpha,\beta,0,1}_{x,y,0,1}$. Let $\chi(x,y)=\chi(y,x)\in \{A,B,C\}$
and $\theta(x,y)=\theta(y,x)\in [w]^{\textstyle 2}$ be such that
$\big(\chi(x,y)\big)^{\theta(x,y),0,1}_{x,y,0,1}$ holds true. 

\begin{claim}
\label{odot4} If $x,y,z\in\bX$ are distinct and $\chi(x,y)=\chi(y,z)=A$, 
  then $\chi(x,z)=A$. 
\end{claim}

\begin{proof}[Proof of the Claim]
Let $\chi(x,y)=A=\chi(y,z)$ and $\theta(x,y)=\{\alpha,\beta\}$,
$\theta(y,z)=\{\gamma,\delta\}$. Assume towards contradiction that
$\chi(x,z)\in\{B,C\}$ and let $\theta(x,z)=\{\xi,\zeta\}$. Then for some
$\xi',\zeta'\in \{\xi,\zeta\}$ we have 
\[x+z\in U_{\xi'}(n)+U_{\zeta'}(n) +W_{0,\xi,\zeta}+W_{1,\xi,\zeta}.\]

\noindent $\bullet$\quad If $|\{\alpha,\beta\}\cap\{\gamma,\delta\}| =1$,
say $\alpha=\gamma$, $\beta\neq\delta$, and $\{\xi,\zeta\}=\{\xi',\zeta'\}
=\{\beta,\delta\}$, then
\[x+z\in  U_\alpha(n)+U_\alpha(n)+U_\beta(n)+U_\delta(n)+
  W_{i,\alpha,\beta} +   W_{i,\alpha,\beta} +   W_{j,\alpha,\delta} +
  W_{j,\alpha,\delta}\]
but also $x+z\in U_\beta(n)+U_\delta(n) + W_{0,\beta,\delta} +
W_{1,\beta,\delta}$. Consequently,
\[\begin{array}{ll}
  0\in &\big((U_\alpha(n)+U_\alpha(n)) + U_\beta(n)\big)+\big(U_\beta(n) +
         (U_\delta(n)+U_\delta(n))\big) +\\
   &\big((W_{i,\alpha,\beta}+
         W_{i,\alpha,\beta}) + W_{0,\beta,\delta} \big)+
\big((W_{j,\alpha,\delta}+W_{j,\alpha,\delta}) +W_{1,\beta,\delta}\big)
     \subseteq\\
&\qquad U_\beta(n-1) +U_\beta(n-1) +V_{0,\beta,\delta} +V_{1,\beta,\delta}
                   \subseteq Q_{0,\beta,\delta} +Q_{1,\beta,\delta}.
\end{array}\]
This immediately contradicts \ref{5.2cp}(A)$(\otimes)_4$. 

\noindent $\bullet$\quad If $|\{\alpha,\beta\}\cap\{\gamma,\delta\}| =1$,
say $\alpha=\gamma$, $\beta\neq\delta$, and $\xi'\neq\zeta'$,
$|\{\xi',\zeta'\}\cap \{\beta,\delta\}|=1$, say $\xi'=\beta$, then by
similar considerations we arrive to
\[\begin{array}{ll}
  0\in &\big((U_\alpha(n)+U_\alpha(n)) + U_\delta(n)\big)+\big((U_\beta(n) +
        U_\beta(n))+U_{\zeta'}(n)\big) +\\
   &\big((W_{i,\alpha,\beta}+
         W_{i,\alpha,\beta}) + W_{0,\beta,\zeta'} \big)+
\big((W_{j,\alpha,\delta}+W_{j,\alpha,\delta}) +W_{1,\beta,\zeta'}\big) 
     \subseteq\\
&\qquad U_\delta(n-1) +U_{\zeta'}(n-1) +V_{0,\beta,\zeta'}
  +V_{1,\beta,\zeta'}.
  \end{array}\]
In our case necessarily $\delta\neq\zeta'$ so we easily get contradiction
with \ref{5.2cp}(A)$(\otimes)_4$.  

\noindent $\bullet$\quad If $|\{\alpha,\beta\}\cap\{\gamma,\delta\}| =1$,
say $\alpha=\gamma$, $\beta\neq\delta$, and $\xi'\neq\zeta'$ and 
$\{\xi',\zeta'\}\cap \{\beta,\delta\}=\emptyset$, then
\[\begin{array}{ll}
  0\in &\big((U_\alpha(n)+U_\alpha(n)) + U_\beta(n)\big)+U_\delta(n) +
        U_\xi(n)+U_\zeta(n) +\\
   &\big((W_{i,\alpha,\beta}+
         W_{i,\alpha,\beta}) + W_{0,\xi,\zeta} \big)+
\big((W_{j,\alpha,\delta}+W_{j,\alpha,\delta}) +W_{1,\xi,\zeta}\big) 
     \subseteq\\
&\qquad U_\beta(n-1) +U_\delta(n) +U_\xi(n)+U_\zeta(n)+ V_{0,\xi,\zeta}
  +V_{1,\xi,\zeta},
  \end{array}\]
and $\beta,\delta,\xi,\zeta$ are all pairwise distinct. This again
contradicts \ref{5.2cp}(A)$(\otimes)_4$.  

\noindent $\bullet$\quad If $|\{\alpha,\beta\}\cap\{\gamma,\delta\}| =1$,
say $\alpha=\gamma$, $\beta\neq\delta$, and $\xi'=\zeta'$, then
\[\begin{array}{ll}
  0\in &\big((U_\alpha(n)+U_\alpha(n)) + U_\beta(n)\big)+\big((U_{\xi'}(n) +
        U_{\xi'}(n))+U_\delta(n)\big) +\\
   &\big((W_{i,\alpha,\beta}+
         W_{i,\alpha,\beta}) + W_{0,\xi,\zeta} \big)+
\big((W_{j,\alpha,\delta}+W_{j,\alpha,\delta}) +W_{1,\xi,\zeta}\big) 
     \subseteq\\
&\qquad U_\beta(n-1) +U_\delta(n-1) +V_{0,\xi,\zeta}
  +V_{1,\xi,\zeta},
  \end{array}\]
and $\beta\neq\delta$. Again contradiction with \ref{5.2cp}(A)$(\otimes)_4$.   

\noindent $\bullet$\quad If $\{\alpha,\beta\}=\{\gamma,\delta\}$, then we
arrive to
\[\begin{array}{ll}
  0\in &\big((U_\alpha(n)+U_\alpha(n)) + U_{\xi'}(n)\big)+\big((U_\beta (n)
         + U_\beta(n))+U_{\zeta'}(n)\big) +\\
   &\big((W_{i,\alpha,\beta}+W_{i,\alpha,\beta}) + W_{0,\xi,\zeta} \big)+ 
\big((W_{j,\alpha,\beta}+W_{j,\alpha,\beta}) +W_{1,\xi,\zeta}\big) 
     \subseteq\\
&\qquad U_{\xi'}(n-1) +U_{\zeta'}(n-1) +V_{0,\xi,\zeta}+V_{1,\xi,\zeta}. 
  \end{array}\]
Considering cases $\xi'=\zeta'$  and $\xi'\neq \zeta'$ separately we easily
get a contradiction with \ref{5.2cp}(A)$(\otimes)_4$.   

\noindent $\bullet$\quad If $\{\alpha,\beta\}\cap \{\gamma,\delta\}=
\emptyset$, then 
\[\begin{array}{ll}
    0\in &\big((W_{i,\alpha,\beta}+W_{i,\alpha,\beta})+U_\alpha(n) \big)+ 
           \big((W_{j,\gamma,\delta}+W_{j,\gamma,\delta}) + U_\beta(n)\big)
           +\\
&U_\gamma (n)+  U_\delta(n)+ U_{\xi'}(n)+U_{\zeta'}(n)+ W_{0,\xi,\zeta}
           +W_{1,\xi,\zeta} 
     \subseteq\\
&U_\alpha(n{-}1)+U_\beta(n{-}1) + U_\gamma(n)+U_\delta(n)+ U_{\xi'}(n) 
                   +U_{\zeta'}(n) +W_{0,\xi,\zeta}+W_{1,\xi,\zeta}.  
  \end{array}\]
If $\xi'=\zeta'$ then this gives
\[0\in U_\alpha(n{-}1)+U_\beta(n{-}1) + U_\gamma(n)+U_\delta(n-1)+
  W_{0,\xi,\zeta}+W_{1,\xi,\zeta},\]
a contradiction. So $\xi'\neq \zeta'$ and we ask what is the intersection
$\{\xi',\zeta'\}\cap \{\alpha,\beta,\gamma,\delta\}$. In each possible case
we also get a contradiction.   
\end{proof}

\begin{claim}
  \label{odot5}
 If $\chi(x,y)=A$ and $z\in \bX\setminus \{x,y\}$, then
  either $\chi(x,z)\neq A$ or $\theta(x,z)\neq \theta (x,y)$. 
\end{claim}

\begin{proof}[Proof of the Claim]
Suppose $\chi(x,y)=\chi(x,z)=A$ and $\theta(x,y)= \theta(x,z)=
\{\alpha,\beta\}$. By \ref{odot4} we know that $\chi(y,z)=A$. Hence for some
$\xi\neq \zeta$ and $i<2$ we have
\[y+z\in U_\xi(n)+U_\zeta(n)+W_{i,\xi,\zeta}+W_{i,\xi,\zeta}.\]
Also,
\[y+z=y+x+x+z\in U_\alpha(n)+U_\beta(n)+W_{0,\alpha,\beta} +
  W_{0,\alpha,\beta} + U_\alpha(n)+U_\beta(n)+W_{0,\alpha,\beta} +
  W_{0,\alpha,\beta}.\] 
Hence $0\in U_\xi(n-1)+U_\zeta(n-1) +V_{i,\xi,\zeta}+V_{i,\xi,\zeta}$, and
we get contradiction as usual.  
\end{proof}

\begin{claim}
  \label{odot6}
   $\chi(x,y)\neq B$ for any distinct $x,y\in\bX$.
 \end{claim}

\begin{proof}[Proof of the Claim]
 Suppose $\chi(x,y)=B$, $\theta(x,y)=\{\alpha,\beta\}$.  By \ref{odot5}
we may choose $z\in\bX\setminus \{x,y\}$ such that
\begin{enumerate}
\item[$(\star)_z$] \quad $\big(\chi(x,z),\theta(x,z)\big)\neq
  \big(A,\{\alpha,\beta\}\big) \neq  \big(\chi(y,z), \theta(y,z)\big)$.
\end{enumerate}
[Why is it possible? First take $t\notin\{x,y\}$ and ask if it has the
property described in $(\star)_t$. If not, then $\theta(x,t)=\theta(y,t)
=\{\alpha,\beta\}$ and either $\chi(x,t)=A$ or $\chi(y,t)=A$. Say the former
holds true. Pick $u\in\bX\setminus\{x,y,t\}$ and ask if this element has the
property $(\star)_u$. By Claim \ref{odot5} we have
\[(\chi(x,u),\theta(x,u))\neq
  (\chi(x,t),\theta(x,t))=(A,\{\alpha,\beta\}),\]
so if $(\star)_u$ fails this can be only because $ (\chi(y,u),\theta(y,u))
=(A,\{\alpha,\beta\})$. Taking $z\in \bX\setminus\{x,y,t,u\}$ and applying
Claim \ref{odot5} twice (with $\{x,t\}$ and $\{y,u\}$) we immediately see
that $(\star)_z$ holds true.]

By Claim \ref{odot4} we know that either $\chi(x,z)\neq
A$ or $\chi(y,z)\neq A$; by the symmetry we may assume $\chi(x,z)\neq
A$. Now we consider the other possibilities for the value of $\chi(x,z)$.\\ 
(i)\quad \underline{If $\chi(x,z)=B$ and $\theta(x,z)=\{\alpha,\beta\}$},
then
\[x+y, x+z\in 
  U_\alpha(n)+U_\alpha(n)+W_{0,\alpha,\beta}+W_{1,\alpha,\beta}.\]
Hence $y+z\in U_\alpha(n-1)+U_\alpha(n-1)+ U_\alpha(n)+U_\alpha(n)$. Also, 
for some $\xi',\zeta'\in\theta(y,z)=\{\xi,\zeta\}$ and $i,j<2$ we have   
\[y+z \in U_{\xi'}(n)+U_{\zeta'}(n)+W_{i,\xi,\zeta}+W_{j,\xi,\zeta},\]
where either $\xi'\neq \zeta'$ or $i\neq j$. Thus
\[0\in   U_\alpha(n-1)+U_\alpha(n-1)+U_\alpha(n)+U_\alpha(n) +U_{\xi'}(n)
  +U_{\zeta'}(n) +W_{i,\xi,\zeta} +W_{j,\xi,\zeta}\stackrel{\rm def}{=} Y.\] 
If $\xi'=\zeta'$ then $i\neq j$ and
\[Y\subseteq U_\alpha(n-1)+U_\alpha(n-1)+U_\alpha(n-1)+U_\alpha(n)
  +W_{0,\xi,\zeta} +W_{1,\xi,\zeta} \subseteq
  Q_{0,\xi,\zeta}+Q_{1,\xi,\zeta},\]
and we get a contradiction with \ref{5.2cp}(A)$(\otimes)_4$.   If
$\xi'\neq\zeta'$ then
\[Y\subseteq U_{\xi'}(n-2)+U_{\zeta'}(n-1)+W_{i,\xi,\zeta}
  +W_{j,\xi,\zeta}\]
and regardless of $i$ being equal to $j$ or not, we may get a contradiction
too. \\
(ii)\quad \underline{If $\chi(x,z)=B$ and $\theta(x,z)=\{\gamma,\delta\}\neq
  \{\alpha, \beta\}$}, then  
\[\begin{array}{l}
x+z\in U_\gamma(n)+U_\gamma(n)+W_{0,\gamma,\delta}+W_{1,\gamma,\delta}\quad
    \mbox{  and}\\ 
    x+y\in U_\alpha(n)+U_\alpha(n)+W_{0,\alpha,\beta}+W_{1,\alpha,\beta}.
  \end{array}    \]
Hence $y+z\in V_{0,\gamma,\delta}+W_{1,\gamma,\delta}+V_{0,\alpha,\beta}+
W_{1,\alpha,\beta}$. Like before, for some $\xi',\zeta'\in\theta(y,z)
=\{\xi,\zeta\}$ and $i,j<2$ we have   
\[y+z \in U_{\xi'}(n)+U_{\zeta'}(n)+W_{i,\xi,\zeta}+W_{j,\xi,\zeta},\]
where either $\xi'\neq \zeta'$ or $i\neq j$. Since  $\{V_{0,\gamma,\delta},
V_{1,\gamma,\delta}\} \cap \{V_{0,\alpha,\beta},V_{1,\alpha,\beta}\}=
\emptyset$, like before we get a contradiction with
\ref{5.2cp}(A)$(\otimes)_4$.\\  
(iii)\quad \underline{If $\chi(x,z)=C$ and $\theta(x,z)=\{\alpha,\beta\}$}, 
then  
\[\begin{array}{l}
   x+y\in  U_\alpha(n)+U_\alpha(n)+W_{0,\alpha,\beta}+W_{1,\alpha,\beta}\\
   x+z\in  U_\alpha(n)+U_\beta(n)+W_{0,\alpha,\beta}+W_{1,\alpha,\beta}.
  \end{array}\]
Also, $y+z\in U_{\xi'}(n)+U_{\zeta'}(n)+W_{i,\xi,\zeta}+W_{j,\xi,\zeta}$,
where $\xi',\zeta'\in \theta(y,z)=\{\xi,\zeta\}$, $i,j<2$ and either
$\xi'\neq\zeta'$ or $i\neq j$. We consider 2 subcases now. \\
If $i=j$ then ($\xi'\neq\zeta'$ and) $\chi(y,z)=A$ so by the choice of $z$
at the beginning we know that $\theta(y,z)\neq\{\alpha,\beta\}$. So we
arrive to 
\[\begin{array}{ll}
 0\in &\big((U_\alpha(n)+  U_\alpha(n))+ U_\alpha(n)\big)+
\big((W_{i,\xi,\zeta}+W_{i,\xi,\zeta})+   U_\beta(n)\big)+\\
    &\big((W_{0,\alpha,\beta}+W_{0,\alpha,\beta})+U_\xi(n) \big)+ 
    \big((W_{1,\alpha,\beta}+W_{1,\alpha,\beta}) + U_\zeta(n)\big) 
           \subseteq\\
&U_\alpha(n-1)+U_\beta(n-1) + U_\xi(n-1) +U_\zeta(n-1)
  \end{array}\]
and since $\{\xi,\zeta\}\neq\{\alpha,\beta\}$ a contradiction follows. \\
If $i\neq j$ then we get 
\[\begin{array}{ll}
 0\in &\big((U_\alpha(n)+  U_\alpha(n))+ U_\alpha(n)\big)+
\big((W_{0,\alpha,\beta}+W_{0,\alpha,\beta})+   U_\beta(n)\big)+\\
    &\big((W_{1,\alpha,\beta}+W_{1,\alpha,\beta})+U_{\xi'}(n) \big)+ 
U_{\zeta'}(n) +W_{0,\xi,\zeta}+W_{1,\xi,\zeta}   \subseteq\\
&U_\alpha(n-1)+U_\beta(n-1) + U_{\xi'}(n-1) +U_{\zeta'}(n)+W_{0,\xi,\zeta} 
                                                               +W_{1,\xi,\zeta}, 
  \end{array}\]
and again a contradiction.\\
(iv)\quad \underline{If $\chi(x,z)=C$ and $\theta(x,z)=\{\gamma,\delta\}\neq
  \{\alpha, \beta\}$}, then  
\[\begin{array}{l}
x+z\in U_\gamma(n)+U_\delta(n)+W_{0,\gamma,\delta}+W_{1,\gamma,\delta}\quad
    \mbox{  and}\\ 
    x+y\in U_\alpha(n)+U_\alpha(n)+W_{0,\alpha,\beta}+W_{1,\alpha,\beta}, \quad
    \mbox{  and}\\
y+z \in U_{\xi'}(n)+U_{\zeta'}(n)+W_{i,\xi,\zeta}+W_{j,\xi,\zeta}, 
  \end{array}    \]
where $\xi',\zeta'\in \theta(y,z)=\{\xi,\zeta\}$, $i,j<2$ and either
$\xi'\neq\zeta'$ or $i\neq j$. Thus 
\[\begin{array}{ll}
    0\in &\big(U_\gamma(n)+(U_\alpha(n)+U_\alpha(n))\big)+ 
    U_\delta(n)+W_{0,\gamma,\delta}+W_{1,\gamma,\delta} +\\
    &W_{0,\alpha,\beta}+W_{1,\alpha,\beta}+
      U_{\xi'}(n)+U_{\zeta'}(n)+W_{i,\xi,\zeta}+W_{j,\xi,\zeta}\subseteq \\
         &U_\gamma(n-1)+  U_\delta(n)+U_{\xi'}(n)+U_{\zeta'}(n)+\\
   &W_{0,\gamma,\delta}+W_{1,\gamma,\delta} +
    W_{0,\alpha,\beta}+W_{1,\alpha,\beta}+
    W_{i,\xi,\zeta}+W_{j,\xi,\zeta}. 
  \end{array}    \]
Since $W_{0,\gamma,\delta},W_{1,\gamma,\delta},  W_{0,\alpha,\beta}$ and
$W_{1,\alpha,\beta}$ are all distinct we get a contradiction in the usual
manner.  
\end{proof}
\medskip

\begin{claim}
  \label{odot7}
 $\chi(x,y)\neq C$ for any distinct $x,y\in\bX$.
\end{claim}

\begin{proof}[Proof of the Claim]
Suppose towards contradiction $\chi(x,y)=C$ and let
$\theta(x,y)=\{\alpha,\beta\}$. Let $z\in \bX\setminus \{x,y\}$. By
\ref{odot4} we know that either $\chi(x,z)\neq A$ 
or $\chi(y,z)\neq A$; by the symmetry we may assume $\chi(x,z)\neq A$. By
\ref{odot6} we know that $\chi(x,z)\neq B$,  so $\chi(x,z)=C$\\
\underline{If $\theta(x,y)=\theta(x,z)=\{\alpha,\beta\}$}, then $y+z\in
U_\alpha(n-1)+U_\alpha(n) +U_\beta(n-1)+U_\beta(n)$. We know that
$\chi(y,z)\in \{A,C\}$ and in both cases $y+z\in U_\gamma(n)+U_\delta(n)+
W_{i,\gamma,\delta}+ W_{j,\gamma,\delta}$, where $\theta(y,z)=\{\gamma,
\delta\}$ and $i,j<2$. Now we may conclude
\[0\in U_\alpha(n-1)+U_\alpha(n)+U_\beta(n-1)+U_\beta(n) +U_\gamma(n)
  +U_\delta(n) +W_{i,\gamma,\delta}+W_{j,\gamma,\delta}\stackrel{\rm def}{=}
  S.\]
If $i\neq j$ then $S\subseteq U_\gamma(n-2)  +U_\delta(n-2)
+W_{0,\gamma,\delta} +W_{1,\gamma,\delta}$ and an immediate contradiction
with \ref{5.2cp}(A)$(\otimes)_4$ follows. If $i=j$ then $S\subseteq
U_\gamma(n-2)  +U_\delta(n-2)+W_{i,\gamma\delta}+W_{i,\gamma\delta}
\subseteq  U_\gamma(n-3)+U_\delta(n-3)$ and we get a contradiction with
$p\in  D^3_1$. \\
\underline{If $\theta(x,z)=\{\xi,\zeta\}\neq \theta(x,y)=\{\alpha,\beta\}$},
then $\{W_{0,\xi,\zeta},W_{1,\xi,\zeta}\} \cap \{W_{0,\alpha,\beta},W_{1,\alpha,\beta}\}=
\emptyset$ and 
\[y+z\in U_\alpha(n)+U_\beta(n) +W_{0,\alpha,\beta}+W_{1,\alpha,\beta}
  +U_\xi(n)+U_\zeta(n) +W_{0,\xi,\zeta}+W_{1,\xi,\zeta}.\] 
Now, by considerations as before, we get a contradiction with
\ref{5.2cp}(A)$(\otimes)_4$. 
\end{proof}

Therefore,
\begin{enumerate}
\item[$(\odot)$] $\chi(x,y)=A$ for all distinct $x,y\in\bX$.
\end{enumerate}
Hence, if $x\neq y$ are from $\bX$ and $\theta(x,y)=\{\alpha,\beta\}$, then
$x+y\in U_\alpha(n-2)+U_\beta(n-2)$.
\medskip

\noindent (2)\quad Like Lemma \ref{getdiff}, using Lemma \ref{5.4cp}. 
\end{proof}

\begin{lemma}
  \label{gettran}
 Let $p=(w,M,\bar{r},n,\bar{\Upsilon},\bar{V},h)\in D^3_1$ and $\bX\subseteq
 \bbH$, $|\bX|\geq 5$. Suppose that
 \begin{enumerate}
 \item[(a)] $\bX+\bX\subseteq \bigcup\{U_\alpha(n)+U_\beta(n):
   \alpha,\beta\in w\}$, and
 \item[(b)] ${\rm diam}_\rho\big(U_\alpha(n)\big) <\rho(x,y)$ for all
   $\alpha\in w$, $(x,y)\in \bX^{\langle 2\rangle}$.
 \end{enumerate}
 Then there is a $c\in\bbH$ such that
 \[\bX+c\subseteq \bigcup\{U_\alpha(n-1): \alpha\in w\}.\]
\end{lemma}

\begin{proof}
 By assumption (b), if $x,y\in\bX$ are distinct and $x+y\in
 U_\alpha(n)+U_\beta(n)$, $\alpha,\beta\in w$, then $\alpha\neq
 \beta$. Also, if $(x,y)\in \bX^{\langle 2\rangle}$ and $x+y\in
 \big(U_\alpha(n)+U_\beta(n)\big) \cap \big(U_\gamma(n)+U_\delta(n)\big)$,
 then $\{\alpha,\beta\}=\{\gamma,\delta\}$ (by
 \ref{5.2cp}(A)$(\otimes)_4$). Consequently, for each $(x,y)\in \bX^{\langle
   2\rangle}$ we may let $\theta(x,y)$ to be the unique $\{\alpha,\beta\}\in
 [w]^{\textstyle 2}$ such that $x+y\in U_\alpha(n)+U_\beta(n)$.

 \begin{claim}
   \label{cl8}
   \[|\theta(x,y)\cap \theta(x,z)|=1\]
   whenever $x,y,z\in\bX$ are distinct.
 \end{claim}

 \begin{proof}[Proof of the Claim]
Let $\theta(x,y)=\{\alpha,\beta\}$, $\theta(x,z)=\{\gamma,\delta\}$ and
$\theta(y,z)=\{\xi,\zeta\}$. Then
\[y+z\in \Big(U_\alpha(n)+U_\beta(n)+U_\gamma(n)+U_\delta(n)\Big)\cap
  \big(U_\xi(n)+ U_\zeta(n)\big).\] 
Hence $0\in U_\alpha(n)+U_\beta(n)+U_\gamma(n)+U_\delta(n) + U_\xi(n)+
U_\zeta(n)$. Since $\alpha\neq\beta$, $\gamma\neq \delta$ and $\xi\neq\zeta$
we conclude that $\{\alpha,\beta\}\cap\{\gamma,\delta\}\neq \emptyset$
(remember  \ref{5.2cp}(A)$(\otimes)_4$). If we had $\{\alpha,\beta\}=
\{\gamma,\delta\}$, then $0\in U_\xi(n-1)+U_\zeta(n-1)$, a contradiction as
well. Consequently $|\{\alpha,\beta\}\cap\{\gamma,\delta\}|=1$.
\end{proof}

Fix distinct $x_0,y_0,z_0\in \bX$. Let $\theta(x_0,y_0)=\{\alpha_0,
\beta_0\}$, $\theta(x_0,z_0)=\{\gamma_0,\alpha_0\}$ and let $a',a''\in
U_{\alpha_0}(n)$, $b_0\in U_{\beta_0}(n)$, $c_0\in U_{\gamma_0}(n)$ be such
that $x_0+y_0=a'+b_0$ and $x_0+z_0=a''+c_0$.

Let $c=a'+x_0$. We will show that $x+c\in \bigcup\{U_\alpha(n-1):
\alpha\in w\}$ for all $x\in\bX$. To this end, first note that 
\begin{itemize}
\item $x_0+c=x_0+a'+x_0=a'\in U_{\alpha_0}(n)$, 
\item $y_0+c=y_0+a'+x_0=a'+b_0+a'=b_0\in U_{\beta_0}(n)$, 
\item $z_0+c=z_0+a'+x_0=a''+c_0+a'\in U_{\gamma_0}(n)+
  (U_{\alpha_0}(n)+U_{\alpha_0}(n))\subseteq U_{\gamma_0}(n-1)$.   
\end{itemize}
Now suppose $x\in \bX\setminus \{x_0,y_0,z_0\}$. Let $\theta(x,x_0)
=\{\delta,\zeta\}$, $x+x_0=d+e$, $d\in U_\delta(n)$, $e\in  U_\zeta(n)$.
\begin{enumerate}
\item[(*)] $\alpha_0\in \{\delta,\zeta\}$.
\end{enumerate}
Why? By Claim \ref{cl8} we have $|\theta(x_0,x)\cap \theta(x_0,y_0)|=
|\theta(x_0,x)\cap \theta(x_0,z_0)|=1$. Hence if $\alpha_0\notin
\{\delta,\zeta\}$, then $\theta(x,x_0)=\{\beta_0,\gamma_0\}$. Take $x'\in
\bX\setminus \{x_0,y_0,z_0,x\}$ and note that (again by Claim \ref{cl8})
\[\big|\theta(x_0,x')\cap \{\alpha_0,\beta_0\}\big| =
  \big|\theta(x_0,x')\cap \{\alpha_0,\gamma_0\}\big| =
  \big|\theta(x_0,x')\cap \{\gamma_0,\beta_0\}\big|=1,\]
and this is clearly impossible.
\medskip

By symmetry we may assume $\alpha_0=\delta$. But now
\[x+c=x+x_0+a'=(d+a')+e\in U_\zeta(n-1),\]
so we are done.
\end{proof}

\begin{lemma}
  \label{5.11cp}
\[\begin{array}{ll}
\forces_\bbQ&\mbox{`` there is no perfect set $P\subseteq \bbH$ such that}\\
 &\quad  \Big(\forall x,y\in P\Big)\Big(\Big|\big(x+
   \bigcup\limits_{m<\omega} \name{\bF}_m\big) \cap  \big(y+
   \bigcup\limits_{m<\omega} \name{\bF}_m\big) \Big|\geq 2k\Big)\mbox{. ''}  
\end{array}\] 
\end{lemma}

\begin{proof}
Suppose towards contradiction that  $G\subseteq \bbQ$ is generic over
$\bV$ and in $\bV[G]$ the following assertion holds true:
\begin{quotation}
for some perfect set $P\subseteq \bbH$ we have
  \[\Big|\big(x+ \bigcup\limits_{m<\omega} \name{\bF}_m^G\big) \cap 
  \big(y+ \bigcup\limits_{m<\omega} \name{\bF}_m^G\big) \Big|\geq 2k\] 
for  all $x,y\in P$.  
\end{quotation}
Then for any distinct $x,y\in P$ there are $c_0,d_0,\ldots, c_{k-1},d_{k-1}
\in \bigcup\limits_{m<\omega} \name{\bF}_m^G$ such that $x+y=c_i+d_i$ (for
all $i<k$) and $\{c_i,d_i\}\cap \{c_{i'},d_{i'}\}=\emptyset$ (for $i<i'<k$);
remember $x+y=c_i+d_i$ implies that $x+c_i,x+d_i$ are distinct elements of 
$\big(x+ \bigcup\limits_{m<\omega} \name{\bF}_m^G\big) \cap \big(y+
\bigcup\limits_{m<\omega} \name{\bF}_m^G\big)$. 
For $\bar{\ell}=\langle \ell_i:i<k\rangle \subseteq \omega$, $\bar{m}
=\langle m_i:i<k\rangle\subseteq \omega$ and $N<\omega$ let    
\[\begin{array}{r}
Z^N_{\bar{\ell},\bar{m}}= \big\{(x,y)\in P^2:
\mbox{there are } c_i\in \name{\bF}_{\ell_i}^G, d_i\in 
\name{\bF}_{m_i}^G \mbox{ (for $i<k$) such that }\ \\
x+y=c_i+d_i \ \mbox{ and } \ 2^{-N}<\min\big(\rho(c_i,c_j),  
    \rho(d_i,d_j),\rho(c_i,d_j)\big)\ \ \\
\mbox{ for all distinct }i,j<k\quad\big\}.     
\end{array}\] 
Now we continue as in \ref{noperf}, but instead of \ref{gettranslMin} we use
\ref{gettran}. In $(\boxdot)^c_4)$ as there we demand $p_\iota,q_\iota\in
D^3_{n_\iota}$. Also under current assumptions on $\bbH$,
$\bX_\iota+c_\iota=c_\iota-\bX_\iota$, so we have only one case. Otherwise
the same proof works.  
\end{proof}

The following theorem is a consequence of results presented in this
section. 

\begin{theorem}
\label{5.12cp}  
Assume that 
  \begin{enumerate}
\item $(\bbH,+,0)$ is an Abelian perfect Polish group,
\item all elements of $\bbH$ have order at most 2,
\item $2\leq k<\omega$ and 
\item $\vare<\omega_1$ and $\lambda$ is an uncountable
  cardinal such that ${\rm NPr}^\vare(\lambda)$ holds true. 
  \end{enumerate}
Then there is a ccc forcing notion $\bbQ$ of cardinality $\lambda$ such that  
\[\begin{array}{ll}
\forces_\bbQ &\mbox{`` for some $\Sigma^0_2$ subset $B$ of $\bbH$ we
               have:}\\  
&\quad \mbox{ there is a set }X\subseteq \bbH\mbox{ of cardinality }
  \lambda\mbox{ such that }\\
&\qquad\big(\forall x,y\in X\big)\big(\big| (x+B)\cap
  (y+B)\big|\geq 2k\big)\\
&\quad \mbox{ but there is no perfect set }P\subseteq \bbH \mbox{ such that 
  } \\
&\qquad\big(\forall x,y\in P\big)\big(\big| (x+B)\cap (y+B)\big|\geq
  2k\big)\mbox{ ''.}
  \end{array}\]
\end{theorem}

\section{Conclusions and Questions}
Let us recall from the Introduction, that {\em the spectrum of  translation 
  $k$--non-disjointness of a set $A\subseteq \bbH$\/} is    
\[\stnd_k (A)=\stnd_k (A,\bbH) =\{(x,y)\in \bbH\times \bbH: 
  |(A+x)\cap (A+y)|\geq k\}.\] 
By the definition,  $X\times X\subseteq  \stnd_k (A)$ if and only if 
\[\big(\forall x,y\in X\big)\big(\big| (x+A)\cap (y+A)\big|\geq k\big).\]
In particular, there is a perfect square $P\times P$ included in 
$\stnd_k (A)$ if and only if $A$ has a perfect set $P$ of
$k$--overlapping  translations.  

\begin{conclusion}
\label{conclu}
  Assume that
\begin{enumerate}
\item[(a)] $\bbH=(\bbH,0,+)$ is a perfect Abelian Polish group,
\item[(b)] $1<\iota<\omega$ and
\begin{itemize}
\item $k=\iota$ if $\{c\in \bbH:c+c\neq 0\}$ is dense in $\bbH$, and 
\item $k=2\iota$ otherwise,
\end{itemize}
\item[(c)] $\lambda$ is an uncountable cardinal such that ${\rm
    NPr}^\vare(\lambda)$ holds true for some countable ordinal $\vare$,  and  
\item[(d)] $\lambda=\lambda^{\aleph_0} \leq \mu=\mu^{\aleph_0}$.   
\end{enumerate}
Then there is a ccc forcing notion $\bbP^*$ and a $\bbP^*$--name $\name{B}$
for a $\Sigma^0_2$ subset of $\bbH$ such that 
\begin{enumerate}
\item $\forces_{\bbP^*}$ `` $2^{\aleph_0}=\mu$ '',
\item $\forces_{\bbP^*}$`` there is a set $X\subseteq \bbH$ of cardinality 
  $\lambda$ such that $X\times X\subseteq \stnd_k(\name{B})$ '',   but
\item $\forces_{\bbP^*}$`` there is {\bf no} set $X\subseteq \bbH$ of cardinality 
  $\lambda^+$ such that $X\times X\subseteq \stnd_k(\name{B})$ '', and   
\item $\forces_{\bbP^*}$`` there is {\bf no} perfect set $P\subseteq \bbH$
  such that $P\times P\subseteq \stnd_k(\name{B})$ ''.
\end{enumerate}
\end{conclusion}

\begin{proof}
Let us consider the case when (in assumption (b) of the Corollary) the set
$\{c\in\bbH: c+c\neq 0\}$ is dense in $\bbH$. The other case is fully
parallel. So we assume
\begin{itemize}
\item $(\bbH,+,0)$, $\bD$, $\rho,\rho^*$ and $\cU$ are as in Assumption
  \ref{hypno} and Assumption \ref{hypno2},
\item $k,\vare,\lambda,\rksp,\bj,\bk$ and $\mu$ satisfy Assumption
  \ref{hypno2} and assumption (d) of the Corollary. 
\end{itemize}
Let $\bbP$ be the forcing notion discussed in Section 5 (cf Theorem
\ref{mainA}) and let $\bbC_\mu$ be the forcing notion adding $\mu$ Cohen
reals, where conditions are finite functions with domains included in $\mu$
and values $0,1$. 

Let $\bbP^*=\bbP\times \bbC_\mu$.

By standard arguments, $\bbP^*$ is a ccc forcing notion and
$\forces_{\bbP^*} 2^{\aleph_0}=\mu$. Let $\name{B}$ be a $\bbP$--name for
the $\Sigma^0_2$ subset of $\bbH$ added by $\bbP\lessdot \bbP^*$. 

\begin{claim}
  \label{cl9}
\begin{enumerate}
\item[(2)] $\forces_{\bbP^*}$ `` there is a set $X\subseteq \bbH$ of
  cardinality  $\lambda$ such that
  \[\big(\forall x,y\in X\big)\big(\big| (x+\name{B})\cap
    (y+\name{B})\big|\geq k\big)\mbox{ '',}\]
  but
\item[(4)] $\forces_{\bbP^*}$ `` there is {\bf no} perfect set $P\subseteq
  \bbH$ such that
  \[\big(\forall x,y\in P\big)\big(\big| (x+\name{B})\cap
    (y+\name{B})\big|\geq k\big)\mbox{ ''.}\]
\end{enumerate}
\end{claim}

\begin{proof}[Proof of the Claim]
If $H\subseteq \bbC_\mu$ is generic over $\bV$, then in $\bV[H]$ we may look
at the definition of the forcing notion $\bbP$ as all the ingredients still
have the required properties. Identifying $\bB^{\bV}_n$ with
$\bB^{\bV[H]}_n$ we easily see that $\bbP^{\bV}=\bbP^{\bV[H]}$. Hence
$\bbP^*$ is equivalent to the iteration $\bbC_\mu *\bbP$ and consequently
the results of Section 5 give the desired conclusion. 
\end{proof}

\begin{claim}
  \label{cl10}
  \begin{enumerate}
\item[(3)] $\forces_{\bbP^*}$ `` there is {\bf no} set $X\subseteq \bbH$ of
  cardinality  $\lambda^+$ such that 
  \[\big(\forall x,y\in X\big)\big(\big| (x+\name{B})\cap 
    (y+\name{B})\big|\geq k\big)\mbox{ ''.}\]
  \end{enumerate}
\end{claim}

\begin{proof}[Proof of the Claim]
Assume $\lambda<\mu$ (otherwise clear). Suppose towards contradiction that
$G=G_0\times G_1\subseteq \bbP\times \bbC_\mu$ is generic over $\bV$ and in
$\bV[G_0][G_1]$ there are distinct $x_\alpha\in \bbH$ (for
$\alpha<\lambda^+$) such that
\[\big|(x_\alpha+\name{B}^G)\cap (x_\beta+\name{B}^G)\big|\geq k\qquad
  \mbox{ for }\alpha,\beta<\lambda^+.\]
Then in $\bV[G_0]$ we may find a condition $q\in G_1$ and $\bbC_\mu$--names
$\name{x}_\alpha$,  $\alpha<\lambda^+$, for elements of the group $\bbH$
such that
\[q\forces_{\bbC_\mu}\mbox{`` } \name{x}_\alpha\neq \name{x}_\beta\ \mbox{
    and }\ \big|(\name{x}_\alpha+\name{B})\cap (\name{x}_\beta+\name{B})
  \big|\geq k\mbox{ ''}\]
for all $\alpha<\beta<\lambda^+$. Each of the names $\name{x}_\alpha$ is
actually a $\bbC_{A_\alpha}$--name for some countable set $A_\alpha\subseteq
\mu$. Since $\bV[G_0]\models 2^{\aleph_0}=\lambda$, we may choose a set
$I\in [\lambda^+]^{\textstyle \lambda^+}$ and a set $u\subseteq \mu$ such
that the following two demands are satisfied (in $\bV[G_0]$).
\begin{enumerate}
\item[$(\clubsuit)_1$] ${\rm otp} (A_\alpha)={\rm otp}(A_\beta)$ for
  $\alpha,\beta\in I$. 
\item[$(\clubsuit)_2$]  For each $\alpha<\beta$ from $I$, letting
  $\pi_{\alpha,\beta}: A_\alpha\longrightarrow A_\beta$ be the order
  isomorphism, we have
 \[u=A_\alpha\cap A_\beta,\quad \pi_{\alpha,\beta}\rest u = {\rm id}_u\quad
   \mbox{ and }\quad A_\alpha\setminus u \mbox{ is infinite.}\]
\end{enumerate}
Let $u^*=u\cup\dom(q)\subseteq \mu$. Dismissing finitely many elements of
$I$ we may assume that  $A_\alpha\setminus u= A_\alpha\setminus u^*$ for all
$\alpha\in I$.

Let $G_1^*=G_1\cap \bbC_{u^*}$ and let us work in $\bV[G_0][G_1^*]$ for a
moment. Each name $\name{x}_\alpha$ (for $\alpha\in I$) can be thought of as
a $\bbC_{A_\alpha\setminus u^*}$--name now. Let $\xi={\rm
  otp}(A_\alpha\setminus u^*)$ for some (equivalently, all) $\alpha\in I$. 
Since $\bV[G_0][G_1^*]\models 2^{\aleph_0}=\lambda$, we may find $I^*\in 
[I]^{\textstyle \lambda^+}$ and a Borel function $\tau:{}^\xi
2\longrightarrow \bbH$ such that
\begin{enumerate}
\item[$(\clubsuit)_3$] $\forces \name{x}_\alpha= \tau\big (\name{c}_\alpha
  \circ \pi^\alpha\big)$, where $\pi^\alpha:\xi\longrightarrow
  A_\alpha\setminus u^*$ is the order isomorphism and $\name{c}_\alpha$ is (a
  name for) the Cohen real added by $\bbC_{A_\alpha\setminus u^*}$. 
\end{enumerate}
Consequently, if $\alpha\neq\beta$ are from $I^*$, then 
\begin{enumerate}
\item[$(\clubsuit)_4$] $\forces_{\bbC_{A_\alpha\setminus u^*}\times
    \bbC_{A_\beta\setminus u^*}}$`` $\big|(\tau(\name{c}_\alpha\circ
  \pi^\alpha)+\name{B}^{G_0}) \cap (\tau(\name{c}_\beta\circ
  \pi^\beta)+\name{B}^{G_0})\big|\geq k$ and 

  $\tau\big (\name{c}_\alpha\circ \pi^\alpha\big) \neq \tau\big
(\name{c}_\beta\circ \pi^\beta\big)$ '' 
\end{enumerate}
Therefore,
\begin{enumerate}
\item[$(\clubsuit)_5$] if $d_0,d_1\in {}^\xi 2$ are (mutually) Cohen reals
  over $\bV[G_0][G^*_1]$, then 
\[\bV[G_0][G^*_1][d_0,d_1]\models
\big|(\tau(d_0)+\name{B}^{G_0}) \cap (\tau(d_1)+\name{B}^{G_0})\big|\geq
k\mbox{ and }\tau(d_0) \neq \tau(d_1).\]
\end{enumerate}
Take $\alpha\in I$ and note that in $\bV^*=\bV[G_0][G^*_1][G_1\cap
\bbC_{A_\alpha\setminus u^*}]$ there is a perfect set $P\subseteq {}^\xi 2$
of mutually Cohen reals over $\bV[G_0][G^*_1]$. By $(\clubsuit)_5$ we know
\[\bV^*\models \tau\rest P\mbox{ is one-to-one and }
\big|(\tau(x)+\name{B}^{G_0}) \cap (\tau(y)+\name{B}^{G_0})\big|\geq
k\mbox{ for }x,y\in P.\]
By upward absoluteness of $\Sigma^1_3$ sentences we may assert now that 
\[
  \begin{array}{ll}
  \bV[G_0\times G_1]\models     &
\mbox{ there is a perfect set $P^*\subseteq 
                                  \bbH$ such that }\\
 &\big(\forall x,y\in P^*\big)\big( \big|(x+\name{B}^{G_0}) \cap
   (y+\name{B}^{G_0})\big|\geq  k\big).
  \end{array}
\]
This, however, contradicts Claim \ref{cl9}. 
\end{proof}
\end{proof}

\begin{conclusion}
[See {\cite[Proposition 3.3(5)]{RoSh:1138}}]
  \label{biglam}
  Assume that
  \begin{enumerate}
\item $\bbH$ is a perfect Polish group and $B\subseteq \bbH$ is a Borel set,
\item a cardinal $\lambda$ is such that ${\rm Pr}^\vare(\lambda)$ holds true
  for  every $\vare<\omega_1$, and
\item $1<k<\omega$, and
\item there is a set $X\subseteq \bbH$ of   cardinality  $\lambda$ such
  that  $X\times X\subseteq \stnd_k(B)$.
  \end{enumerate}
Then there is a perfect set $P\subseteq \bbH$ such that $P\times P\subseteq
\stnd_k(B)$.
\end{conclusion}

\begin{proof}
Under our assumptions on $\lambda$, if an analytic set $B\subseteq
\can\times\can$ includes a $\lambda$--square, it includes a perfect square
(see \cite[Claim 1.12(1)]{Sh:522}).

The space $\bbH$ is Borel isomorphic with $\can$; let
$f:\bbH\longrightarrow\can$ be a Borel isomorphism and let
$f_2:\bbH\times \bbH\longrightarrow \can\times\can:(x,y)\mapsto
(f(x),f(y))$. Then the set $f_2[\stnd_k(B)]$ is analytic and
$f[X]\times f[X]\subseteq f_2[\stnd_k(B)]$. Consequently there is a perfect
set $P^*\subseteq \can$ such that $P^*\times P^*\subseteq
f_2[\stnd_k(B)]$. We may choose a perfect set $P\subseteq f^{-1}[P^*]
\subseteq \bbH$ -- it will also satisfy $P\times P\subseteq \stnd_k(B)$. 
\end{proof}

Now, in Claim \ref{cl10} we used the upward absoluteness to show
$\forces_{\bbP^*}$`` (3) ''. If the group $\bbH$ is compact and
$B\subseteq \bbH$ is $\Sigma^0_2$, then the set $\stnd_k(B)$ is $\Sigma^0_1$
and hence the assertion in (4) of \ref{conclu} is $\Pi^1_2$, so
also absolute. However, in the case of general $\bbH$ the corresponding assertion
appears to be $\Pi^1_3$ so not so obviously absolute. Its absoluteness could
be establish if we can introduce corresponding rank. (This would be helpful
for natural consequences under MA.)

\begin{problem}
Develop the rank and the results parallel to ${\rm ndrk}_\iota$ and cute
$\mathcal{YZR}$--systems presented in \cite{RoSh:1170} for the case of
general perfect Abelian Polish groups.    
\end{problem}

The forcing notions presented in this article for various Abelian Polish
groups look similar, but the particular group structures may have different
impacts.

\begin{problem}
Is it consistent that for some  perfect Abelian Polish groups $\bbH_0$,
$\bbH_1$ and $2<k<\omega$ and an uncountable cardinal $\lambda$ we have: 
\begin{enumerate}
\item for some Borel set $B_0\subseteq \bbH_0$,
  \begin{enumerate}
\item there is a set $X\subseteq \bbH_0$ of   cardinality  $\lambda$ such
  that  $X\times X\subseteq \stnd_k(B_0,\bbH_0)$ (i.e.,
  $\stnd_k(B_0,\bbH_0)$ includes a $\lambda$--square) , but 
\item there is {\bf no} perfect set $P\subseteq \bbH_0$ such that $P\times
  P\subseteq \stnd_k(B_0,\bbH_0)$ (i.e., $\stnd_k(B_0,\bbH_0)$ does not
  include any perfect square)
\end{enumerate}
and
\item for every Borel set $B\subseteq \bbH_1$, if $\stnd_k(B,\bbH_1)$
  includes a $\lambda$--square, then it includes a perfect square ?
\end{enumerate}
\end{problem}

Considering differences caused by various choices of parameters, it is
natural to ask about the impact of $k$.

\begin{problem}
Is it consistent that for some  perfect Abelian Polish group $\bbH$ and
$2<k<\ell<\omega$ and an uncountable cardinal $\lambda$ the following two
statements are true. 
\begin{enumerate}
\item For some Borel set $B_0\subseteq \bbH$,
\begin{enumerate}
\item there is a set $X\subseteq \bbH$ of   cardinality  $\lambda$ such
  that  $X\times X\subseteq \stnd_\ell(B_0,\bbH)$, but 
\item there is {\bf no} perfect set $P\subseteq \bbH_0$ such that $P\times
  P\subseteq \stnd_\ell(B_0,\bbH_0)$.
\end{enumerate}
\item For every Borel set $B\subseteq \bbH$, if $\stnd_k(B,\bbH)$
  includes a $\lambda$--square, then it includes a perfect square.
\end{enumerate}  
\end{problem}

Of course, the next steps could be to investigate $\stnd_\omega$ and
$\stnd_{\omega_1}$:

\begin{problem}
Let $\bbH$ be a perfect Abelian Polish group. Is it consistent that for some 
Borel set $B\subseteq \bbH$:
\begin{itemize}
\item there is an uncountable set $X\subseteq \bbH$ such that $(B+x)\cap
  (B+y)$ is uncountable for every $x,y\in X$, but 
\item for every perfect set $P\subseteq \bbH$ there are $x,y\in P$ with
  $(B+x)\cap (B+y)$ countable?
\end{itemize}
Similarly if ``uncountable / countable'' are replaced with ``infinite /
finite'', respectively. 
\end{problem}

Let us also remind two other questions related to our results. The first one
calls for a ``dual'' results.

\begin{problem}
Is it consistent to have a Borel set $B\subseteq \bbH$ such that 
\begin{itemize}
\item $B$ has uncountably many pairwise disjoint translations, but
\item there is no perfect of pairwise disjoint translations of $B$ ? 
\end{itemize}
\end{problem}

Assumptions of Conclusion \ref{conclu} and Conclusion \ref{biglam}
bring the question what is the value of the first cardinal
$\lambda=\lambda_{\omega_1}$ such that ${\rm Pr}^\vare(\lambda)$ holds true
every $\vare<\omega_1$. 

\begin{problem}
Is $\lambda_{\omega_1}=\aleph_{\omega_1}$ ?  Does ${\rm
  Pr}^\vare(\aleph_{\omega_1})$ hold true for all $\vare<\omega_1$?
\end{problem}

\bigskip \bigskip \bigskip


\end{document}